\documentclass[11pt,a4paper]{article}
\usepackage[T1]{fontenc}
\usepackage{newtxtext,microtype}
\usepackage{amsmath,amsthm,aliascnt}
\usepackage{newtxmath}
\usepackage[shortlabels]{enumitem}
\usepackage[a4paper,left=2.8cm,right=2.8cm,top=2.6cm,bottom=2.65cm,headheight=23pt,headsep=12pt,footskip=24pt]{geometry}
\usepackage{authblk,cite}
\usepackage{tikz}
\usepackage{placeins,needspace}
\usepackage{titlesec,fancyhdr,caption}
\usetikzlibrary{arrows.meta,positioning}
\definecolor{figInk}{HTML}{243344}
\definecolor{figBlue}{HTML}{305680}
\definecolor{figTeal}{HTML}{02736B}
\definecolor{figOrange}{HTML}{A95523}
\tikzset{
 every picture/.style={font=\small,text=figInk},
 figArrow/.style={draw=figInk,-{Latex[length=2.1mm]},line width=.65pt},
 figPanel/.style={font=\small\bfseries},
 figBox/.style={draw=figBlue!65,fill=figBlue!5,rounded corners=1.5pt,
   minimum height=.63cm,inner xsep=6pt,align=center}
}

\definecolor{referenceBlue}{HTML}{0645AD}
\usepackage[colorlinks=true,allcolors=referenceBlue]{hyperref}

\AddToHook{begindocument/end}{%
  \NewCommandCopy{\OriginalCite}{\cite}%
  \RenewDocumentCommand{\cite}{o m}{%
    \textcolor{referenceBlue}{%
      \IfNoValueTF{#1}{\OriginalCite{#2}}{\OriginalCite[#1]{#2}}}}%
  \NewCommandCopy{\OriginalEqref}{\eqref}%
  \RenewDocumentCommand{\eqref}{m}{%
    \textcolor{referenceBlue}{\OriginalEqref{#1}}}%
}
\makeatletter
\NewCommandCopy{\OriginalBibLabel}{\@biblabel}
\renewcommand{\@biblabel}[1]{\textcolor{referenceBlue}{\OriginalBibLabel{#1}}}
\makeatother

\definecolor{articleInk}{HTML}{243A53}
\definecolor{articleQuiet}{HTML}{536270}
\definecolor{articleRule}{HTML}{BCC7D2}
\setlist[enumerate]{leftmargin=*,itemsep=3pt,topsep=4pt,parsep=0pt}
\titleformat{\section}
  {\normalfont\Large\bfseries}{\textcolor{articleInk}{\thesection}}{.65em}{}
\titleformat{\subsection}
  {\normalfont\large\bfseries}{\textcolor{articleInk}{\thesubsection}}{.55em}{}
\titleformat{\subsubsection}
  {\normalfont\normalsize\bfseries}{\textcolor{articleInk}{\thesubsubsection}}{.55em}{}
\titlespacing*{\section}{0pt}{21pt plus 3pt minus 2pt}{7pt plus 2pt}
\titlespacing*{\subsection}{0pt}{15pt plus 2pt minus 2pt}{5pt plus 1pt}
\titlespacing*{\subsubsection}{0pt}{11pt plus 2pt minus 1pt}{4pt}
\titlespacing*{\paragraph}{0pt}{8pt plus 2pt minus 1pt}{.65em}

\renewcommand{\headrulewidth}{.35pt}

\renewcommand{\headrule}{%
  \hbox to\headwidth{\color{articleRule}\leaders\hrule height\headrulewidth\hfill}}
\fancypagestyle{plain}{%
  \fancyhf{}%
  \fancyfoot[C]{\small\color{articleQuiet}\thepage}%
  \renewcommand{\headrulewidth}{0pt}%
}

\makeatletter
\renewcommand{\@maketitle}{%
  \newpage
  \begin{flushleft}
  \noindent{\color{articleInk}\rule{\linewidth}{.65pt}}\par
  \vspace{.8\baselineskip}
  {\color{articleInk}\bfseries\fontsize{22}{25}\selectfont\@title\par}
  \vspace{.8\baselineskip}
  {\@author\par}
  \end{flushleft}
  \vspace{.35\baselineskip}
}
\makeatother
\renewenvironment{abstract}{%
  \par\small
  \begin{list}{}{%
    \setlength{\leftmargin}{1em}%
    \setlength{\rightmargin}{1em}%
    \setlength{\listparindent}{1em}%
    \setlength{\parsep}{3pt}%
    \setlength{\topsep}{3pt}%
  }%
  \item\relax
  \noindent{\color{articleInk}\bfseries\abstractname}\par\noindent\ignorespaces
}{%
  \end{list}%
  \vspace{2pt}%
}

\newtheoremstyle{articleResult}
  {9pt plus 2pt minus 2pt}{9pt plus 2pt minus 2pt}
  {\normalfont}{0pt}{\bfseries\color{articleInk}}{.}{.65em}{}
\theoremstyle{articleResult}

\newtheorem{theorem}{Theorem}[section]
\newaliascnt{lemma}{theorem}
\newtheorem{lemma}[lemma]{Lemma}
\aliascntresetthe{lemma}
\newaliascnt{proposition}{theorem}
\newtheorem{proposition}[proposition]{Proposition}
\aliascntresetthe{proposition}
\newaliascnt{corollary}{theorem}
\newtheorem{corollary}[corollary]{Corollary}
\aliascntresetthe{corollary}

\theoremstyle{remark}

\numberwithin{equation}{section}
\allowdisplaybreaks[2]

\newcommand{\R}{\mathbb{R}}
\newcommand{\T}{\mathbb{T}}
\newcommand{\norm}[1]{\left\lVert #1\right\rVert}

\newcommand{\eone}{\mathbf e_1}
\newcommand{\Diss}{\mathcal D}
\title{Nonlinear instability and relaxation to balanced states in a rotating Boussinesq strip}
\author[a]{Song Jiang\thanks{jiang@iapcm.ac.cn}}
\author[b]{Quan Wang\thanks{Corresponding author. Email: xihujunzi@scu.edu.cn}}
\affil[a]{LCP, Institute of Applied Physics and Computational Mathematics, Huayuan Road 6, Beijing 100088, China}
\affil[b]{College of Mathematics, Sichuan University, Chengdu 610065, China}

\date{}

\hypersetup{
  pdftitle={Nonlinear instability and relaxation to balanced states in a rotating Boussinesq strip},
  pdfauthor={Song Jiang and Quan Wang},
  pdfkeywords={rotating Boussinesq system, Rayleigh--Taylor instability, partial dissipation}
}
\begin{document}
\maketitle
\begin{abstract}
We study a rotating Boussinesq system in a periodic strip with a stress-free upper wall and a Navier-slip lower wall. The planar velocity is viscous, while two scalar labels are transported by the flow. A quadratic Stokes pencil gives a necessary and sufficient pointwise criterion for growing modes and an exact count with multiplicity. For vertical stratification, we prove a joint Weyl law for rates near zero and their horizontal wave numbers. The leading coefficients are independent of bottom friction, while fixed-mode thresholds can distinguish equimeasurable stratifications. In the nonlinear problem, instability is realized by exact rearrangements of the reference labels that leave a uniform energy gap at every subsequent weak balanced limit. For monotone vertical potentials, we classify Lyapunov stability on the reference joint-distribution class, including zero effective Coriolis coupling. We characterize the minimum relaxed balanced defect by quadratic optimal transport and prove selection of the reference state under geometric control of the label maps. Conditional algebraic decay estimates are complemented by a uniform slow-relaxation lower bound and a stress-free wall invariant that constrains scalar convergence.
\end{abstract}
\noindent\begin{minipage}{\linewidth}
\small
\textcolor{articleInk}{\textbf{Keywords.}} Rotating Boussinesq system, Rayleigh--Taylor instability, partial dissipation, geostrophic balance, asymptotic projection.\par
\smallskip
\textcolor{articleInk}{\textbf{MSC 2020.}} 35B35, 35B40, 35Q35, 76E20.
\end{minipage}\par

\clearpage
\thispagestyle{plain}
\pdfbookmark[1]{Contents}{contents}
\begingroup
\hypersetup{linktoc=all}
\setcounter{tocdepth}{2}
\tableofcontents
\endgroup
\clearpage

\section{Introduction}\label{sec:intro}

\subsection{Selective dissipation and the problem of balance}

On an $f$-plane, rotation couples the horizontal velocity components and a conservative body force couples the flow to density \cite{Pedlosky1987,Majda-2003,Smyth-2019}. The stability and bifurcation of zonally symmetric circulations in this setting have been studied by Kieu and Wang \cite{KieuWang2019}. We consider perturbations of equilibria with an affine meridional velocity $v_s=a_0x+a_1$ in a periodic zonal--vertical strip. The planar velocity $(u,w)$ is viscous, with a stress-free upper wall and a lower wall with positive Navier friction. Density and meridional velocity have no diffusion. The force potential $\Psi(x,z)$ may depend on both coordinates.

Two consequences of this dissipation mechanism motivate the paper. First, coercivity of the planar Stokes form gives dissipation of $(u,w)$, while the coupling to the undiffused variables can generate growing modes or damping rates arbitrarily close to zero. The vanishing planar velocity also leaves a family of equilibria in geostrophic and hydrostatic balance. Even for vertical gravity, horizontal density variation can be balanced by vertical shear in the meridional velocity. Second, the total density and the shifted total meridional velocity are transported by the same incompressible flow. Their \emph{joint} law is conserved on each reflection-invariant half cell. Planar velocity decay thus leaves a selection problem for a coupled scalar pair. The weak limit may lose part of the joint variance, and conservation of the joint law alone still permits stationary folded configurations.

\subsection{Relation to earlier work}

The global theory for Boussinesq flow with transported buoyancy includes \cite{Chae-Nam1997,Hou-li2005,Chae2006,Hmidi-2010,LiTi2016,LaiPanZhao2011,AJ2025}. Ayd{\i}n, Kukavica and Ziane \cite{AydinKukavicaZiane2025} constructed solutions that justify long-time a priori estimates for the non-diffusive system. For planar Couette flow without thermal diffusion, Masmoudi, Said-Houari and Zhao \cite{MasmoudiSaidHouariZhao2022} proved Gevrey stability. Hou, Wang and Xing \cite{HouWangXing2026} studied affine meridional backgrounds in the rotating problem on the unbounded strip, including nonlinear Hadamard instability and decay for constant stratification coefficients. In the bounded strip, the discrete horizontal modes provide a way to resolve the distribution of growth and damping rates and to relate mode thresholds to the vertical profile.

Asymptotic balance and decay for non-diffusive Boussinesq equations under stress-free, no-slip or Navier-slip boundary conditions have been studied in \cite{Doering2018,BCN,Tao2020,Lai-2021,JK2023,Park2024}. In Stokes--transport, Dalibard, Guillod and Leblond \cite{DGL2025} identified the decreasing rearrangement as a limiting density and obtained decay and boundary-layer asymptotics. Their clamped bilaplacian gives the order of individual slow rates \cite[Lemma~B.2]{DGL2025}. Other settings exhibit slowly damped viscous modes or infinitely many growing modes \cite{GurskiPego2002,BarkerNguyenSettle2026}. Our setting combines inertia, two force contributions and different wall conditions. Its scalar equilibrium problem involves a joint distribution, which leads to a two-label transport problem with a balance constraint.

\subsection{Main results and methods}

Eliminating the two scalars from a nonzero normal mode yields a quadratic pencil on the solenoidal velocity space. The driving form is indefinite in general, and testing a pointwise sign condition requires trial fields that satisfy both incompressibility and the reflection symmetry. Localized oscillatory fields provide such tests. Monotone crossings of the pencil give a necessary and sufficient pointwise criterion for growing modes and an exact count with algebraic multiplicity. In the strictly stable regime, bounded analytic evolution converges strongly to the balanced projection, with sharp uniform decay rates for the planar $H^1$ norm and the time derivative of the state. The accumulation of damping rates at zero prevents uniform convergence to the projection in the state norm.

For vertical stratification the count near zero requires control of large vertical frequencies while the horizontal cutoff grows like $s^{-1/2}$ at rate threshold $s$. We bracket the fourth-order forms by Dirichlet sine problems and bound the total index error from changing the wall conditions and from partitioning the vertical profile. This gives a joint Weyl law uniform in the horizontal cutoff and a leading coefficient independent of bottom friction. On either spectral branch with positive Weyl coefficient, the scaled rate $r/s$ and the scaled horizontal wave number $k\sqrt r$ converge jointly to independent variables. The positive Weyl coefficient, viewed as a function of the Coriolis coupling, determines the value distribution of the positive stratification. The fixed-mode thresholds give additional information on the vertical arrangement. They distinguish certain equimeasurable profiles and recover the midplane-symmetric part of a profile perturbation to first order. See \autoref{sec:linear-results} and \autoref{sec:spectral-results}.

Variational approaches to viscous buoyancy instability appear in \cite{JJ2014,JJ2013,Jiang2016a}, and related rotating instability results in \cite{Mao-2024,Xing-2024,FanHanWangXing2025}. Here the initial direction of instability must be realized by a nonlinear perturbation that preserves the \emph{exact} joint law. We use the flow of the planar component of a fastest growing mode to rearrange both labels together. On this class the negative second variation of mechanical energy is the same quadratic form that drives linear growth. The constructed solution follows the linear mode up to logarithmic escape time and dissipates an amount of energy bounded below independently of the initial amplitude. Every subsequent weak balanced limit is therefore separated from the reference equilibrium by a uniform energy gap. For a monotone vertical potential, the complementary sign regime is Lyapunov stable on the reference joint-distribution class, including zero effective Coriolis coupling. See \autoref{sec:stability-results}.

The $L^q$ theory for $q>2$ constructs the characteristic flow and gives decay of the force residual. Weak scalar convergence can nevertheless lose joint variance, so marginal rearrangements do not determine whether a balanced limit retains the two initial labels. We identify the weak closure of the joint rearrangement class by convex order and express the lost variance as a second-moment defect. The ensuing minimization imposes pressure balance in addition to a constraint on the label law. Its infimum equals a quadratic optimal transport cost and is attained by a balanced field. Convex transport duality and energy minimization under rearrangements provide the background for this construction \cite{Brenier1991,CullenKunaPelloniWilkinson2019}. A difficulty is that the target labels can have a singular distribution when the potential collapses a vertical interval. The conditional barycenter of an optimal coupling still satisfies the pressure balance, by a subgradient identity in physical coordinates. With absolutely continuous target laws, equality between the trajectory's limiting defect and this minimum is equivalent to strong convergence to the unique optimizer. See \autoref{sec:global-results} and \autoref{sec:transport-results}.

The joint law allows stationary folded profiles arbitrarily close to a nonconstant reference state. To obtain selection, we use the side-preserving geometry of the area-preserving maps and the balance identity to prove rigidity for diffeomorphic rearrangements. Compactness of the flow maps then identifies the reference state as the limit. Quantitative bounds on the evolving maps yield algebraic decay through a projected-force estimate and a compensated energy. Smooth exact rearrangements with slow relaxation show that a uniform $L^2$ decay estimate from small $H^m$ data cannot improve on the power $m/2$ over the stated initial-data class. At the stress-free wall, the tangential and normal derivatives of the shifted meridional label transform by reciprocal stretch factors. This yields a conserved wall integral and a sharp trace estimate coupling scalar decay to Sobolev norm growth. On the reference diffeomorphic rearrangement class, a positive wall integral also precludes finite total strain. These results identify separate restrictions on the decay rate and on the geometry needed for selection. See \autoref{sec:rigidity-results} and \autoref{sec:rates-results}.

The model and energy identities precede the spectral and nonlinear proofs in \autoref{sec:linear-proof} and \autoref{sec:nonlinear-proof}. The Stokes estimates are in \autoref{sec:stokes}.

\section{The model and its steady states}\label{sec:model}

Let $(x,y,z)$ denote the zonal, meridional and vertical coordinates. For the total velocity $\mathbf U=(u,v,w)$, write $\nabla_3=(\partial_x,\partial_y,\partial_z)$, $\Delta_3=\nabla_3\cdot\nabla_3$ and $\mathbf e_z=(0,0,1)^T$. The three-dimensional starting system is
\begin{equation}\label{eq_rotating_boussinesq_3d}
\left\{
\begin{aligned}
 &\rho_t+\mathbf U\cdot\nabla_3\rho=0,\qquad \nabla_3\cdot\mathbf U=0,\\
 &\mathbf U_t+\mathbf U\cdot\nabla_3\mathbf U+f\mathbf e_z\times\mathbf U
   =\operatorname{diag}(\nu_1,\nu_2,\nu_3)\Delta_3\mathbf U
       -\nabla_3 p-\rho\nabla_3\Psi .
\end{aligned}
\right.
\end{equation}
We assume all fields, including pressure and potential, are independent of $y$. Let $\T=\R/(2\pi\mathbb Z)$ and $\Omega=\T\times(0,h)$ with $h>0$. For $\mathbf u=(u,w)$ we write $\nabla=(\partial_x,\partial_z)$, $\Delta=\partial_x^2+\partial_z^2$ and $\eone=(1,0)^T$. Then $\mathbf U\cdot\nabla_3=\mathbf u\cdot\nabla$ on such fields, which yields the total-field equations
\begin{equation}\label{eqaa02}
\left\{
\begin{aligned}
&\partial_t\rho+\mathbf u\cdot\nabla\rho=0,\qquad \nabla\cdot\mathbf u=0,\\
&\partial_tu+\mathbf u\cdot\nabla u=\nu_1\Delta u+fv-\partial_xp-\rho\partial_x\Psi,\\
&\partial_tv+\mathbf u\cdot\nabla v=\nu_2\Delta v-fu,\\
&\partial_tw+\mathbf u\cdot\nabla w=\nu_3\Delta w-\partial_zp-\rho\partial_z\Psi .
\end{aligned}
\right.
\end{equation}
Here $f>0$ is the Coriolis parameter and $\Psi$ is time independent. In the planar equations the top wall is stress free and the bottom wall has positive Navier friction. More precisely,
\begin{equation}\label{bd-con-2}
 w|_{z=0,h}=0,\qquad
 \partial_z u|_{z=h}=0,\qquad
 (\partial_z u-\alpha u)|_{z=0}=0,\qquad \alpha>0 .
\end{equation}
If $\nu_2>0$, we impose the same tangential conditions on $v$. For $\nu_2=0$, the transport field in the equation for $v$ is tangent to the walls, so no boundary data are imposed on $v$.

\begin{proposition}\label{prop:steady}
Assume $\nu_1,\nu_3>0$, $\nu_2\ge0$, and let the total fields and $\Psi$ be smooth and periodic in $x$. If $\nu_2=0$, the steady solutions are exactly the fields $(\rho_s,0,v_s,0,p_s)$ satisfying
\begin{equation}\label{hy-geo-ba}
 fv_s-\partial_xp_s=\rho_s\partial_x\Psi,\qquad
 -\partial_zp_s=\rho_s\partial_z\Psi .
\end{equation}
For a given $\rho_s$, every such solution is determined by the periodic bottom pressure trace $p_s(\cdot,0)$ through
\begin{equation}\label{eq_partial_steady_parameterization}
\begin{aligned}
p_s(x,z)&=p_s(x,0)-\int_0^z\rho_s(x,\zeta)\,\partial_\zeta\Psi(x,\zeta)\,d\zeta,\\
v_s(x,z)&=f^{-1}\bigl(\partial_xp_s(x,z)+\rho_s(x,z)\partial_x\Psi(x,z)\bigr).
\end{aligned}
\end{equation}
If $\nu_2>0$, the steady solutions are exactly the fields with $u=v=w=0$ and $-\nabla p_s=\rho_s\nabla\Psi$.
\end{proposition}

\begin{proof}
The steady transport equation and impermeability give $\int_\Omega\rho\,\mathbf u\cdot\nabla\Psi=-\int_\Omega\Psi\,\mathbf u\cdot\nabla\rho=0$. We test the three steady momentum equations with $u,v,w$, respectively, and add the resulting identities. The transport and pressure integrals vanish and the Coriolis terms cancel. The resulting identity is
\[
 \nu_1\norm{\nabla u}_2^2+\nu_2\norm{\nabla v}_2^2+\nu_3\norm{\nabla w}_2^2+
 \alpha\int_{\T}(\nu_1u^2+\nu_2v^2)(x,0)\,dx=0 .
\]
The planar components are constant, and the wall conditions imply $u=w=0$. When $\nu_2>0$, the same argument gives $v=0$. The remaining equations yield the stated balances. Substitution establishes the converse. Integrating the vertical balance and then solving the horizontal balance for $v_s$ proves \eqref{eq_partial_steady_parameterization}.
\end{proof}

For example, if $\Psi=gz$ with $g>0$, the fields
$\rho_s=2+\cos x$, $p_s=-gz(2+\cos x)$ and $v_s=(gz/f)\sin x$ satisfy
\eqref{hy-geo-ba}. In particular, the density of a balanced rotating state can vary horizontally even under vertical gravity.

\subsection{Affine meridional backgrounds}\label{sec:background}

From now on let $\nu_1=\nu_3=\nu>0$ and $\nu_2=0$. Take $\rho_s$ and $\Psi$ smooth, even and periodic in $x$, and assume $\nabla\rho_s=\delta\nabla\Psi$ with $\delta\in L^\infty(\Omega)$. The backgrounds considered below are given by
\begin{equation}\label{eq_background}
\begin{aligned}
 \nabla\rho_s&=\delta\nabla\Psi, & v_s(x)&=a_0x+a_1,\\
 \nabla p_s&=-\rho_s\nabla\Psi, & p_0(x)&=\tfrac12 fa_0x^2+fa_1x .
\end{aligned}
\end{equation}
Here $a_0,a_1\in\R$, and we write $\beta_f=f+a_0$ for the effective Coriolis coupling. The vector field $-\rho_s\nabla\Psi$ is curl free. Its horizontal period integral vanishes because $\rho_s\partial_x\Psi$ is odd. It follows that the vector field has an even, periodic potential $p_s$. Since $\nabla p_0=fv_s\eone$, the fields $(\rho_s,0,v_s,0,p_s+p_0)$ satisfy \eqref{eqaa02} on the representative strip $(-\pi,\pi)\times(0,h)$. The affine velocity and the pressure correction are defined on this strip. Periodicity applies to the perturbations and to the coefficients of their equations. The affine label $v_s+fx=\beta_f x+a_1$ is used on each half cell.

Write the total fields as $\rho_s+\rho$, $v_s+v$, $\mathbf u$ and $p_s+p_0+p$, where $(\rho,v,\mathbf u,p)$ now denote perturbations. Subtracting the background equations gives
\begin{equation}\label{eq_perturbation}
\left\{
\begin{aligned}
&\rho_t+\mathbf u\cdot\nabla\rho=-\delta\,\mathbf u\cdot\nabla\Psi,\\
&v_t+\mathbf u\cdot\nabla v=-(f+a_0)u,\\
&\mathbf u_t+\mathbf u\cdot\nabla\mathbf u
  =\nu\Delta\mathbf u+fv\eone-\nabla p-\rho\nabla\Psi,\\
&\nabla\cdot\mathbf u=0.
\end{aligned}
\right.
\end{equation}
The data are denoted by $(\rho_0,v_0,\mathbf u_0)$. We normalize the perturbation pressure by $\int_\Omega p=0$. We impose \eqref{bd-con-2} and work in the invariant reflection class
\begin{equation}\label{eq_perturbation_parity}
 \rho,w,p\text{ are even in }x,\qquad u,v\text{ are odd in }x.
\end{equation}
To verify invariance, substitute the reflected fields into \eqref{eq_perturbation} and use the evenness of $\rho_s$ and $\Psi$. In this class $u=0$ at $x=0,\pm\pi$, so neither half of the representative cell exchanges fluid with the other. The linear equations are
\begin{equation}\label{eq_linear_perturbation}
\left\{
\begin{aligned}
&\rho_t=-\delta\,\mathbf u\cdot\nabla\Psi,\qquad v_t=-(f+a_0)u,\\
&\mathbf u_t=\nu\Delta\mathbf u+fv\eone-\nabla p-\rho\nabla\Psi,\qquad
 \nabla\cdot\mathbf u=0 .
\end{aligned}
\right.
\end{equation}

\subsection{Notation}

Unless stated otherwise, spatial norms are taken over $\Omega$, and spaces on $\Omega$ are periodic in $x$. We write $\|\cdot\|_2$ for the $L^2$ norm. In the force-residual estimate, $H^{-1}$ denotes the dual of the periodic $H^1$ space with zero trace at the horizontal walls. The continuous-dependence estimate uses the dual of periodic $H^1$ with arbitrary wall traces, denoted by $(H^1_{\rm per})'$ and equipped with the norm $\|\cdot\|_{-1,N}$. We set
\begin{equation}\label{eq:diss}
 \Diss(\mathbf u)=\norm{\nabla\mathbf u}_2^2+
    \alpha\norm{u(\cdot,0)}_{L^2(\T)}^2,\qquad
 \mathbf F= fv\eone-\rho\nabla\Psi .
\end{equation}
We denote the two half cells by $\Omega_-=(-\pi,0)\times(0,h)$ and $\Omega_+=(0,\pi)\times(0,h)$. The geometry and wall conditions are shown in \autoref{fig:strip}.

\begin{figure}[!htbp]
\centering
\begin{tikzpicture}[x=1.10cm,y=.94cm]
  \fill[figBlue!8] (0,0) rectangle (4,2.15);
  \fill[figTeal!9] (4,0) rectangle (8,2.15);
  \draw[figInk,line width=.8pt] (0,0) rectangle (8,2.15);
  \draw[figInk,dashed] (4,0)--(4,2.15);
  \foreach \x in {0,.24,...,7.92}{\draw[figInk!55] (\x,0)--++(-.12,-.12);}
  \node[above=3pt] at (4,2.15) {stress-free top: $w=0$, $u_z=0$};
  \node[below=10pt] at (4,0) {Navier bottom: $w=0$, $u_z=\alpha u$};
  \node[left] at (0,2.15) {$h$};
  \node[left] at (0,0) {$0$};
  \node[below left=1pt and 0pt] at (0,0) {$-\pi$};
  \node[below=1pt] at (4,0) {$0$};
  \node[below right=1pt and 0pt] at (8,0) {$\pi$};
  \node[figBlue] at (2,1.52) {$\Omega_-$};
  \node[figTeal] at (6,1.52) {$\Omega_+$};
  \path[figArrow,draw=figBlue] (3.02,.76) .. controls (2.52,.22) and (1.47,.23) .. (1.14,.88)
                        .. controls (.9,1.37) and (1.32,1.43) .. (1.64,1.23);
  \path[figArrow,draw=figTeal] (4.98,.76) .. controls (5.48,.22) and (6.53,.23) .. (6.86,.88)
                        .. controls (7.1,1.37) and (6.68,1.43) .. (6.36,1.23);
  \node[above,fill=white,inner sep=2pt] at (4,.57) {$u=0$};
  \draw[figArrow] (8.25,0)--(8.9,0) node[right] {$x$};
  \draw[figArrow] (8.25,0)--(8.25,.64) node[above] {$z$};
\end{tikzpicture}
\caption{Representative strip and reflection-invariant half cells. The vertical sides at $x=-\pi$ and $x=\pi$ are identified for periodic perturbations. The odd parity of $u$ prevents transport across the seams, so both scalar labels are transported on each half cell separately. Curves indicate schematic planar trajectories.}
\label{fig:strip}
\end{figure}

\FloatBarrier

\Needspace{11\baselineskip}
\section{Main results}\label{sec:results}

\subsection{Linear instability and relaxation}\label{sec:linear-results}

We first describe the growing spectrum. Under strict stability, we then establish strong convergence to a projected balanced state. The proofs are in \autoref{sec:linear-proof}.

\begin{theorem}\label{thm_linear_stability}
Consider the linear system \eqref{eq_linear_perturbation} about the background \eqref{eq_background}, wall conditions \eqref{bd-con-2} and parity \eqref{eq_perturbation_parity}. For a real, divergence-free velocity $\mathbf u\in H^1(\Omega)^2$ with zero normal trace at the walls and the prescribed parity, set
\begin{equation}\label{eq_linear_variational_drive}
 E_2(\mathbf u)=\int_\Omega
 \left[\delta(\mathbf u\cdot\nabla\Psi)^2-f(f+a_0)u^2\right]\,dx\,dz .
\end{equation}
A growing normal mode exists if and only if $E_2(\mathbf u)>0$ for some such velocity. All growing normal modes have real growth rate. When they exist, the largest growth rate is the unique positive root of
\begin{equation}\label{eq:max-growth}
 s^2+\inf_{\norm{\mathbf u}_2=1}
          \bigl\{\nu s\Diss(\mathbf u)-E_2(\mathbf u)\bigr\}=0,
\end{equation}
where the infimum is over the same velocity space.

The spectrum of the projected generator in the open right half-plane consists of semisimple positive eigenvalues. Let $N(s)$ count the eigenvalues strictly larger than $s>0$, with algebraic multiplicity. The infimum and subspace dimension in this theorem may be taken in the real velocity space, since the forms have real coefficients. With $\mathcal V$ denoting the above trial space,
\begin{equation}\label{eq:mode-count}
 N(s)=\max\bigl\{\dim L:\ L\text{ a subspace of }\mathcal V,\ 
 E_2(\mathbf u)>s^2\|\mathbf u\|_2^2+\nu s\Diss(\mathbf u)
 \text{ for all }0\ne\mathbf u\in L\bigr\}.
\end{equation}
A growing mode exists exactly when the matrix
\begin{equation}\label{eq:pointwise-drive}
 \mathsf M=\nabla\rho_s\otimes\nabla\Psi-f(f+a_0)\eone\otimes\eone
\end{equation}
has a positive direction at an interior point. Whenever this holds, infinitely many positive growth rates tend to zero. In particular, this occurs if $\Psi$ is harmonic and nonconstant and $\delta$ is continuous and positive at an interior point, for every $f+a_0\in\R$. It also occurs whenever $f+a_0<0$. If $\delta\le\delta_0<0$ for a constant $\delta_0$ and $f+a_0>0$, the linear evolution is bounded uniformly in time in the product $L^2$ norm.
\end{theorem}

\begin{theorem}\label{thm_linear_large_time}
Assume the background \eqref{eq_background}, with $\rho_s$ and $\Psi$
even and periodic in $x$. Suppose that $\nabla\Psi\in W^{1,\infty}(\Omega)$, $\delta\in L^\infty(\Omega)$,
$\delta(x,z)\leq\delta_0<0$ for some constant $\delta_0$, and $f+a_0>0$.
Let $\mathscr X$ be the product of the solenoidal, impermeable $L^2$ velocity space and the two scalar $L^2$ spaces, restricted to \eqref{eq_perturbation_parity}. Use the product norm obtained by adding $\|\mathbf u\|_2^2$ to the squared scalar norm
\begin{equation}\label{eq_linear_balance_projection_weight}
 \norm{(\rho,v)}_{\mathrm{sc}}^2
 =\int_\Omega\left(\frac{|\rho|^2}{-\delta}+\frac{f}{f+a_0}|v|^2\right)\,dx\,dz .
\end{equation}
The linear evolution $S(t)$ of \eqref{eq_linear_perturbation} is a bounded analytic semigroup on $\mathscr X$. For all initial data in $\mathscr X$ it converges strongly to the orthogonal projection $\Pi_{\mathscr E}$ onto the equilibrium space $\mathscr E$ in this norm. The equilibrium velocity is zero and the scalar pair satisfies
\begin{equation}\label{eq_linear_balance_subspace}
 fv\eone-\nabla p-\rho\nabla\Psi=0
 \quad\text{for some periodic mean-zero } p\in H^1(\Omega).
\end{equation}
In particular, the limiting scalar pair depends only on $(\rho_0,v_0)$. On $\mathscr E^\perp$ the generator has negative eigenvalues $-a_j$ with $a_j\to0$. More precisely, if $A$ is the mixed-slip Stokes operator, $B(\rho,v)=\mathbb P(fv\eone-\rho\nabla\Psi)$ with the weighted adjoint $B^*$, and $\kappa_j$ are the positive eigenvalues of $B^*(\nu A)^{-1}B$ in decreasing order, there is a choice of $a_j$ with $a_j/\kappa_j\to1$. In particular,
\begin{equation}\label{eq:nonuniform}
 \norm{S(t)-\Pi_{\mathscr E}}_{\mathscr X\to\mathscr X}=1
 \qquad(t\ge0).
\end{equation}
For $X_0=(\mathbf u_0,\rho_0,v_0)\in\mathscr X$ and $t>0$,
\begin{equation}\label{eq:linear-uniform-rates}
 \|\mathbf u(t)\|_{H^1}\le Ct^{-1/2}\|X_0\|_{\mathscr X},
 \qquad
 \|\partial_t S(t)X_0\|_{\mathscr X}\le Ct^{-1}\|X_0\|_{\mathscr X}.
\end{equation}
Both exponents are optimal for estimates uniform on the unit ball of $\mathscr E^\perp$ as $t\to\infty$.

If in addition $\mathbf u_0\in H^2(\Omega)^2$ satisfies \eqref{bd-con-2}, the solution with mean-zero pressure satisfies
\begin{subequations}
\label{eq_linear_global_regularity}
\begin{align}\label{theorem-4-conc-1}
&\mathbf{u} \in L^{\infty}\left(\left(0,\infty\right);H^2(\Omega)\right)
\cap L^{p}\left(\left(0,\infty\right);H^{1}(\Omega)\right),\quad 2\leq p<\infty,\\
\label{theorem-4-conc-2}
&\mathbf{u}_t \in
L^{\infty}\left(\left(0,\infty\right);L^2(\Omega)\right)
\cap L^{2}\left(\left(0,\infty\right);H^{1}(\Omega)\right),\\
\label{theorem-4-conc-3}
&p\in L^{\infty}\left(\left(0,\infty\right);H^1(\Omega)\right),\\
\label{theorem-4-conc-4}
&(\rho,v) \in L^{\infty}\left(\left(0,\infty\right);L^2(\Omega)\right).
 \end{align}
   \end{subequations}
For every $\tau>0$,
\begin{equation}\label{eq_linear_positive_time_regularity}
\mathbf{u}_t\in L^\infty((\tau,\infty);H^1(\Omega))
\cap L^2((\tau,\infty);H^2(\Omega)).
\end{equation}
The estimates in \eqref{eq_linear_positive_time_regularity} hold with $\tau=0$ if the projected initial acceleration $\mathbf u_t(0)=-\nu A\mathbf u_0+\mathbb P\mathbf F(0)$ belongs to the form domain $\mathcal V$. This requires $H^1$ regularity, zero normal trace, incompressibility and the velocity parity in \eqref{eq_perturbation_parity}.

As $t\to\infty$,
\begin{subequations}
\label{eq_linear_decay}
     \begin{align}\label{theorem-3-conc-1}
&\norm{\mathbf{u} }_{H^{2}}\to 0,\quad t\to \infty,\\
\label{theorem-3-conc-2}
&\norm{\mathbf{u}_t }_{H^{1}}\to 0,\quad t\to \infty,\\
\label{theorem-3-conc-3}
&\norm{fv{\bf e}_{1} -\nabla p-\rho\nabla \Psi}_{L^{2}}\to 0 ,\quad t\to \infty,\\
\label{theorem-3-conc-4}
&\norm{v_t }_{H^{2}}\to 0,\quad t\to \infty.
 \end{align}
 \end{subequations}
The normalized pressure converges in $H^1(\Omega)$ to the pressure associated with the limiting scalar pair.
\end{theorem}

\Needspace{9\baselineskip}
\subsection{Spectral counts and vertical stratification}\label{sec:spectral-results}

For vertically stratified backgrounds, the horizontal cutoff resolves the joint distribution of rates and wave numbers. Fixed-wave-number thresholds then give information on how the stratification is arranged vertically. See \autoref{sec:spectral-proof}.

\begin{theorem}\label{thm:vertical-weyl}
Suppose the potential and the background density depend only on $z$. Set
\[
 c(z)=\rho_s'(z)\Psi'(z),\qquad b=f(f+a_0).
\]
For $0\le K\le\infty$, define
\[
 \mathcal A_K(p,r)
 =\left|\left\{(\xi,\eta)\in(0,\infty)^2:
 \xi\le K,\ \nu(\xi^2+\eta^2)^2<p\xi^2+r\eta^2\right\}\right|.
\]
Here $|\cdot|$ denotes planar Lebesgue measure. Given a threshold $s>0$, let $N_+(s,K)$ count the eigenvalues $\lambda>s$, with algebraic multiplicity, in the horizontal modes $1\le k\le K/\sqrt{s}$ of the prescribed reflection class. Then, as $s\downarrow0$,
\begin{equation}\label{eq:weyl-positive}
 sN_+(s,K)=\frac1\pi\int_0^h \mathcal A_K(c(z),-b)\,dz+o(1).
\end{equation}

Under the strict stability hypotheses of \autoref{thm_linear_large_time}, let $d_0>0$ satisfy $\nu A\ge d_0I$. In $-d_0/2<\Re\lambda<0$, the spectrum consists of semisimple real negative eigenvalues. Choose $a_{\rm c}\in(0,d_0/2)$ with $-a_{\rm c}$ in the resolvent set. Let $N_-(s,K)$ count the eigenvalues $-a$ with $s<a<a_{\rm c}$ in the same range of horizontal modes. Then
\begin{equation}\label{eq:weyl-negative}
 sN_-(s,K)=\frac1\pi\int_0^h \mathcal A_K(-c(z),b)\,dz+o(1).
\end{equation}
The convergence in both limits is uniform in $K\in[0,\infty]$. For Lipschitz $c$, both remainders are $O(s^{1/4})$ in the displayed formulas, and for constant $c$ they are $O(s^{1/2})$.

In particular, writing $r_+=\max(r,0)$, the total growing-mode coefficient is
\begin{equation}\label{eq:weyl-total-positive}
 C_+=\frac{1}{2\pi\nu}\int_0^h\int_0^{\pi/2}
       \bigl(c(z)\cos^2\vartheta-b\sin^2\vartheta\bigr)_+\,d\vartheta\,dz .
\end{equation}
In the stable regime, the total slowly damped coefficient is
\begin{equation}\label{eq:weyl-total-negative}
 C_-=\frac{1}{8\nu}\int_0^h\bigl(-c(z)+b\bigr)\,dz .
\end{equation}
The coefficients are independent of the bottom friction parameter $\alpha$. If $C_+>0$, the positive eigenvalues in decreasing order obey $\lambda_j\sim C_+/j$. In the stable regime the corresponding sequence of damping rates, ordered decreasingly with multiplicity, obeys $a_j\sim C_-/j$.
For either count with $C_\sigma>0$, $\sigma\in\{+,-\}$, sample uniformly from the modes counted with multiplicity by $N_\sigma(s,\infty)$. If the same eigenvalue occurs at different horizontal wave numbers, each mode is sampled with its own wave number. Write $r$ for its positive growth or damping rate and $k$ for its horizontal wave number. As $s\downarrow0$, the pair
\begin{equation}\label{eq:weyl-joint-pair}
 (U,Q)=(r/s,\,k\sqrt{r})
\end{equation}
converges in law to two independent variables. The limiting $U$ has tail $\Pr(U>u)=u^{-1}$ for $u\ge1$. For the limiting $Q$, define
\[
 A_\sigma(z,\vartheta)
 =\begin{cases}
 \nu^{-1}(c(z)\cos^2\vartheta-b\sin^2\vartheta)_+,&\sigma=+,\\
 \nu^{-1}(-c(z)\cos^2\vartheta+b\sin^2\vartheta)_+,&\sigma=-.
 \end{cases}
\]
Its distribution is the pushforward of the probability measure proportional to $A_\sigma(z,\vartheta)\,dz\,d\vartheta$ on $(0,h)\times(0,\pi/2)$ under $Q=\sqrt{A_\sigma(z,\vartheta)}\cos\vartheta$. In the constant isotropic stable case $-c=b>0$, put $R_*=\sqrt{b/\nu}$. For $0\le q\le1$ the two horizontal normalizations give
\begin{equation}\label{eq:weyl-isotropic-cdfs}
 \Pr(Q/R_*\le q)=\tfrac{2}{\pi}\arcsin q,\qquad
 \lim_{s\downarrow0}\Pr(k\sqrt{s}/R_*\le q)
  =\tfrac{2}{\pi}(\arcsin q+q\sqrt{1-q^2}).
\end{equation}

For Lipschitz $c$, the same sampling gives a uniform rectangle error for the reciprocal rate $V=s/r$ and threshold-scaled wave number $X=k\sqrt{s}$. Set $F_\sigma(K)$ equal to the Weyl coefficient with cutoff $K$ in \eqref{eq:weyl-positive} or \eqref{eq:weyl-negative}, divided by $C_\sigma$. Then for $0<v\le1$ and $K\ge0$,
\begin{equation}\label{eq:weyl-rectangle-error}
 \left|\Pr_s(V<v,\,X\le K)-v F_\sigma(K/\sqrt v)\right|
 \le C s^{\zeta},
 \qquad
 \zeta=\begin{cases}1/4,& c\text{ Lipschitz},\\1/2,& c\text{ constant}.
 \end{cases}
\end{equation}
Here $\Pr_s$ denotes the finite-threshold sampling law and the constant is uniform in $v,K$.
\end{theorem}

\autoref{fig:joint-spectral-law}(a) gives the common rate tail for the two families. In panel (b), normalizing the horizontal wave number by the sampled eigenvalue's own rate replaces the quarter-disk coordinate law by the angular law.

\begin{figure}[!htbp]
\centering
\begin{tikzpicture}[x=1cm,y=1cm]
  \node[figPanel] at (2.25,2.98) {(a) Relative rate};
  \node[figPanel] at (8.10,2.98) {(b) Horizontal normalizations};
  \begin{scope}[xscale=1.38,yscale=2.08]
    \draw[figArrow] (0,0)--(3.29,0) node[right] {$u$};
    \draw[figArrow] (0,0)--(0,1.14) node[above] {$\Pr(U>u)$};
    \draw[figInk!27,dashed] (0,1)--(3.04,1);
    \draw[figTeal,line width=1pt,domain=0:3.04,samples=82] plot(\x,{1/(1+\x)});
    \foreach \x/\a in {0/1,1/2,2/3,3/4}{
      \draw (\x,.02)--(\x,-.02) node[below=2pt] {$\a$};
    }
    \node[left] at (0,1) {$1$};
    \node[figTeal,right] at (1.67,.56) {$u^{-1}$};
  \end{scope}
  \begin{scope}[shift={(6.15,0)},xscale=3.80,yscale=2.08]
    \draw[figArrow] (0,0)--(1.16,0) node[right] {$q$};
    \draw[figArrow] (0,0)--(0,1.14) node[above] {CDF};
    \draw[figInk!27,dashed] (0,1)--(1,1)--(1,0);
    \draw[figBlue,line width=.95pt] plot[variable=\t,domain=0:90,samples=81]
      ({sin(\t)},{(\t/90)+(2/pi)*sin(\t)*cos(\t)});
    \draw[figOrange,densely dashed,line width=1pt] plot[variable=\t,domain=0:90,samples=81]
      ({sin(\t)},{\t/90});
    \node[left] at (0,1) {$1$};
    \draw (.5,.02)--(.5,-.02) node[below=2pt] {$1/2$};
    \node[below=2pt] at (1,0) {$1$};
    \node[figBlue,above] at (.52,.68) {$k\sqrt{s}/R_*$};
    \node[figOrange,below right] at (.49,.29) {$Q/R_*$};
  \end{scope}
\end{tikzpicture}
\caption{Limiting distributions in \autoref{thm:vertical-weyl}. (a) For either family of sampled eigenvalues with positive Weyl coefficient, the normalized rate $U=r/s$ has the exact limiting tail $u^{-1}$. (b) In the constant isotropic stable case, the exact limiting cumulative laws of $k\sqrt{s}/R_*$ and $Q/R_*=k\sqrt{r}/R_*$ are the two different curves in \eqref{eq:weyl-isotropic-cdfs}. The independence assertion concerns $U$ and $Q$.}
\label{fig:joint-spectral-law}
\end{figure}

\begin{proposition}\label{prop:inverse-stratification}
In \autoref{thm:vertical-weyl}, keep $c$ and $\nu$ fixed and vary $b>0$. Write $C_+(b)$ for the coefficient in \eqref{eq:weyl-total-positive} and $M_p=\int_0^h c_+(z)^p\,dz$ for $p>0$. For $b>\|c_+\|_\infty$,
\begin{equation}\label{eq:inverse-moment-expansion}
 \sqrt b\, C_+(b)
 =\frac{1}{2\pi\nu}\sum_{n=0}^\infty
 \frac{2(-1)^n(n+1)}{(2n+1)(2n+3)}\,\frac{M_{n+3/2}}{b^n}.
\end{equation}
Consequently, the values of $C_+(b)$ on any sequence of distinct $b\to\infty$ determine the distribution of $c_+$ with respect to $dz$, with $h$ and $\nu$ fixed. Two profiles with the same distribution of $c_+$ have the same function $C_+(b)$ for every $b>0$.

For every $b>0$, one has the finite-coupling estimate
\begin{equation}\label{eq:inverse-single-moment}
 0\le M_{3/2}-3\pi\nu\sqrt b\, C_+(b)\le\frac{2M_{5/2}}{5b}.
\end{equation}
In particular, if $|\widetilde C_+(b)-C_+(b)|\le\delta$, the error of the estimator $3\pi\nu\sqrt b\,\widetilde C_+(b)$ is at most $2M_{5/2}/(5b)+3\pi\nu\sqrt b\,\delta$.
\end{proposition}

In particular, \eqref{eq:inverse-single-moment} gives $\sqrt b\,C_+(b)\to M_{3/2}/(3\pi\nu)$ as $b\to\infty$. If $c>0$ somewhere, \eqref{eq:weyl-total-positive} gives $C_+>0$ for every finite $b\ge0$, so the growing count is infinite by \autoref{thm:vertical-weyl}.

\autoref{fig:inverse-stratification} shows what can be reconstructed from the large-coupling count and gives two different vertical profiles with identical leading responses.

\begin{figure}[!htbp]
\centering
\begin{tikzpicture}[x=1cm,y=1cm]
 \node[figPanel] at (2.38,3.03) {(a) Vertical profiles};
 \node[figPanel] at (8.02,3.03) {(b) Information in the response};
 \begin{scope}[x=3.88cm,y=1.28cm,shift={(.11,.12)}]
   \draw[figArrow] (0,0)--(1.09,0) node[right] {$z/h$};
   \draw[figArrow] (0,0)--(0,1.92) node[left] {$c_+$};
   \draw[figBlue,line width=.9pt,domain=0:1,samples=81] plot(\x,{1.2+.45*cos(360*\x)});
   \draw[figOrange,densely dashed,line width=.95pt,domain=0:1,samples=121] plot(\x,{1.2+.45*cos(720*\x)});
   \node[below=2pt] at (1,0) {$1$};
   \node[figBlue] at (.25,.31) {$c_1$};
   \node[figOrange] at (.60,.31) {$c_2$};
 \end{scope}
 \node[figBox,minimum width=4.56cm] (law) at (8.02,2.23)
     {value distribution $(c_+)_\#dz$};
 \node[figBox,minimum width=4.56cm] (mom) at (8.02,1.23)
     {weighted moments $M_{n+3/2}$};
 \node[figBox,minimum width=4.56cm] (obs) at (8.02,.23)
     {coupling response $C_+(b)$};
 \draw[figInk,<->,>=Latex] (law)--(mom);
 \draw[figInk,<->,>=Latex] (mom)--(obs);
\end{tikzpicture}
\caption{Inverse spectral information in \autoref{prop:inverse-stratification}. The positive profiles $c_1(z)=1.2+0.45\cos(2\pi z/h)$ and $c_2(z)=1.2+0.45\cos(4\pi z/h)$ have the same distribution against $dz$, so they give the same response $C_+(b)$ for all $b>0$. The response recovers the value distribution through the weighted moments.}
\label{fig:inverse-stratification}
\end{figure}

\FloatBarrier

\begin{proposition}\label{prop:fixed-mode-count}
We use the vertically stratified setting and the notation of \autoref{thm:vertical-weyl}. For a fixed integer $k\ge1$, let $N_{+,k}(s)$ denote the growing-mode count at $k$.
\begin{enumerate}[(1)]
\item If $b>0$, there are finitely many growing modes at $k$. Their number equals the number of strictly negative eigenvalues of
\[
 -b\,\partial_z^2-k^2c(z)\quad\text{on }(0,h)
 \quad\text{with Dirichlet conditions}.
\]
\item If $b=0$, then
\[
 \lim_{s\downarrow0}s^{1/4}N_{+,k}(s)
  =\frac1\pi\int_0^h \left(\frac{k^2c(z)_+}{\nu}\right)^{1/4}dz .
\]
If $c\le0$, there are no growing modes.
\item If $b<0$, then
\[
 \lim_{s\downarrow0}s^{1/2}N_{+,k}(s)
   =\frac h\pi\sqrt{\frac{-b}{\nu}} .
\]
\item Under the strict stability hypotheses of \autoref{thm_linear_large_time}, the slowly damped count $N_{-,k}(s)$ at $k$, computed with the fixed cutoff $a_{\rm c}$, satisfies
\[
 \lim_{s\downarrow0}s^{1/2}N_{-,k}(s)
   =\frac h\pi\sqrt{\frac{b}{\nu}} .
\]
\end{enumerate}
\end{proposition}

For constant adverse stratification $c=c_*>0$ and $b>0$, the Dirichlet eigenvalues in \autoref{prop:fixed-mode-count}(1) are $b(n\pi/h)^2-k^2c_*$, $n\ge1$. With $\zeta=b\pi^2/(k^2c_*h^2)$, the total number at this $k$ is therefore
\begin{equation}\label{eq:fixed-constant-count}
 N_{+,k}(0+)=\#\{n\ge1:n^2\zeta<1\}
            =\lceil\zeta^{-1/2}\rceil-1 .
\end{equation}
The inequality is strict. At $\zeta=n^{-2}$ the $n$th Dirichlet mode does not contribute to the growing-mode count. \autoref{fig:fixed-count} shows the loss of one growing mode at each threshold, including the endpoints.

\begin{figure}[!htbp]
\centering
\begin{tikzpicture}[x=8.35cm,y=.63cm]
  \draw[figArrow] (0,0)--(1.15,0) node[right] {$\zeta$};
  \draw[figArrow] (0,0)--(0,5.29) node[above] {$N_{+,k}(0+)$};
  \foreach \n in {1,2,3,4,5}{
    \pgfmathsetmacro{\threshold}{1/(\n*\n)}
    \draw[figInk!27,dashed] (\threshold,0)--(\threshold,\n);
    \ifnum\n<5
      \draw[figBlue,line width=1pt] ({1/((\n+1)*(\n+1))},\n)--(\threshold,\n);
    \fi
  }
  \draw[figBlue,line width=1pt] (1,0)--(1.105,0);
  \draw[figBlue,densely dotted] ({1/36},5)--({1/25},5);
  \foreach \n in {1,2,3,4,5}{
    \pgfmathsetmacro{\threshold}{1/(\n*\n)}
    \fill[figBlue] (\threshold,{\n-1}) circle[radius=1.85pt];
    \filldraw[fill=white,draw=figBlue,line width=.70pt] (\threshold,\n) circle[radius=1.85pt];
  }
  \foreach \y in {1,2,3,4,5}{
    \draw (.004,\y)--(-.004,\y) node[left=2pt] {$\y$};
  }
  \node[below=3pt] at (.25,0) {$1/4$};
  \node[below=3pt] at ({1/9},0) {$1/9$};
  \node[below=3pt] at (1,0) {$1$};
  \node[figBlue,align=center] at (.60,2.94) {drops at $\zeta=n^{-2}$};
\end{tikzpicture}
\caption{Exact coupling thresholds at one fixed horizontal wave number $k$ for constant adverse stratification $c=c_*>0$. The horizontal variable is $\zeta=b\pi^2/(k^2c_*h^2)>0$. On the displayed range, a closed point gives the count at the threshold, and an open point gives the left limit. The steps continue to higher counts as $\zeta\downarrow0$. After the normalization following \eqref{eq:threshold-count-envelope}, the same staircase gives the sharp count envelope over equimeasurable profiles.}
\label{fig:fixed-count}
\end{figure}

\FloatBarrier

\begin{proposition}\label{prop:threshold-location}
In \autoref{prop:fixed-mode-count}, suppose $c>0$ on $[0,h]$. Let $D_z=-\partial_z^2$ with Dirichlet conditions on $(0,h)$. For each $k\ge1$ the ordered thresholds $B_{n,k}[c]$ are $k^2$ times the positive eigenvalues of $D_z^{-1/2}cD_z^{-1/2}$, in decreasing order. They are simple and satisfy
\begin{equation}\label{eq:threshold-string-count}
 N_{+,k}(0+;b)=\#\{n:B_{n,k}[c]>b\},\qquad
 B_{1,k}[c]=k^2\sup_{\phi\in H^1_0,\ \phi\ne0}
                     \frac{\int_0^h c|\phi|^2\,dz}{\int_0^h|\phi'|^2\,dz}.
\end{equation}
All these thresholds are independent of the Navier coefficient $\alpha$ and are unchanged by reflecting $c$ about the midplane.

Let $c^\#$ be the symmetric decreasing rearrangement of $c$ about $h/2$, preserving its distribution against $dz$. Then
\begin{equation}\label{eq:threshold-rearrangement}
 B_{1,k}[c]\le B_{1,k}[c^\#].
\end{equation}
In particular, for $0<\epsilon<c_0$ the profiles $c_{\rm w}=c_0+\epsilon\cos(2\pi z/h)$ and $c_{\rm m}=c_0-\epsilon\cos(2\pi z/h)$ have identical leading Weyl responses, while $B_{1,k}[c_{\rm w}]<B_{1,k}[c_{\rm m}]$ for every $k$.

For a perturbation $c=c_0+\epsilon g$ of a positive constant $c_0$, with $g$ smooth and $\epsilon$ near zero, the thresholds have the first-order responses
\begin{equation}\label{eq:threshold-linearized-data}
 \mathcal D_n:=\frac{(n\pi/h)^2}{k^2}
           \left.\frac{d}{d\epsilon}B_{n,k}[c_0+\epsilon g]\right|_{\epsilon=0}
   =\overline g-\widehat g_{2n},
 \qquad
 \widehat g_j=\frac1h\int_0^h g(z)\cos(j\pi z/h)\,dz,\quad
 \overline g=\widehat g_0 .
\end{equation}
Thus $\overline g=\lim_{n\to\infty}\mathcal D_n$, and the data $\{\mathcal D_n\}_{n\ge1}$ determine the midplane-symmetric part of $g$ by
\begin{equation}\label{eq:threshold-symmetric-reconstruction}
 \frac{g(z)+g(h-z)}2
 =\overline g+2\sum_{n\ge1}(\overline g-\mathcal D_n)\cos(2n\pi z/h)
 \quad\text{in }L^2(0,h).
\end{equation}
\end{proposition}

The instability thresholds are thus determined by a weighted Dirichlet eigenvalue problem. Because $B_{n,k}/k^2$ is independent of $k$, a second wave number rescales the same sequence. Additional data in inverse string problems include norming constants and spectra obtained with different endpoint conditions \cite{RundellSacks2013}. \autoref{fig:threshold-location} compares the two equimeasurable profiles and shows which linearized threshold responds to their first cosine component. For finite measurements with known mean $\bar g$, an $m$-term reconstruction using observations $\widetilde{\mathcal D}_n$ has normalized squared $L^2$ error
\begin{equation}\label{eq:threshold-truncation}
 \frac1h\left\|\overline g+
  2\sum_{n=1}^m(\overline g-\widetilde{\mathcal D}_n)\cos(2n\pi z/h)
       -\frac{g+g(h-\cdot)}2\right\|_2^2
 =2\sum_{n=1}^m|\widetilde{\mathcal D}_n-\mathcal D_n|^2
  +2\sum_{n>m}|\widehat g_{2n}|^2 .
\end{equation}
This follows by cosine orthogonality in \eqref{eq:threshold-symmetric-reconstruction} and separates the effect of errors in measured slopes from the omitted spectral components.

\begin{figure}[!htbp]
\centering
\begin{tikzpicture}[x=1cm,y=1cm]
 \node[figPanel] at (2.40,3.09) {(a) Wall and midplane profiles};
 \node[figPanel] at (8.52,3.09) {(b) Normalized threshold slopes};
 \begin{scope}[x=4.12cm,y=1.35cm,shift={(.02,.25)}]
   \draw[figArrow] (0,0)--(1.09,0) node[right] {$z/h$};
   \draw[figArrow] (0,0)--(0,1.77) node[left] {$c$};
   \draw[figBlue,line width=.95pt,domain=0:1,samples=84] plot(\x,{1.12+.37*cos(360*\x)});
   \draw[figOrange,densely dashed,line width=.95pt,domain=0:1,samples=84] plot(\x,{1.12-.37*cos(360*\x)});
   \node[figBlue,above right] at (.07,1.43) {$c_{\rm w}$};
   \node[figOrange,above] at (.52,1.49) {$c_{\rm m}$};
   \node[below] at (1,0) {$1$};
   \node[below] at (.50,0) {$1/2$};
 \end{scope}
 \begin{scope}[shift={(6.45,1.26)},x=1.27cm,y=1.62cm]
   \draw[figArrow] (0,0)--(3.38,0) node[right] {$n$};
   \draw[figArrow] (0,-.65)--(0,.68) node[above] {$\mathcal D_n$};
   \draw[figInk!26,dashed] (0,.50)--(2.98,.50) (0,-.50)--(2.98,-.50);
   \node[left] at (0,.50) {$1/2$};
   \node[left] at (0,-.50) {$-1/2$};
   \node[below left=1pt] at (.84,0) {$1$};
   \foreach \x/\n in {1.74/2,2.64/3}{
      \node[below=2pt] at (\x,0) {$\n$};
   }
   \draw[figBlue,thick] (.84,0)--(.84,-.50);
   \fill[figBlue] (.84,-.50) circle[radius=2pt];
   \draw[figOrange,densely dashed,thick] (.84,0)--(.84,.50);
   \fill[figOrange] (.84,.50) circle[radius=2pt];
   \foreach \x in {1.74,2.64}{
     \filldraw[fill=white,draw=figBlue,line width=.8pt] (\x,0) circle[radius=2.2pt];
     \fill[figOrange] (\x,0) circle[radius=1.13pt];
   }
   \node[anchor=west,figOrange] at (1.29,.42) {$g=-\cos(2\pi z/h)$};
   \node[anchor=west,figBlue] at (1.29,-.40) {$g=\cos(2\pi z/h)$};
 \end{scope}
\end{tikzpicture}
\caption{Vertical information in \autoref{prop:threshold-location}. (a) Equimeasurable wall-enhanced and midplane-enhanced profiles have the same leading Weyl response. Their first fixed-mode thresholds satisfy $B_{1,k}[c_{\rm w}]<B_{1,k}[c_{\rm m}]$. (b) At a constant reference profile, the normalized derivatives in \eqref{eq:threshold-linearized-data} already distinguish the two perturbations at $n=1$. For these perturbations the derivatives vanish exactly for every $n\ge2$.}
\label{fig:threshold-location}
\end{figure}

\FloatBarrier

\begin{corollary}[Sharp count over equimeasurable weights]\label{cor:threshold-envelope}
Let $b>0$ and let $c$ be a continuous positive weight on $[0,h]$. For $m\ge1$, let $c^{\#,[m]}$ repeat a rescaled copy of $c^\#$ on each of the $m$ equal subintervals. The sharp threshold bound is
\begin{equation}\label{eq:threshold-envelope}
 B_{m,k}[c]\le B_{m,k}[c^{\#,[m]}]=m^{-2}B_{1,k}[c^\#].
\end{equation}
In particular, with the fixed-mode count defined by the Dirichlet index,
\begin{equation}\label{eq:threshold-count-envelope}
 \max_{\substack{\widetilde c\text{ continuous}\\
                 \widetilde c\text{ equimeasurable with }c}}
 N_{+,k}(0+;b,\widetilde c)
 =\left\lceil\sqrt{B_{1,k}[c^\#]/b}\right\rceil-1 .
\end{equation}
\end{corollary}

This is the fixed-mode consequence of the classical equimeasurable-string comparison \cite{Schwarz1961}. We give a direct proof in \autoref{sec:spectral-proof}. With $\zeta=b/B_{1,k}[c^\#]$, the staircase in \autoref{fig:fixed-count} also plots the sharp envelope. For the cosine distributions in \autoref{fig:threshold-location}, the optimizing repeats are smooth.

\Needspace{9\baselineskip}
\subsection{Global solutions and weak scalar limits}\label{sec:global-results}

We establish global solvability for rough scalar data and then identify the force balance satisfied by every weak scalar limit. The joint second moment tests for strong convergence. See \autoref{sec:global-proof} and \autoref{sec:limits-proof}.

\begin{theorem}\label{thm_nonlinear_regularity}
Assume \eqref{eq_background} and $\nabla\Psi\in W^{1,\infty}(\Omega)$. Let $2<q<\infty$. For every $(\rho_0,v_0)\in L^q(\Omega)^2$ and $\mathbf u_0\in D(A)$ with the parity in \eqref{eq_perturbation_parity}, the problem \eqref{eq_perturbation} has a unique global strong solution, with $p$ of spatial mean zero. Here $D(A)$ comprises the solenoidal $H^2$ velocities satisfying \eqref{bd-con-2} and the prescribed parity. The data may have arbitrary size, and the result holds for all signs of $\delta$ and $f+a_0$.

\begin{enumerate}[(1)]
\item The solution satisfies
\begin{equation}\label{eq:strong-continuity}
\begin{aligned}
 (\rho,v)&\in C([0,\infty);L^q(\Omega)^2),&
 \mathbf u&\in C([0,\infty);D(A)),\\
 \mathbf u_t&\in C([0,\infty);L^2(\Omega)^2),&
 p&\in C([0,\infty);H^1(\Omega)).
\end{aligned}
\end{equation}
For each $T>0$,
\begin{equation}\label{eq:velocity-lq}
 \int_0^T \|\mathbf u(t)\|_{W^{2,q}}^2\,dt\le C(1+T).
\end{equation}
The accumulated strain $\Lambda(T)=\int_0^T\|\nabla\mathbf u(t)\|_{L^\infty}\,dt$ satisfies
\begin{equation}\label{eq:sublinear-strain}
 \Lambda(T)\le C(1+T)^{\eta_q},
 \qquad \Lambda(T)=o(T^{\eta_q}),\qquad
 \eta_q=1-\frac{q-2}{4(q-1)} .
\end{equation}
For each finite time, the flow and its inverse are $C^1$ and bi-Lipschitz on each closed half cell, and the flow preserves area. On each half cell, both $\rho+\rho_s$ and $v+v_s+fx$ are the pullbacks of their respective initial values by the inverse flow. In addition,
\begin{subequations}
\label{eq_nonlinear_global_regularity}
     \begin{align}\label{theorem-1-conc-1}
&\mathbf{u} \in L^{\infty}\left(\left(0,\infty\right);H^2(\Omega)\right)
\cap L^{p}\left(\left(0,\infty\right);W^{1,p}(\Omega)\right),\quad 2\leq p<\infty,\\
\label{theorem-1-conc-2}
&\mathbf{u}_t \in
L^{\infty}\left(\left(0,\infty\right);L^2(\Omega)\right)
\cap L^{2}\left(\left(0,\infty\right);H^{1}(\Omega)\right),\\
\label{theorem-1-conc-3}
&p\in L^{\infty}\left(\left(0,\infty\right);H^1(\Omega)\right),\\
\label{theorem-1-conc-4}
&(\rho,v) \in L^{\infty}\left(\left(0,\infty\right);L^s(\Omega)\right),\quad 1\leq s\leq q.
 \end{align}
   \end{subequations}
\item If additionally $(\rho_0,v_0)\in W^{1,r}(\Omega)^2$ for some $2\le r\le\infty$, then for every $T>0$,
\begin{subequations}
\label{eq_nonlinear_finite_time_regularity}
     \begin{align}\label{theorem-1-conc-11}
&\mathbf{u} \in L^{2}\left(\left(0,T\right);H^3(\Omega)\right)
\cap L^{\frac{2p}{p-2}}\left(\left(0,T\right);W^{2,p}(\Omega)\right),\quad 2<p<\infty,\\
\label{eq_nonlinear_h2_endpoint}
&\mathbf{u}\in L^{\infty}\left(\left(0,T\right);H^2(\Omega)\right),\quad p=2,\\
\label{theorem-1-conc-12}
&(\rho,v) \in L^{\infty}\left(\left(0,T\right);W^{1,s}(\Omega)\right),\quad 1\leq s\leq r.
 \end{align}
   \end{subequations}
\item Let $(\rho^{(i)},v^{(i)},\mathbf u^{(i)},p^{(i)})$, $i=1,2$, be two solutions with the same background and viscosity. Set $U=\mathbf u^{(1)}-\mathbf u^{(2)}$, $d_\rho=\rho^{(1)}-\rho^{(2)}$ and $d_v=v^{(1)}-v^{(2)}$. Write $H^1_{\rm per}$ for the periodic $H^1$ space with arbitrary wall traces and $\|\cdot\|_{-1,N}$ for its dual norm. Then
\begin{equation}\label{eq:euler-dependence}
\begin{aligned}
 \sup_{0\le t\le T}\bigl(
  \|U(t)\|_2^2+\|d_\rho(t)\|_{-1,N}^2+\|d_v(t)\|_{-1,N}^2\bigr)
 +\int_0^T\|U(t)\|_{H^1}^2\,dt\\
 \le C_T\bigl(\|U(0)\|_2^2+\|d_\rho(0)\|_{-1,N}^2+\|d_v(0)\|_{-1,N}^2\bigr).
\end{aligned}
\end{equation}
The constant is uniform for initial data in bounded subsets of $D(A)\times L^q(\Omega)^2$.
\end{enumerate}
\end{theorem}

\begin{theorem}\label{thm_nonlinear_large_time}
Let $(\rho,v,\mathbf u,p)$ be the global solution of \autoref{thm_nonlinear_regularity}. Then, for every $1\le s<\infty$,
\begin{subequations}
\label{eq_nonlinear_decay}
     \begin{align}\label{theorem-2-conc-1}
&\norm{\mathbf{u} }_{W^{1,s}}\to 0,\quad t\to \infty,\\
\label{theorem-2-conc-2}
&\norm{\mathbf{u}_t }_{L^{2}}\to 0,\quad t\to \infty, \\
\label{theorem-2-conc-3}
&\norm{fv{\bf e}_{1}-\nabla p-\rho\nabla \Psi}_{H^{-1}}\to 0 ,\quad t\to \infty.
 \end{align}
 \end{subequations}
In addition, the force residual in \eqref{theorem-2-conc-3} belongs to $L^2(0,\infty;H^{-1})$. With $p$ normalized to have zero spatial mean, every weak limit point $(\rho_*,v_*,p_*)$ of $(\rho,v,p)$ in $L^2(\Omega)\times L^2(\Omega)\times H^1(\Omega)$ as $t\to\infty$ satisfies the distributional identities
\begin{equation}\label{eq_nonlinear_weak_balance}
\begin{aligned}
fv_*-\partial_x p_* &= \rho_*\partial_x\Psi,\\
-\partial_z p_* &= \rho_*\partial_z\Psi .
\end{aligned}
\end{equation}

For $\gamma>0$ and $\beta\in\R$, define
\begin{equation}\label{eq:profile-energy}
 \begin{aligned}
 r_{\gamma,\beta}&=-\gamma\Psi+\beta,\\
 \mathcal E_{\gamma,\beta}(t)
 &=\norm{\rho(t)+\rho_s-r_{\gamma,\beta}}_2^2
    +\gamma\norm{v(t)+v_s}_2^2+\gamma\norm{\mathbf u(t)}_2^2 .
 \end{aligned}
\end{equation}
The scalar defect satisfies
\begin{equation}\label{theorem-2-conc-4}
 \norm{\rho(t)+\rho_s-r_{\gamma,\beta}}_2^2
    +\gamma\norm{v(t)+v_s}_2^2
 \longrightarrow C_{\gamma,\beta},
\end{equation}
where
\begin{equation}\label{eq_nonlinear_energy_constant}
 C_{\gamma,\beta}=\mathcal E_{\gamma,\beta}(0)-
           2\gamma\nu\int_0^\infty \Diss(\mathbf u(t))\,dt\ge0 .
\end{equation}
For this same pair $(\gamma,\beta)$,
\begin{equation}\label{eq_nonlinear_exact_profile}
\begin{aligned}
\rho(t)+\rho_s&\longrightarrow-\gamma\Psi+\beta
   &&\text{in }L^2(\Omega),\\
v(t)+v_s&\longrightarrow0
   &&\text{in }L^2(\Omega)
\end{aligned}
\qquad\text{as }t\to\infty
\end{equation}
if and only if $C_{\gamma,\beta}=0$.
\end{theorem}

\begin{proposition}\label{prop:limit-set}
Let $(\rho,v,\mathbf u,p)$ be the solution of \autoref{thm_nonlinear_regularity}, and set $Z=(\rho+\rho_s,v+v_s+fx)$ on the representative cell.
\begin{enumerate}[(1)]
\item In the product dual space with norm $\|\cdot\|_{-1,N}$, the set of all scalar limit points as $t\to\infty$ is nonempty, compact and connected. Its members, after subtracting the fixed backgrounds in $Z$, satisfy \eqref{eq_nonlinear_weak_balance} for a unique periodic mean-zero pressure.
\item Along any sequence $t_j\to\infty$ with $Z(t_j)\rightharpoonup Z_\infty$ in $L^2(\Omega)^2$,
\begin{equation}\label{eq:variance-defect}
 \lim_{j\to\infty}\|Z(t_j)-Z_\infty\|_2^2
  =\sum_{\sigma\in\{-,+\}}
       \bigl(\|Z(0)\|_{L^2(\Omega_\sigma)}^2
               -\|Z_\infty\|_{L^2(\Omega_\sigma)}^2\bigr).
\end{equation}
Consequently, strong convergence holds along this sequence if and only if $Z_\infty$ has the initial joint distribution on each half cell. In that case the velocity converges to zero in $H^2$ and the pressure converges in $H^1$ along the same sequence.
\item If $\int_0^\infty\|\nabla\mathbf u(t)\|_{L^\infty}\,dt<\infty$, the flow maps and their inverses converge in $C^1$ on each closed half cell to measure-preserving bi-Lipschitz diffeomorphisms $X_\infty$ and $X_\infty^{-1}$. Then $Z(t)\to Z(0)\circ X_\infty^{-1}$ strongly in $L^q$, the velocity converges to zero in $H^2$, and the pressure converges in $H^1$. If $Z(0)$ is in $W^{1,q}$ on each half cell, then
\begin{equation}\label{eq:strain-tail}
 \|Z(t)-Z(0)\circ X_\infty^{-1}\|_{L^q}
 \le C\left(\sum_{\sigma\in\{-,+\}}\|\nabla Z(0)\|_{L^q(\Omega_\sigma)}\right)
           \int_t^\infty\|\nabla\mathbf u(s)\|_{L^\infty}\,ds ,
\end{equation}
where $C$ may depend on the total accumulated strain.
\end{enumerate}
\end{proposition}

\Needspace{9\baselineskip}
\subsection{Joint distribution and optimal transport}\label{sec:transport-results}

The joint laws constrain the defect and restrict the possible weak limits. The relaxed minimization below is over balanced fields with a convex-order constraint. For a given trajectory, achieving this minimum is an additional condition on its limiting defect. See \autoref{sec:transport-proof} for the proofs and examples.

To state the joint transport bounds, we assume the hypotheses of \autoref{thm_nonlinear_large_time} and let $\mu_\pm$ be normalized Lebesgue measure on $\Omega_\pm$. Fix $\gamma>0$ and $\beta\in\R$, and set
\[
\begin{aligned}
 Z_0&=(\rho_0+\rho_s,\sqrt{\gamma}\,(v_0+v_s+fx)),\\
 Z_*&=(r_{\gamma,\beta},\sqrt{\gamma}\,fx).
\end{aligned}
\]
Then the defect in \eqref{eq_nonlinear_energy_constant} satisfies
\begin{equation}\label{eq:joint-lower-bound}
 C_{\gamma,\beta}\ge
 \sum_{\sigma\in\{-,+\}}|\Omega_\sigma|\,
 W_2^2\bigl((Z_0)_\#\mu_\sigma,(Z_*)_\#\mu_\sigma\bigr),
\end{equation}
where $W_2$ is the quadratic Wasserstein distance on $\R^2$. In particular, convergence to the target in \eqref{eq_nonlinear_exact_profile} requires equality of the target and initial joint distributions on each half cell. The oddness of $v+a_0x$ also gives $C_{\gamma,\beta}\ge\gamma|\Omega|a_1^2$.

\begin{proposition}\label{prop:convex-order}
With $Z_0$ and $Z_*$ as above, set $Z_t=(\rho(t)+\rho_s,\sqrt\gamma(v(t)+v_s+fx))$. On a half cell, write $\lambda_0=(Z_0)_\#\mu_\sigma$. For laws with finite second moment, write $\lambda\preceq_{\rm cx}\lambda_0$ if $\int\phi\,d\lambda\le\int\phi\,d\lambda_0$ for every continuous convex $\phi$ of at most quadratic growth. In particular the two laws have the same mean, by testing with both signs of each linear function. Put $m_2(\lambda)=\int |z|^2\,d\lambda(z)$.

\begin{enumerate}[(1)]
\item The sequential weak $L^2(\mu_\sigma)$ closure of the fields $Z_0\circ S$, where $S$ runs over measure-preserving bijections modulo null sets of the half cell, consists exactly of the fields $W$ with $W_\#\mu_\sigma\preceq_{\rm cx}\lambda_0$. In particular, every weak scalar accumulation point of the flow obeys this convex-order constraint on each half cell.
\item Suppose $Z_{t_j}\rightharpoonup W$ in $L^2$ as $t_j\to\infty$, and let $\lambda_\sigma=W_\#\mu_\sigma$ and $\lambda_{*,\sigma}=(Z_*)_\#\mu_\sigma$. The covariance defect
\[
 D_\sigma=\int z\otimes z\,d(Z_0)_\#\mu_\sigma
                     -\int w\otimes w\,d\lambda_\sigma
\]
is positive semidefinite. A scalar combination $e\cdot Z_{t_j}$ converges strongly to $e\cdot W$ on the half cell exactly when $e\in\ker D_\sigma$. Moreover
\begin{equation}\label{eq:mixing-target-defect}
\begin{aligned}
 C_{\gamma,\beta}
 &=\sum_\sigma|\Omega_\sigma|
    \left[\int |W-Z_*|^2\,d\mu_\sigma+\operatorname{tr}D_\sigma\right]\\
 &\ge\sum_\sigma|\Omega_\sigma|
    \left[W_2^2(\lambda_\sigma,\lambda_{*,\sigma})
                       +m_2((Z_0)_\#\mu_\sigma)-m_2(\lambda_\sigma)\right].
\end{aligned}
\end{equation}
The last expression is at least the right-hand side of \eqref{eq:joint-lower-bound}.
\item If the initial and target joint laws have identical one-dimensional marginals on both half cells, weak convergence to $Z_*$ along any sequence of times tending to infinity already requires equality of the joint laws on both half cells.
\end{enumerate}
\end{proposition}

\begin{theorem}\label{thm:balanced-transport-optimizer}
We use the notation preceding \autoref{prop:convex-order}. Set
\[
 \lambda_{0,\sigma}=(Z_0)_\#\mu_\sigma,\qquad
 \lambda_{*,\sigma}=(Z_*)_\#\mu_\sigma,\qquad
 \mathcal W=\sum_\sigma|\Omega_\sigma|W_2^2(\lambda_{0,\sigma},\lambda_{*,\sigma}).
\]
We call a field $W$ admissible for the relaxed minimum if it has the reflection of $Z_t$, its reconstructed perturbations satisfy \eqref{eq_nonlinear_weak_balance}, and $W_\#\mu_\sigma\preceq_{\rm cx}\lambda_{0,\sigma}$ for both half cells. Define
\begin{equation}\label{eq:balanced-relaxed-cost}
 \mathcal J(W)=\|W-Z_*\|_2^2+
  \sum_\sigma|\Omega_\sigma|\bigl[m_2(\lambda_{0,\sigma})-m_2(W_\#\mu_\sigma)\bigr].
\end{equation}
\begin{enumerate}[(1)]
\item For every smooth potential in the model, including potentials depending on $x$ and $z$, the minimum of $\mathcal J$ over the admissible fields is attained and equals $\mathcal W$. Each minimizer gives an equilibrium with zero planar velocity. Along a nonlinear trajectory, $C_{\gamma,\beta}=\mathcal W$ exactly when a weak scalar accumulation point is a relaxed minimizer. In that case every weak scalar accumulation point is a relaxed minimizer.
\item If $\lambda_{*,\sigma}$ is absolutely continuous on both half cells, there is a unique minimizer $Z_{\rm opt}$ and its joint laws are exactly $\lambda_{0,\sigma}$. This absolute continuity holds in particular if $\partial_z\Psi\ne0$ almost everywhere. The formula is $Z_{\rm opt}=T_\sigma(Z_*)$ on each half cell, with $T_\sigma=\nabla\varphi_\sigma$ the quadratic optimal transport map. The field also uniquely minimizes $\|Z-Z_*\|_2^2$ among all fields with the initial joint laws. Its reconstructed perturbations, together with zero planar velocity and the normalized pressure, form a stationary solution. Along the original nonlinear trajectory,
\begin{equation}\label{eq:transport-floor-attainment}
 C_{\gamma,\beta}=\mathcal W
 \quad\Longleftrightarrow\quad
 Z_t\longrightarrow Z_{\rm opt}\quad\text{strongly in }L^2(\Omega)^2 .
\end{equation}
In this case the velocity and pressure converge in $H^2$ and $H^1$, respectively, to their balanced values.
\item Under the hypotheses of part (2), if the transporting potentials admit global convex $C^{1,1}$ extensions with gradient Lipschitz constants at most $L>0$, then
\begin{equation}\label{eq:transport-excess-control}
 \|Z_t-Z_{\rm opt}\|_2^2\le L\bigl(\|Z_t-Z_*\|_2^2-\mathcal W\bigr).
\end{equation}
When $C_{\gamma,\beta}=\mathcal W$, the right-hand side is at most $2L\gamma\nu\int_t^\infty\Diss(\mathbf u(s))\,ds$.
\end{enumerate}
\end{theorem}

\FloatBarrier

\Needspace{9\baselineskip}
\subsection{Nonlinear stability and escape}\label{sec:stability-results}

On the reference joint rearrangement class, the second variation relates the linear instability criterion to nonlinear escape. The energy minimum describes the stable regime. See \autoref{sec:stability-proof}.

\begin{proposition}\label{prop:joint-second-variation}
Let $\mathbf W=(W_1,W_2)$ be a $C^1$ periodic solenoidal velocity in $\mathcal V$, the impermeable reflection class of \autoref{thm_linear_stability}. Let $X_\eta$ be its autonomous flow on each half cell. Rearrange the background pair by setting
\[
 (R_\eta,M_\eta)
   =(\rho_s,v_s+fx)\circ X_{-\eta},\qquad
 V_\eta=M_\eta-fx .
\]
For
\[
 \mathcal P(\eta)=\tfrac12\|V_\eta\|_2^2+\int_\Omega R_\eta\Psi\,dx\,dz ,
\]
one has $\mathcal P'(0)=0$ and
\begin{equation}\label{eq:joint-second-variation}
 \mathcal P''(0)=-E_2(\mathbf W).
\end{equation}
In particular, a growing mode exists if and only if a smooth impermeable joint rearrangement direction gives negative second variation of $\mathcal P$.
\end{proposition}

\begin{proposition}\label{prop:nonlinear-stability}
Assume the hypotheses of the perturbation formulation and suppose $\rho_s=\mathfrak r(\Psi)$, where $\mathfrak r\in C^1(I)$ on an interval $I$ containing the range of $\Psi$, with
\[
 0<m\le-\mathfrak r'(\xi)\le M<\infty\quad(\xi\in I),
 \qquad f+a_0>0.
\]
Let $G$ be a uniformly convex $C^2$ function on $\mathbb R$ with
$G'(\rho_s)=-\Psi$ and $M^{-1}\le G''\le m^{-1}$. For the solution of \autoref{thm_nonlinear_regularity}, set $R=\rho_s+\rho$ and
\[
 \mathcal B(t)=\tfrac12\|\mathbf u\|_2^2+\frac{f}{2(f+a_0)}\|v\|_2^2
 +\int_\Omega\!\left[G(R)-G(\rho_s)-G'(\rho_s)(R-\rho_s)\right]dx\,dz .
\]
Then for every $t\ge0$,
\begin{equation}\label{eq:nonlinear-lyapunov}
 \mathcal B(t)+\nu\int_0^t\Diss(\mathbf u(s))\,ds=\mathcal B(0).
\end{equation}
In particular, the background is nonlinearly Lyapunov stable in the product $L^2$ norm in this global solution class. The energy estimate holds for perturbations of arbitrary size. Moreover, if $(\rho+\rho_s,v+v_s+fx)$ has the same joint distribution as $(\rho_s,v_s+fx)$ on each half cell, then
\begin{equation}\label{eq:lyapunov-mechanical}
 E_{\rm mech}-E_s=\mathcal B .
\end{equation}
Here $E_{\rm mech}$ is defined by \eqref{eq:generic-mechanical-energy} and $E_s$ is its background value. In particular, the equilibrium strictly minimizes mechanical energy on its joint rearrangement class, with a quadratic lower bound in the product $L^2$ norm.
\end{proposition}

\begin{theorem}\label{thm:nonlinear-escape}
Assume that the pointwise instability criterion in \autoref{thm_linear_stability} holds. Let $\lambda_*>0$ be the largest growth rate and let $Y_*=(\mathbf u_*,\rho_*,v_*)$ be a corresponding real mode, normalized in the product $L^2$ norm. Let $X_s^*$ be the autonomous flow generated by $\mathbf u_*/\lambda_*$. Take initial planar velocity $\epsilon\mathbf u_*$ and the total initial scalar pair
\begin{equation}\label{eq:escape-invariant-data}
 (\rho_s+\rho^\epsilon(0),v_s+v^\epsilon(0)+fx)
        =(\rho_s,v_s+fx)\circ X_{-\epsilon}^* .
\end{equation}
The initial scalar pair therefore has the same joint distribution as the equilibrium pair on each half cell. Write $Y^\epsilon=(\mathbf u^\epsilon,\rho^\epsilon,v^\epsilon)$ for the corresponding global solution.

There exist $\theta>0$, $C<\infty$ and $c_*>0$ such that for all sufficiently small $\epsilon>0$,
\begin{equation}\label{eq:escape-approximation}
 \|Y^\epsilon(t)-\epsilon e^{\lambda_*t}Y_*\|_2
 \le C\epsilon^2e^{2\lambda_*t},\qquad
 0\le t\le T_\epsilon:=\lambda_*^{-1}\log(\theta/\epsilon).
\end{equation}
In particular, $\|\mathbf u^\epsilon(T_\epsilon)\|_2\ge c_*$ and the equilibrium is nonlinearly unstable in the product $L^2$ norm.

We define the mechanical energy and its background value by
\begin{equation}\label{eq:mechanical-energy}
\begin{aligned}
 E_{\rm mech}^\epsilon(t)
 &=\tfrac12\|\mathbf u^\epsilon(t)\|_2^2
   +\tfrac12\|v_s+v^\epsilon(t)\|_2^2
   +\int_\Omega(\rho_s+\rho^\epsilon(t))\Psi\,dx\,dz,\\
 E_s&=\tfrac12\|v_s\|_2^2+\int_\Omega\rho_s\Psi\,dx\,dz.
\end{aligned}
\end{equation}
The initial energy satisfies
\begin{equation}\label{eq:escape-initial-energy}
 E_{\rm mech}^\epsilon(0)=E_s-\frac{\nu}{2\lambda_*}\Diss(\mathbf u_*)\epsilon^2+o(\epsilon^2).
\end{equation}
There is a constant $d_*>0$, independent of $\epsilon$, such that
\begin{equation}\label{eq:energy-release}
 \nu\int_0^{T_\epsilon}\Diss(\mathbf u^\epsilon(t))\,dt
 \ge 2d_* ,\qquad
 E_{\rm mech}^\epsilon(t)\le E_s-d_*\quad(t\ge T_\epsilon).
\end{equation}
Consequently, every weak $L^2$ accumulation point $(\rho_\infty^\epsilon,v_\infty^\epsilon)$ along a sequence $t_j\to\infty$ satisfies
\begin{equation}\label{eq:energy-limit-gap}
 \tfrac12\|v_s+v_\infty^\epsilon\|_2^2+
     \int_\Omega(\rho_s+\rho_\infty^\epsilon)\Psi\,dx\,dz
 \le E_s-d_* .
\end{equation}
All these balanced limits stay a positive product $L^2$ distance from the original equilibrium, uniformly for the small initial amplitudes considered here.
\end{theorem}

\begin{theorem}\label{thm:vertical-nonlinear-dichotomy}
Assume $\Psi=\Psi(z)$ is monotone on $[0,h]$ and set $c(z)=\rho_s'(z)\Psi'(z)$ and $\beta_f=f+a_0$. Consider the global solution class of \autoref{thm_nonlinear_regularity} restricted to initial scalar pairs whose joint law on each half cell is that of $(\rho_s,v_s+fx)$.
\begin{enumerate}[(1)]
\item If $\beta_f<0$ or $c$ is positive somewhere, the equilibrium is nonlinearly unstable within this class, with the escape and energy separation of \autoref{thm:nonlinear-escape}.
\item If $\beta_f\ge0$ and $c\le0$, the equilibrium is nonlinearly Lyapunov stable within this class. For any $L>0$ with $L\ge\|\delta\|_{L^\infty}$ and for every $t\ge0$,
\begin{equation}\label{eq:critical-leaf-coercivity}
 E_{\rm mech}(t)-E_s
 \ge \tfrac12\|\mathbf u(t)\|_2^2+\tfrac{1}{2L}\|\rho(t)\|_2^2
              +\frac{f}{2\beta_f}\|v(t)\|_2^2
 \qquad(\beta_f>0).
\end{equation}
For $\beta_f=0$, one has $v\equiv0$ and the lower bound consists of the planar and density terms in \eqref{eq:critical-leaf-coercivity}. In either case the equilibrium is the unique mechanical energy minimizer within the class.
\end{enumerate}
\end{theorem}

\autoref{fig:nonlinear-classification} displays the two alternatives of \autoref{thm:vertical-nonlinear-dichotomy}. It includes the critical edges in the stable region. On the zero-coupling edge the reference joint law forces $v\equiv0$.

\begin{figure}[!htbp]
\centering
\begin{tikzpicture}[x=1cm,y=.89cm]
  \fill[figOrange!11] (-4,-2.1) rectangle (4,1.78);
  \fill[figTeal!15] (0,-2.1) rectangle (4,0);
  \draw[figArrow] (-4.17,0)--(4.32,0) node[right] {$\beta_f$};
  \draw[figArrow] (0,-2.22)--(0,2.04) node[above] {$\max_{[0,h]}c$};
  \draw[figTeal,line width=1.35pt] (0,-2.1)--(0,0)--(4,0);
  \fill[figTeal] (0,0) circle[radius=2pt];
  \node[figOrange,align=center] at (-2.27,1.02) {nonlinear instability\\$\beta_f<0$};
  \node[figOrange,align=center] at (-2.27,-1.02) {nonlinear instability\\$\beta_f<0$};
  \node[figOrange,align=center] at (2.19,1.02) {nonlinear instability\\$\max c>0$};
  \node[figTeal,align=center] at (2.19,-.80) {Lyapunov stability\\$\beta_f\ge0,\ \max c\le0$};
  \draw[figTeal] (0,-1.65)--(.42,-1.65);
  \node[figTeal,anchor=west] at (.42,-1.65) {$\beta_f=0:\ v\equiv0$};
\end{tikzpicture}
\caption{Nonlinear classification in the reference joint-distribution class for a monotone vertical potential. The teal region includes its two boundary segments; at $\beta_f=0$ the fixed meridional label forces $v\equiv0$. In the orange region, the escape and energy-gap assertions apply. Selection of the reference equilibrium within the stable region has the additional geometric hypotheses in \autoref{thm:rigid-balance}.}
\label{fig:nonlinear-classification}
\end{figure}

\FloatBarrier

\Needspace{9\baselineskip}
\subsection{Rigidity and equilibrium selection}\label{sec:rigidity-results}

Side-preserving diffeomorphisms give rigidity under the stable hypotheses. Stationary folds show what can occur in the larger class specified by the joint law. See \autoref{sec:rigidity-proof}.

\begin{theorem}\label{thm:rigid-balance}
Suppose the stable hypotheses of \autoref{thm:vertical-nonlinear-dichotomy} hold, namely $\Psi=\Psi(z)$ is monotone, $\rho_s'\Psi'\le0$ and $\beta_f=f+a_0\ge0$. Write $Z_s=(\rho_s,v_s+fx)$.
\begin{enumerate}[(1)]
\item On either closed half cell, let $S$ be an orientation-preserving, area-preserving $C^1$ diffeomorphism mapping each side of that half cell onto itself. Suppose the rearrangement
\[
 (\rho_s+\rho,v_s+v+fx)=Z_s\circ S
\]
obeys the joint force balance \eqref{eq_nonlinear_weak_balance} on that half cell, with $(\rho,v)$ in place of $(\rho_*,v_*)$. Then $\rho=v=0$ on the half cell. If $\beta_f>0$, one also has $S=\mathrm{id}$.
\item Let the initial pair of the global solution be $Z(0)=Z_s\circ S_0$ on each half cell, with $S_0$ as in part (1). Suppose there is a sequence $t_j\to\infty$ along which the flow maps are relatively compact in $C^1$ on each closed half cell. Then the full trajectory converges to the original equilibrium: $\rho,v\to0$ in every finite $L^q$, $\mathbf u\to0$ in $H^2$ and $p\to0$ in $H^1$. Moreover $E_{\rm mech}(t)\to E_s$.
\item Finite total strain, $\int_0^\infty\|\nabla\mathbf u(t)\|_{L^\infty}\,dt<\infty$, implies the compactness in part (2) and gives $\rho,v\to0$ in $C^1$ on each closed half cell. If $\beta_f>0$, the flow maps converge in $C^1$ to $S_0$. For every finite $q\ge1$, the scalar estimate \eqref{eq:strain-tail} holds with the limiting pair equal to $Z_s$.
\item If $\beta_f>0$ and $-\rho_s'\Psi'\ge c_0>0$, the selection conclusion of part (2) already follows if the maps $X_{t_j}$ and $X_{t_j}^{-1}$ are equicontinuous on each closed half cell along some sequence $t_j\to\infty$. In particular, it suffices that these maps be uniformly bi-Lipschitz along one such sequence.
\end{enumerate}
\end{theorem}

\begin{proposition}\label{prop:folded-equilibria}
Assume the stable hypotheses of \autoref{thm:rigid-balance}. If $\rho_s$ is nonconstant or $\beta_f>0$, there are nonzero stationary perturbations with zero planar velocity and with the reference joint distribution on each half cell, arbitrarily small in $L^\infty$ and in $W^{1,r}$ for every fixed finite $r$. If $\rho_s$ is constant and $\beta_f=0$, the reference equilibrium is the only stationary state with this joint distribution.
\end{proposition}

\Needspace{9\baselineskip}
\subsection{Quantitative relaxation and lower bounds}\label{sec:rates-results}

Geometric control of the labels yields upper bounds for the decay. Exact rearrangements yield a nonlinear lower bound and a separate obstruction at the stress-free wall. Kiselev, Park and Yao \cite{KiselevParkYao2025} constructed solutions with long-time growth of density gradients in viscous, non-diffusive Boussinesq flow. Here the obstruction is expressed through the wall derivatives of the shifted meridional velocity. See \autoref{sec:rates-proof}.

\begin{theorem}\label{thm:quantitative-selection}
Assume the hypotheses of \autoref{thm:rigid-balance}, with $\beta_f>0$ and
\[
 -\rho_s'(z)\Psi'(z)\ge c_0>0\qquad(0\le z\le h).
\]
Suppose $Z(0)=Z_s\circ S_0$ on each half cell, with $S_0$ as in that theorem. Write $S_t=S_0\circ X_t^{-1}$ for the transported labels and $d(t)=S_t-\mathrm{id}$ for their displacement. The $H^m$ norm of $d$ below is the norm of the product Sobolev space on the two half cells. Set $\mathcal E(t)=E_{\rm mech}(t)-E_s$.

There exists $\varepsilon_*>0$ with the following properties.
\begin{enumerate}[(1)]
\item If for some $T\ge0$,
\begin{equation}\label{eq:small-label-gradient}
 \sup_{t\ge T}\|Dd(t)\|_{L^\infty}\le\varepsilon_*,
\end{equation}
then
\begin{equation}\label{eq:quant-relaxation}
 \|\rho(t)\|_2+\|v(t)\|_2+\|\mathbf u(t)\|_{H^2}
   +\|\mathbf u_t(t)\|_2+\|p(t)\|_{H^1}
 \le C(1+t-T)^{-1/2},\qquad t\ge T.
\end{equation}
If $\sup_{t\ge T}\|d(t)\|_{H^m}<\infty$ for a real $m\ge1$, the exponent $1/2$ in \eqref{eq:quant-relaxation} can be replaced by $m/2$, and $\mathcal E(t)\le C(1+t-T)^{-m}$.
\item Let $m>2$ be real. On any family of such trajectories with
\[
 \sup_{t\ge0}\|d(t)\|_{H^m}\le M,\qquad \|\mathbf u_0\|_{H^2}\le R,
\]
the estimates of part (1) hold with $T=0$, exponent $m/2$ and a constant uniform over the family. The energy satisfies $\mathcal E(t)\le C(1+t)^{-m}$ and
\begin{equation}\label{eq:quant-strain-tail}
 \int_t^\infty\|\nabla\mathbf u(s)\|_{L^\infty}\,ds
       +\|X_t-S_0\|_{C^1}
 \le C(1+t)^{1-m/2},\qquad t\ge0 .
\end{equation}
Here $C$ depends only on the fixed background, $m$, $M$ and $R$, and the flow norm is taken on the closed half cells.
\item Let $0<\theta,\kappa\le1$. Suppose $S_t$ is uniformly $C^{0,\theta}$ and $S_t^{-1}$ is uniformly $C^{0,\kappa}$ on the closed half cells for $t\ge0$, with both H\"older seminorms bounded by $L$. Put
\begin{equation}\label{eq:holder-rate}
 \gamma_{\theta,\kappa}=\frac{\theta\kappa(1+\theta)}{2(2+\theta)},\qquad
 \mu_{\theta,\kappa}=\frac{\gamma_{\theta,\kappa}}{1-\gamma_{\theta,\kappa}} .
\end{equation}
Then
\begin{equation}\label{eq:lip-energy-decay}
 \mathcal E(t)\le C(1+t)^{-\mu_{\theta,\kappa}}.
\end{equation}
The sum on the left of \eqref{eq:quant-relaxation} is bounded by $C(1+t)^{-\mu_{\theta,\kappa}/2}$, and
\begin{equation}\label{eq:lip-uniform-decay}
 \|(\rho(t),v(t))\|_{L^\infty}+\|X_t-S_0\|_{C^0}
 \le C(1+t)^{-\mu_{\theta,\kappa}\theta/[2(1+\theta)]}.
\end{equation}
The constants are uniform for a common $L$ and $\|\mathbf u_0\|_{H^2}\le R$, with fixed exponents and background. In particular, $(\theta,\kappa)=(1,1)$ gives the exponents $1/2$, $1/4$ and $1/8$, respectively.
\end{enumerate}
\end{theorem}

\autoref{fig:decay-regimes} compares the small-gradient and Sobolev regimes with the H\"older estimate. \autoref{fig:holder-rates} plots the exponents for equal H\"older regularity of the labels and their inverses.

\begin{figure}[!htbp]
\centering
\begin{tikzpicture}[x=1cm,y=.95cm]
  \node[figPanel] at (1.85,2.62) {Label control};
  \node[figPanel] at (5.55,2.62) {Energy bound};
  \node[figPanel] at (9.16,2.62) {Scalar $L^2$ bound};
  \foreach \yy/\hyp/\en/\rate in {
    {1.89}/{Uniform H\"older}/{(1+t)^{-\mu_{\theta,\kappa}}}/{(1+t)^{-\mu_{\theta,\kappa}/2}},
    {1.03}/{Small $\|Dd\|_\infty$, $t\ge T$}/{(1+t-T)^{-1}}/{(1+t-T)^{-1/2}},
    {.17}/{Uniform $H^m$, $m>2$}/{(1+t)^{-m}}/{(1+t)^{-m/2}}}{
    \node[figBox,minimum width=3.48cm] (a) at (1.85,\yy) {\hyp};
    \node[figBox,minimum width=2.68cm] (b) at (5.55,\yy) {$\en$};
    \node[figBox,minimum width=3.0cm] (c) at (9.16,\yy) {$\rate$};
    \draw[figArrow,draw=figTeal] (a)--(b);
    \draw[figArrow,draw=figTeal] (b)--(c);
  }
\end{tikzpicture}
\caption{Three conditional upper bounds in \autoref{thm:quantitative-selection}, ordered by the type of control on the evolving labels. The small-gradient row uses only the $H^1$ bound supplied by the gradient hypothesis. The uniform $H^m$ row includes a finite total strain and the tail estimate \eqref{eq:quant-strain-tail}.}
\label{fig:decay-regimes}
\end{figure}

\begin{figure}[!htbp]
\centering
\begin{tikzpicture}[x=6.8cm,y=4.7cm,
 declare function={eng(\x)=\x*\x*(1+\x)/(2*(2+\x)-\x*\x*(1+\x));}]
  \draw[figArrow] (0,0)--(1.22,0) node[right] {$\theta=\kappa$};
  \draw[figArrow] (0,0)--(0,.59) node[above] {decay exponent};
  \draw[figInk!38,densely dashed] (1,0)--(1,.50);
  \foreach \y/\name in {.125/{1/8},.25/{1/4},.50/{1/2}}{
    \draw[figInk!25,dotted] (0,\y)--(1,\y);
    \node[left=2pt] at (0,\y) {$\name$};
  }
  \foreach \xx/\name in {0/0,.5/{1/2},1/1}{
    \draw (\xx,.008)--(\xx,-.008) node[below=3pt] {$\name$};
  }
  \draw[figBlue,line width=1pt,domain=0:1,samples=100] plot(\x,{eng(\x)});
  \draw[figTeal,line width=1pt,domain=0:1,samples=100] plot(\x,{eng(\x)/2});
  \draw[figOrange,densely dashed,line width=1pt,domain=0:1,samples=100]
     plot(\x,{eng(\x)*\x/(2*(1+\x))});
  \node[right=3pt,figBlue] at (1,.50) {energy};
  \node[right=3pt,figTeal] at (1,.25) {$L^2$ scalar norm};
  \node[right=3pt,figOrange] at (1,.125) {$L^\infty$ scalar norm};
\end{tikzpicture}
\caption{Decay exponents in the symmetric H\"older case $\theta=\kappa$ of \autoref{thm:quantitative-selection}(3). The curves plot the exponents proved for the energy and the scalar norms in \eqref{eq:lip-energy-decay}--\eqref{eq:lip-uniform-decay}. The right endpoint is the bi-Lipschitz case. The horizontal origin is the limit of the formulas as $\theta$ decreases to zero.}
\label{fig:holder-rates}
\end{figure}

\subsubsection{Slow relaxation and a wall obstruction}

\begin{theorem}\label{thm:nonlinear-slow-relaxation}
Assume the strictly stable setting of \autoref{thm:quantitative-selection}. Fix any real $m\ge1$. There exist $r_*>0$, $c>0$ and $t_*>0$ with the following property. For every $0<r\le r_*$ and $t\ge t_*$, there are smooth admissible initial data whose total scalar pair is $Z_s\circ S_0$ on each half cell, with $S_0$ a smooth area-preserving, side-preserving diffeomorphism, and
\[
 \|S_0-\mathrm{id}\|_{H^m}
 +\|\rho_0\|_{H^m}+\|v_0\|_{H^m}+\|\mathbf u_0\|_{H^m}\le r .
\]
The corresponding nonlinear solution satisfies
\begin{equation}\label{eq:nonlinear-slow-lower}
 \| (\rho(t),v(t))\|_2\ge c r(1+t)^{-m/2}.
\end{equation}
Consequently, the exponent $m/2$ is an upper bound for any uniform scalar decay exponent over this class of initial data. The upper estimate in \autoref{thm:quantitative-selection}(2) is uniform over trajectories with the stated Sobolev bound on the evolving labels.
\end{theorem}

\begin{proposition}\label{prop:wall-obstruction}
In the strictly stable setting of \autoref{thm:quantitative-selection}, take smooth initial labels and put $M=v_s+fx+v$. Then
\begin{equation}\label{eq:wall-invariant}
 \mathcal I(t)=\int_{\mathbb T}|M_z(t,x,h)|^{1/2}|M_x(t,x,h)|^{3/2}\,dx
\end{equation}
is conserved. If, in addition, the initial label pair is $Z_s\circ S_0$ on each half cell for a side-preserving, area-preserving $C^1$ diffeomorphism $S_0$ as in \autoref{thm:rigid-balance}, then $\mathcal I(0)>0$ implies infinite total strain. Under this rearrangement hypothesis no sequence of flow maps with times tending to infinity is relatively compact in $C^1$ on both closed half cells.

For arbitrary smooth initial labels with $\mathcal I(0)>0$, every $s>3/2$ satisfies the following bound at times when $v(t)\in H^s(\Omega)$:
\begin{equation}\label{eq:wall-tradeoff}
 \|v(t)\|_2^{\,1-3/(2s)}\|v(t)\|_{H^s}^{\,3/(2s)}\ge c_s>0.
\end{equation}
In particular, if $\|v(t)\|_2\to0$, the $H^s$ norm must grow at least as fast as $\|v(t)\|_2^{-(2s-3)/3}$. The estimate holds with exponent $3/(2s)$ without a loss in the trace order.

Arbitrarily small, smooth, exact joint rearrangements in every fixed $H^m$ can have $\mathcal I(0)>0$, even with zero planar initial velocity. The exponent in \eqref{eq:wall-tradeoff} is sharp among smooth exact joint rearrangements with a fixed positive wall invariant. More precisely, there exist such configurations supported in layers of thickness $\ell$ adjacent to the top wall, with
\begin{equation}\label{eq:wall-layer-scaling}
 \|v_\ell\|_2\asymp\ell^{3/2},\qquad
 \|v_\ell\|_{H^s}\asymp\ell^{3/2-s}\qquad (s>3/2).
\end{equation}
Finally, for data with $M_x(0,x,h)>0$ everywhere, the uniform label bound of \autoref{thm:quantitative-selection}(2) requires $v_{0,z}(x,h)=0$.
\end{proposition}

For a trajectory satisfying both the uniform H\"older hypotheses of \autoref{thm:quantitative-selection}(3) and $\mathcal I(0)>0$, the scalar $L^2$ upper bound and \eqref{eq:wall-tradeoff} give
\[
 \|v(t)\|_{H^s}\ge c_s(1+t)^{\mu_{\theta,\kappa}(2s-3)/6}\qquad(s>3/2)
\]
whenever that norm is finite.

\FloatBarrier

\section{Transport invariants and relative energy}\label{sec:transport}

The density and the shifted meridional velocity satisfy transport
equations on the representative strip. Conservation of these
quantities yields the force bounds used below.

\subsection{Transport estimates}

\begin{lemma}\label{lem_transport_lp}
Let $(\rho,v,\mathbf u,p)$ be the solution of \autoref{thm_nonlinear_regularity}. On each of the half cells $(-\pi,0)\times(0,h)$ and $(0,\pi)\times(0,h)$, the joint distribution of $\rho+\rho_s$ and $v+v_s+fx$ is independent of time. For every $1\le r\le\infty$ for which the initial norms are finite and all $t\ge0$,
\begin{equation}\label{eq_transport_lp}
\begin{aligned}
\|\rho(t)\|_{L^r(\Omega)}
&\le \|\rho_0+\rho_s\|_{L^r(\Omega)}+\|\rho_s\|_{L^r(\Omega)},\\
\|v(t)\|_{L^r(\Omega)}
&\le \|v_0+v_s+fx\|_{L^r(\Omega)}+\|v_s+fx\|_{L^r(\Omega)}.
\end{aligned}
\end{equation}
The norms of the affine terms are evaluated on the representative cell $(-\pi,\pi)\times(0,h)$.
\end{lemma}

\begin{proof}
By \eqref{eq_perturbation}, $D_t(\rho+\rho_s)=0$ and $D_t(v+v_s+fx)=0$, where $D_t=\partial_t+\mathbf u\cdot\nabla$: the background derivatives are $\nabla\rho_s=\delta\nabla\Psi$ and $(v_s+fx)_x=f+a_0$. Odd reflection gives $u=0$ on $x=0,\pm\pi$, and $w=0$ on the horizontal walls. Thus the divergence-free flow preserves both half cells and their measure. The common pullback preserves the joint law and each $L^r$ norm of the labels. Subtracting the background terms and using the triangle inequality gives \eqref{eq_transport_lp}.
\end{proof}

\subsection{Relative energy and dissipation}

The relative energy compares the density with an affine function of the potential and includes the three velocity components. We derive it here at the regularity of the constructed solutions. Define the mechanical energy by
\begin{equation}\label{eq:generic-mechanical-energy}
 E_{\rm mech}
 =\tfrac12\bigl(\|\mathbf u\|_2^2+\|v+v_s\|_2^2\bigr)
      +\int_\Omega(\rho+\rho_s)\Psi\,dx\,dz .
\end{equation}
By conservation of $\|\rho+\rho_s-\beta\|_2^2$ in \autoref{lem_transport_lp}, the difference $\mathcal E_{\gamma,\beta}/(2\gamma)-E_{\rm mech}$ is constant in time. The bottom friction in \eqref{bd-con-2} contributes to both dissipation identities.

\begin{lemma}\label{lem_relative_energy}
Let $(\rho,v,\mathbf u,p)$ be the solution of \autoref{thm_nonlinear_regularity}. For every fixed $\gamma>0$ and $\beta\in\mathbb R$ and every $t$ in the interval of existence,
\begin{equation}\label{eq_relative_energy}
 \mathcal E_{\gamma,\beta}(t)+2\gamma\nu\int_0^t\Diss(\mathbf u(s))\,ds
 =\mathcal E_{\gamma,\beta}(0),
\end{equation}
with the notation of \eqref{eq:profile-energy} and \eqref{eq:diss}.
In particular, the velocity satisfies
\begin{equation}\label{eq_velocity_energy_class}
\mathbf{u}\in L^\infty(0,\infty;L^2(\Omega))\cap L^2(0,\infty;H^1(\Omega)).
\end{equation}
\end{lemma}

\begin{proof}
Put $R=\rho+\rho_s$ and $V=v+v_s$. The characteristic formulas of \autoref{lem_transport_lp} give
\[
 D_t R=0,\qquad D_t V=-fu,\qquad
 D_t(R-r_{\gamma,\beta})=\gamma\mathbf u\cdot\nabla\Psi ,
 \qquad D_t=\partial_t+\mathbf u\cdot\nabla .
\]
These identities hold along almost every trajectory. Since the flow preserves measure and the right-hand sides are locally integrable in time with values in $L^2$, differentiation along the trajectories gives the $L^2$ energy identities for $V$ and $R-r_{\gamma,\beta}$ even when the initial scalars are discontinuous.

In the planar momentum equation, we use the total force $fV\eone-R\nabla\Psi$ and total pressure $p+p_s+p_0$. The pressure pairing vanishes by impermeability on the representative cell. The pairing of $r_{\gamma,\beta}\nabla\Psi$ with $\mathbf u$ also vanishes, since this force is the gradient of $-\gamma\Psi^2/2+\beta\Psi$. Testing with $\gamma\mathbf u$ thus cancels the two scalar energy pairings, leaving
$\frac12 \frac{d}{dt}\mathcal E_{\gamma,\beta}+\gamma\nu\Diss(\mathbf u)=0$.
Coercivity \eqref{eq_stokes_coercivity} controls the $H^1$ velocity norm. Integration in time yields both conclusions.
\end{proof}

\section{Linear instability and relaxation}\label{sec:linear-proof}
\subsection{The quadratic operator pencil}

\begin{proof}[Proof of \autoref{thm_linear_stability}]
Let $A$, $\mathcal V$ and $\mathbb P$ be as in \autoref{lem_stokes_estimates}. The Stokes form satisfies $a(\mathbf u,\mathbf u)=\Diss(\mathbf u)$ and controls the $H^1$ norm on $\mathcal V$.

Suppose that $(\mathbf u,\rho,v,p)e^{\lambda t}$ is a nonzero normal mode with $\Re\lambda>0$ and mean-zero pressure. Its scalar amplitudes satisfy
\begin{equation}
 \lambda\rho=-\delta\,\mathbf u\cdot\nabla\Psi,\qquad
 \lambda v=-(f+a_0)u .
 \label{eq_linear_mode_scalars}
\end{equation}
Consequently, $\mathbf u\ne0$. Applying $\mathbb P$ to the momentum equation and using \eqref{eq_linear_mode_scalars}, we obtain
\begin{equation}
  \lambda^2\mathbf u+\nu\lambda A\mathbf u
   =\mathbb P\bigl[\delta(\mathbf u\cdot\nabla\Psi)\nabla\Psi
                    -f(f+a_0)u\mathbf e_1\bigr].
  \label{eq_linear_mode_velocity}
\end{equation}
We use the complex $L^2$ pairing, linear in its first argument, to test with $\mathbf u$. The right-hand side is real, so
\[
 \Im\lambda\,
 \bigl(2\Re\lambda\norm{\mathbf u}_2^2+\nu\Diss(\mathbf u)\bigr)=0 .
\]
Thus $\lambda>0$ is real. A nonzero real or imaginary part of the eigenfunction gives a real eigenfunction. The resulting identity is
\begin{equation}
  E_2(\mathbf u)=\lambda^2\|\mathbf u\|_{L^2(\Omega)}^2
                          +\nu\lambda \Diss(\mathbf u)>0.
  \label{eq_linear_mode_identity}
\end{equation}

For $s>0$, define the quadratic form
\begin{equation}\label{eq:threshold-form}
 Q_s(\mathbf u)=s^2\|\mathbf u\|_2^2+\nu s\Diss(\mathbf u)-E_2(\mathbf u)
 \qquad(\mathbf u\in\mathcal V).
\end{equation}
For the projected linear evolution on the unweighted product $L^2$ space, put $D=\nu A$ and write the generator as
$\mathcal L_{\rm u}(\mathbf u,q)=(-D\mathbf u+Bq,C\mathbf u)$ on $D(A)$ times the scalar space, where
\[
 Bq=\mathbb P(fv\eone-\rho\nabla\Psi),\qquad
 C\mathbf u=(-\delta\,\mathbf u\cdot\nabla\Psi,-(f+a_0)u),\qquad q=(\rho,v).
\]
Both couplings are bounded. The operator $T=BC$ is self-adjoint on $\mathcal H$ and $(T\mathbf u,\mathbf u)_2=E_2(\mathbf u)$ for real velocities. For $\Re\lambda>0$, eliminating $q$ in the resolvent equation gives the velocity operator $\lambda^2+\lambda D-T$. After multiplication by $(\lambda+D)^{-1}/\lambda$ this is the identity minus a compact operator on $\mathcal H$. The Fredholm alternative then shows that each point of the spectrum with positive real part is an eigenvalue: a trivial kernel gives an inverse by velocity elimination. The compact family is analytic and invertible for large positive $\lambda$. The analytic Fredholm theorem therefore gives isolated eigenvalues of finite algebraic multiplicity in the open right half-plane.

The form in \eqref{eq:threshold-form} is closed and bounded from below on the solenoidal $L^2$ space $\mathcal H$. Since $\mathcal V$ embeds compactly into $\mathcal H$, its operator has eigenvalues $\eta_1(s)\le\eta_2(s)\le\cdots$ repeated according to multiplicity. For every $j$ the min--max principle gives, when $s>r>0$,
\[
 \eta_j(s)-\eta_j(r)\ge s^2-r^2.
\]
To prove continuity, write $Q_s=(s/r)Q_r+(s-r)(s\,\|\cdot\|_2^2+r^{-1}E_2)$ as forms. The second form is bounded in $L^2$ by $C|s-r|$ locally on $(0,\infty)$. The min--max comparison with $(s/r)Q_r$ proves continuity of $\eta_j$. Also $\eta_j(s)\ge s^2-C$ as $s\to\infty$.

For each $s>0$, \eqref{eq_linear_mode_velocity} shows that the velocity amplitude of a mode belongs to the kernel of the operator of $Q_s$. Conversely, a vector in this kernel belongs to $D(A)$, since its form equation expresses $\nu s A\mathbf u$ as a bounded $L^2$ force. Define the scalar amplitudes by \eqref{eq_linear_mode_scalars}. The unit-viscosity Stokes problem in \autoref{lem_stokes_estimates} with force $(fv\eone-\rho\nabla\Psi-s\mathbf u)/\nu$ recovers the velocity wall traces and the mean-zero pressure $p/\nu$. Multiplying the recovered pressure by $\nu$ gives the pressure of the original mode. Thus every vector in the kernel gives a mode. The eigenvalue is semisimple. Indeed, if
$(s-\mathcal L_{\rm u})(\mathbf u_1,q_1)=(\mathbf u_0,q_0)$ for a nonzero eigenvector $(\mathbf u_0,q_0)$ at $s$, elimination of $q_1$ and use of $sq_0=C\mathbf u_0$ give
\[
 (s^2+sD-T)\mathbf u_1= s\mathbf u_0+Bq_0=(2s+D)\mathbf u_0.
\]
The left-hand side has zero pairing with $\mathbf u_0$ by self-adjointness, and the right-hand side has strictly positive pairing. Hence the dimension of the kernel is also the algebraic multiplicity. Each negative $\eta_j(s)$ crosses zero exactly once above $s$ and each nonnegative $\eta_j(s)$ stays positive above $s$. The min--max argument proves \eqref{eq:mode-count}, including multiplicities. In particular, there are only finitely many growing modes above each positive threshold.

A single direction of positivity of $E_2$ gives a negative direction of $Q_s$ for sufficiently small $s$. Conversely every growing mode yields such a positive direction by \eqref{eq_linear_mode_identity}. The largest growth rate is the crossing of $\eta_1(s)$, which is \eqref{eq:max-growth}. More generally, if $E_2$ is positive on an $n$-dimensional subspace, then $Q_s$ is negative there for all sufficiently small $s$, by uniformity on its unit sphere.

\medskip
\noindent\textit{The pointwise criterion.}
The continuous, symmetric matrix in \eqref{eq:pointwise-drive} satisfies
$E_2(\mathbf u)=\int_\Omega \mathbf u^T\mathsf M\mathbf u\,dx\,dz$.
A trial field with positive quadratic form therefore requires a positive direction of $\mathsf M$ somewhere. Conversely, suppose a unit vector $d=(d_1,d_2)$ has $d^T\mathsf M d>0$ at a point. This holds uniformly on a sufficiently small open ball, which we may take away from $x=0,\pm\pi$ and the walls. Choose a nonzero cutoff $\chi$ supported in that ball. With $\xi=(d_2,-d_1)$ and $X=(x,z)$, set
\[
 \psi_k(X)=k^{-1}\chi(X)\sin(k\,\xi\cdot X),\qquad
 \mathbf u_k=(-\partial_z\psi_k,\partial_x\psi_k).
\]
Extend the stream function by odd reflection to the reflected ball and extend it by zero outside the two balls. The resulting velocity belongs to $\mathcal V$. On the first ball it equals $d\chi\cos(k\,\xi\cdot X)$ up to a term of size $O(k^{-1})$ in $L^2$. On the reflected ball the velocity has the required parity. With $J=\operatorname{diag}(-1,1)$, the evenness of $\rho_s$ and $\Psi$ gives $\mathsf M(-x,z)=J\mathsf M(x,z)J$. Thus the drive agrees on the two balls, and the Riemann--Lebesgue lemma gives
\[
 E_2(\mathbf u_k)\longrightarrow
     \int_{\operatorname{supp}\chi}\chi^2\,d^T\mathsf M d\,dx\,dz>0 .
\]

For each $n\ge1$, choose $n$ disjoint balls in the same open region and on the same side of $x=0$, together with their reflections. On each pair take one sufficiently oscillatory field of positive drive. The cross terms of $E_2$ vanish by disjoint support, so $E_2$ is positive on an $n$-dimensional subspace. Applying \eqref{eq:mode-count} at sufficiently small $s>0$ and letting $n$ vary gives infinitely many growing modes. Each positive rate has finite multiplicity and the count is finite above every positive threshold, so infinitely many distinct rates tend to zero.

If $\Psi$ is nonconstant and harmonic, $\Psi_z$ cannot vanish in any open ball. Otherwise harmonic unique continuation would give $\Psi_z\equiv0$, and harmonicity and periodicity in $x$ would force $\Psi$ to be constant. In an open region where $\delta>0$, choose a point where $\Psi_z\ne0$ and test $\mathsf M$ with $d=(0,1)$. This gives a positive direction for every $a_0$. If $f+a_0<0$, at any interior height there is a horizontal critical point of $\Psi(\cdot,z)$ by periodicity. At such a point $d=(1,0)$ gives $d^T\mathsf M d=-f(f+a_0)>0$.

Finally, assume $\delta\le\delta_0<0$ and $f+a_0>0$. Testing the three evolution equations in \eqref{eq_linear_perturbation} with $\rho/(-\delta)$, $fv/(f+a_0)$ and $\mathbf u$, respectively, gives
\begin{equation}
  \frac12\frac{d}{dt}\left[\|\mathbf u\|_{L^2(\Omega)}^2
       +\int_\Omega\left(\frac{\rho^2}{-\delta}
                       +\frac{f}{f+a_0}v^2\right)\,dx\,dz\right]
       +\nu \Diss(\mathbf u)=0.
  \label{eq_linear_stable_energy}
\end{equation}
The pressure term vanishes by incompressibility and the zero normal trace. The energy in brackets is equivalent to the square of the product $L^2$ norm. For $L^2$ data, the projected evolution is a bounded perturbation of the product semigroup generated by $(-\nu A,0,0)$. Approximation by data in the operator domain extends the energy bound to the corresponding $L^2$ solutions. This proves linear stability.
\end{proof}

\FloatBarrier

\subsection{Strong relaxation, analyticity and sharp decay}

\begin{proof}[Proof of \autoref{thm_linear_large_time}]
Let $A$, $\mathcal H$ and $\mathbb P$ be as in \autoref{lem_stokes_estimates}. The space $\mathscr X$ is equipped with the product of the $L^2$ velocity pairing and the weighted scalar pairing in \eqref{eq_linear_balance_projection_weight}.
The assumptions on $\delta$, $f$ and $f+a_0$ make this norm equivalent to the product $L^2$ norm. The projected equation has generator
\begin{equation}
  \mathcal L(\mathbf u,\rho,v)
  =\bigl(-\nu A\mathbf u+\mathbb P(fv\mathbf e_1-\rho\nabla\Psi),\
           -\delta\,\mathbf u\cdot\nabla\Psi,\ -(f+a_0)u\bigr),
  \label{eq_linear_generator}
\end{equation}
with domain given by $\mathbf u\in D(A)$ and the scalars in the corresponding $L^2(\Omega)$ spaces. In particular $\delta\,\mathbf u\cdot\nabla\Psi=\mathbf u\cdot\nabla\rho_s$ is even in $x$. In the inner product of \eqref{eq_linear_balance_projection_weight}, the coupling terms in \eqref{eq_linear_generator} define a bounded skew-adjoint operator. The bounded perturbation theorem applied to $(-\nu A,0,0)$ gives a strongly continuous semigroup. For initial data in $D(\mathcal L)$,
\begin{equation}
  \frac12\frac{d}{dt}\|(\mathbf u,\rho,v)\|^2
    +\nu\|A^{1/2}\mathbf u\|_{L^2(\Omega)}^2=0.
  \label{eq_linear_evolution_dissipation}
\end{equation}
By density, the semigroup $e^{t\mathcal L}$ is contractive. Integrating \eqref{eq_linear_evolution_dissipation} and using \eqref{eq_stokes_domains} gives
\begin{equation}
  \sup_{t\ge0}\|(\mathbf u(t),\rho(t),v(t))\|^2
       +\int_0^\infty\|\mathbf u(t)\|_{H^1(\Omega)}^2\,dt
  \le C\|(\mathbf u_0,\rho_0,v_0)\|^2.
  \label{eq_linear_global_energy_bound}
\end{equation}

\medskip
\noindent\textit{Analyticity and uniform upper bounds.}
Write $D=\nu A$, so that $D\ge d_0 I$ for some $d_0>0$, and write $q=(\rho,v)$ and $Bq=\mathbb P(fv\eone-\rho\nabla\Psi)$. In the weighted product space the generator is
\[
 \mathcal L(\mathbf u,q)=(-D\mathbf u+Bq,-B^*\mathbf u).
\]
For $z=a+i\omega$, $a\ge0$, $z\ne0$, we solve $(z-\mathcal L)(\mathbf u,q)=(F,G)$ by eliminating the velocity. With $R_z=(z+D)^{-1}$,
\begin{equation}\label{eq:schur}
 Q_zq=G-B^*R_z F,\qquad
 \mathbf u=R_z(F+Bq),\qquad Q_z=zI+B^*R_zB .
\end{equation}
If $J_z=((D+a)^2+\omega^2)^{-1}$, the real and imaginary parts of the scalar quadratic pairing for $Q_z$ are
\[
\begin{aligned}
 \Re(Q_zq,q)_{\mathrm{sc}}
   &=a\|q\|_{\mathrm{sc}}^2+((D+a)J_zBq,Bq)_2,\\
 \Im(Q_zq,q)_{\mathrm{sc}}
   &=\omega\bigl(\|q\|_{\mathrm{sc}}^2-(J_zBq,Bq)_2\bigr).
\end{aligned}
\]
The estimate $D+a\ge d_0I$ bounds $(J_zBq,Bq)_2$ by $d_0^{-1}\Re(Q_zq,q)_{\mathrm{sc}}$. For $z$ in any bounded part of the closed right half-plane, this gives
\[
 |z|\|q\|_{\mathrm{sc}}^2\le C\bigl|(Q_zq,q)_{\mathrm{sc}}\bigr|.
\]
The same estimate holds for $Q_z^*=Q_{\bar z}$. Thus $Q_z$ has closed and dense range and $\|Q_z^{-1}\|\le C/|z|$. The bound $\|R_z\|\le d_0^{-1}$ and \eqref{eq:schur} give $\|(z-\mathcal L)^{-1}\|\le C/|z|$ on bounded parts of that half-plane. For large $|z|$ the same bound follows by a Neumann series for the bounded coupling, since the resolvent of $(-D,0)$ has norm at most $1/|z|$ there. A second use of the resolvent identity extends the bound to a sector $|\arg z|\le\pi/2+\vartheta$ for some $\vartheta>0$: the disk of radius $c|\omega|$ about each $i\omega$, $\omega\ne0$, lies in the resolvent set with the same order of bound. The sectorial criterion \cite[Theorem~II.4.6]{EngelNagel2000} gives a bounded analytic semigroup and $\|\mathcal L S(t)\|\le C/t$.

For $X=S(t)X_0\in D(\mathcal L)$, $t>0$, the weighted energy pairing gives
\[
 c\|\mathbf u(t)\|_{H^1}^2
 \le \nu\Diss(\mathbf u(t))
 =-\Re(\mathcal L X,X)_{\mathscr X}
 \le \|\mathcal L X\|_{\mathscr X}\,\|X\|_{\mathscr X}.
\]
Contractivity and the analytic semigroup estimate yield both bounds in \eqref{eq:linear-uniform-rates}.

\medskip
\noindent\textit{Strong convergence and regularity.}
In the weighted space, $\mathcal L^*$ has the same domain as $\mathcal L$ and the opposite coupling terms. Dissipativity applied to either kernel gives $\mathbf u=0$, so
\begin{equation}\label{eq_linear_generator_kernel}
 \ker\mathcal L=\ker\mathcal L^*
  =\{(0,q):Bq=0\}.
\end{equation}
By commutation, $\|S(t)\mathcal L X\|_{\mathscr X}\le Ct^{-1}\|X\|_{\mathscr X}$ for $X\in D(\mathcal L)$. Contractivity gives strong decay on $\overline{\operatorname{Ran}\mathcal L}=(\ker\mathcal L^*)^\perp$, and $S(t)$ is the identity on $\ker\mathcal L$. The Stokes pressure reconstruction identifies $\ker B$ with the balanced pairs in \eqref{eq_linear_balance_subspace}. The strong limit of $S(t)$ is therefore the weighted orthogonal projection in \eqref{eq_linear_balance_projection_weight}.

For initial data in $D(\mathcal L)$, the time derivative of the solution is $S(t)\mathcal L X_0$. Its velocity component satisfies
\begin{equation}
  \begin{aligned}
  \mathbf u_{tt}+\nu A\mathbf u_t
      &=\mathbb P\bigl[\delta(\mathbf u\cdot\nabla\Psi)\nabla\Psi
                      -f(f+a_0)u\mathbf e_1\bigr],\\
  \mathbf u_t(0)&=-\nu A\mathbf u_0
                  +\mathbb P(fv_0\mathbf e_1-\rho_0\nabla\Psi).
  \end{aligned}
  \label{eq_linear_differentiated_velocity}
\end{equation}
The force in \eqref{eq_linear_differentiated_velocity} has $L^2$ norm bounded by $C\|\mathbf u\|_2$. Testing with $\mathbf u_t$ and $A\mathbf u_t$, and absorbing the force by coercivity and Young's inequality, gives respectively
\begin{equation}\label{eq_linear_acceleration_energy}
 \frac{d}{dt}\|\mathbf u_t\|_2^2+\nu\|A^{1/2}\mathbf u_t\|_2^2
 \le C\|\mathbf u\|_2^2,
\end{equation}
\begin{equation}\label{eq_linear_acceleration_form_energy}
 \frac{d}{dt}\|A^{1/2}\mathbf u_t\|_2^2+\nu\|A\mathbf u_t\|_2^2
 \le C\|\mathbf u\|_2^2 .
\end{equation}
Both tests can be justified by spectral truncation. The first estimate and \eqref{eq_linear_global_energy_bound} give $\mathbf u_t\in C([0,\infty);L^2)\cap L^2(0,\infty;\mathcal V)$, with a uniform $L^2$ bound. For each $\tau>0$ the second estimate starts at some $s\in(0,\tau)$ with $\mathbf u_t(s)\in\mathcal V$. Applying the same estimates to differences of truncations gives continuity in $\mathcal V$ and square integrability in $D(A)$ on $[s,\infty)$. If $\mathbf u_t(0)\in\mathcal V$, take $s=0$. The identifications in \autoref{lem_stokes_estimates} give \eqref{eq_linear_positive_time_regularity}.

We regard the momentum equation as a stationary Stokes problem with force $fv\mathbf e_1-\rho\nabla\Psi-\mathbf u_t$ and viscosity $\nu$. \autoref{lem_stokes_estimates} gives the mean-zero pressure and
\begin{equation}
  \|\mathbf u\|_{H^2(\Omega)}+\|p\|_{H^1(\Omega)}
     \le C\bigl(\|v\|_{L^2(\Omega)}+\|\rho\|_{L^2(\Omega)}
                         +\|\mathbf u_t\|_{L^2(\Omega)}\bigr).
  \label{eq_linear_uniform_stokes}
\end{equation}
The right-hand side is bounded uniformly in time by
\eqref{eq_linear_global_energy_bound} and \eqref{eq_linear_acceleration_energy}.
Together with $\mathbf u\in L^2(0,\infty;H^1)$, the uniform $H^1$ bound gives $\mathbf u\in L^r(0,\infty;H^1)$ for every $2<r<\infty$. This proves \eqref{eq_linear_global_regularity}. Choose $t_j\to\infty$ with $\|A^{1/2}\mathbf u_t(t_j)\|_2\to0$ by its square integrability. Integrating \eqref{eq_linear_acceleration_form_energy} gives
\begin{equation}\label{eq_linear_acceleration_decay_bound}
 \sup_{t\ge t_j}\|A^{1/2}\mathbf u_t(t)\|_2^2
 \le \|A^{1/2}\mathbf u_t(t_j)\|_2^2
     +C\int_{t_j}^\infty\|\mathbf u(r)\|_2^2\,dr\longrightarrow0 .
\end{equation}

The convergence of $(\rho,v)$ and \eqref{eq_linear_generator_kernel} imply $\mathbb P(fv(t)\mathbf e_1-\rho(t)\nabla\Psi)\to0$ in $L^2(\Omega)$. By \eqref{eq_linear_generator},
\begin{equation}
  \nu A\mathbf u(t)=-\mathbf u_t(t)
              +\mathbb P(fv(t)\mathbf e_1-\rho(t)\nabla\Psi)
               \longrightarrow0
   \quad\text{in }L^2(\Omega).
  \label{eq_linear_velocity_operator_decay}
\end{equation}
The $D(A)$ estimate gives $\mathbf u(t)\to0$ in $H^2(\Omega)$. Finally,
\begin{equation}
  \begin{aligned}
    fv(t)\mathbf e_1-\nabla p(t)-\rho(t)\nabla\Psi
       &=\mathbf u_t(t)-\nu\Delta\mathbf u(t)\longrightarrow0
                        &&\text{in }L^2(\Omega),\\
    v_t(t)&=-(f+a_0)u(t)\longrightarrow0
                        &&\text{in }H^2(\Omega).
  \end{aligned}
  \label{eq_linear_force_transverse_decay}
\end{equation}
The second component of the stationary Stokes estimate, applied to the difference from the limiting equilibrium, gives convergence of the mean-zero pressure in $H^1$ because the forcing converges in $L^2$ and $\mathbf u_t\to0$ in $L^2$.

\medskip
\noindent\textit{Eigenvalues tending to zero and sharpness.}
The coupling has infinite rank. Indeed, the second component of $B^*\mathbf u$ is $(f+a_0)u$. Evaluate $B^*$ on velocities $(-\partial_z\psi_k,\partial_x\psi_k)$ with $\psi_k=\chi(z)\sin(kx)$, $k\ge1$, and nonzero $\chi\in C_c^\infty(0,h)$. The resulting second components are linearly independent. For $0\le a<d_0/2$ set
\[
 K(a)=B^*(D-aI)^{-1}B .
\]
This is a positive compact operator on the scalar space, by the Stokes estimate and Rellich compactness. Its kernel is $\ker B$, and it has infinitely many positive eigenvalues $\kappa_j(a)$, listed with multiplicity in decreasing order. They are continuous in $a$ by norm continuity and min--max. The form comparison
\[
 K(0)\le K(a)\le (1-a/d_0)^{-1}K(0)
\]
follows from the spectral theorem for $D$. Thus, for every sufficiently large $j$, the continuous function $\kappa_j(a)-a$ is nonnegative at $a=\kappa_j(0)$ and negative at $a=2\kappa_j(0)<d_0/2$. Let $a_j$ be a root between these endpoints and let $q_j$ be a corresponding eigenvector of $K(a_j)$. For $0<a_1<a_2<d_0/2$, the spectral theorem for $D$ gives
\[
 \frac{K(a_2)}{a_2}\ \le\
 \max_{d\in[d_0,\infty)}
 \frac{a_1(d-a_1)}{a_2(d-a_2)}\,
 \frac{K(a_1)}{a_1}.
\]
The maximum is strictly less than one, because $a(d-a)$ increases for $0<a<d_0/2$ and the displayed ratio is largest at $d=d_0$. The positive min--max branches of $K(a)/a$ therefore cross one strictly downwards. If several branches give the same rate, they account for the dimension of the eigenspace of $K(a)$ at $a$. Then
\[
 \mathbf u_j=(D-a_jI)^{-1}Bq_j,\qquad
 \mathcal L(\mathbf u_j,q_j)=-a_j(\mathbf u_j,q_j).
\]
These eigenvectors lie in $\mathscr E^\perp$ because $q_j\perp\ker B=\{q:(0,q)\in\mathscr E\}$. The numbers $a_j$ are positive and tend to zero. Since $a_j=\kappa_j(a_j)$, the form comparison above also gives
\[
 1\le \frac{a_j}{\kappa_j(0)}\le(1-a_j/d_0)^{-1}\longrightarrow1 .
\]

Normalize $X_j=(\mathbf u_j,q_j)$ to unit energy norm. For fixed $t\ge0$, $\|(S(t)-\Pi_{\mathscr E})X_j\|_{\mathscr X}=e^{-a_jt}\to1$, giving \eqref{eq:nonuniform} by contractivity. Finally $\nu\Diss(\mathbf u_j)=a_j$ by the energy identity. The form norm is equivalent to $H^1$ on $\mathcal V$. At $t_j=a_j^{-1}$ this yields
\[
 \|\mathbf u(t_j)\|_{H^1}\ge c t_j^{-1/2},\qquad
 \|\partial_t S(t_j)X_j\|_{\mathscr X}=e^{-1}t_j^{-1}.
\]
This proves the optimality of both exponents in \eqref{eq:linear-uniform-rates} for uniform estimates.
\end{proof}

\subsection{Spectral counting in a vertically stratified strip}\label{sec:spectral-proof}

\begin{proof}[Proof of \autoref{thm:vertical-weyl}]
The linearized operator and its projected velocity pencils preserve the horizontal Fourier modes. In the stable weighted space we may choose $\delta$ to be a function of $z$: where $\Psi'=0$ its value can be any admissible negative constant, since the evolution uses the product $\delta\Psi'=\rho_s'$. This gives an equivalent weighted norm. For $\lambda\ne0$, the scalar amplitudes are determined by the velocity. The reflection condition gives one real phase at each integer $k\ge1$. The velocity at $k=0$ vanishes by parity, incompressibility and impermeability, so there is no contribution at $k=0$ to either count.

\medskip
\noindent\textit{Exact count for real negative eigenvalues near zero.}
In the weighted space of \autoref{thm_linear_large_time}, write $D=\nu A$ and use the block generator $(-D\mathbf u+Bq,-B^*\mathbf u)$, with $q=(\rho,v)$. Eliminating $q$ at $\lambda\ne0$ gives
\begin{equation}\label{eq:stable-count-pencil}
 P(\lambda)\mathbf u=0,\qquad
 P(\lambda)=\lambda^2 I+\lambda D+BB^* .
\end{equation}
For $\Re\lambda>-d_0/2$ and $\lambda\ne0$, the operator $\lambda+D$ is invertible and preconditioning the pencil by $\lambda^{-1}(\lambda+D)^{-1}$ gives the identity plus a compact operator. The analytic Fredholm theorem applies on this punctured half-plane, since the resolvent exists at real $\lambda>0$. The imaginary part of the pencil pairing on an eigenvector is
\[
 \Im\lambda\,\bigl(2\Re\lambda\|\mathbf u\|_2^2+(D\mathbf u,\mathbf u)_2\bigr)=0 .
\]
Thus every eigenvalue in the stated negative strip is real. Each such eigenvalue $\lambda$ is semisimple. Indeed, write the predecessor relation as $(\mathcal L-\lambda)(\mathbf u_1,q_1)=(\mathbf u_0,q_0)$. The scalar equation gives $\lambda q_1=-B^*\mathbf u_1-q_0$, and the eigenvector equation gives $Bq_0=(D+\lambda)\mathbf u_0$. Substitution yields
\[
 P(\lambda)\mathbf u_1=-(2\lambda+D)\mathbf u_0 .
\]
The left-hand side has zero pairing with $\mathbf u_0$ by self-adjointness. The right-hand side has strictly negative pairing since $D\ge d_0I$, $\lambda>-d_0/2$, and $\mathbf u_0\ne0$.

Fix a horizontal mode and let the positive index of $P(-a)$ mean the maximal dimension of a subspace on which its form is strictly positive. The form on $\mathcal V$ has compact resolvent and is bounded above for each $a>0$. For $0<a_1<a_2<d_0/2$,
\[
 P(-a_2)-P(-a_1)
 \le -(a_2-a_1)(d_0-a_1-a_2)I
\]
as forms. The variational eigenvalues in decreasing order are continuous and strictly decreasing on this interval. Continuity follows from
\[
 P(-a_2)=\frac{a_2}{a_1}P(-a_1)
                +(a_2-a_1)(a_2 I-a_1^{-1}BB^*) ,
\]
using the bounded form comparison in the min--max principle on compact subintervals of $(0,d_0/2)$. Thus the number of roots with $s<a<a_{\rm c}$ is the positive index of $P(-s)$ minus the positive index of $P(-a_{\rm c})$. There is no kernel at the upper endpoint by the choice of $a_{\rm c}$. The subtraction at $a_{\rm c}$ is bounded by a fixed finite total over all horizontal modes, because the full form of $P(-a_{\rm c})$ has compact resolvent and is bounded above. It is therefore negligible after multiplication by $s$, uniformly over the horizontal cutoffs.

\medskip
\noindent\textit{Reduction of the two counts to one-dimensional forms.}
In mode $k\ge1$ write
\[
 \mathbf u=(-\phi'(z)\sin(kx),\,k\phi(z)\cos(kx)),\qquad
 \phi\in H^2(0,h)\cap H^1_0(0,h).
\]
Incompressibility and the normal wall conditions identify this space with the velocity form domain for the mode. After dropping the common factor $\pi$, the kinetic mass is
$k^2\|\phi\|_2^2+\|\phi'\|_2^2$, and the dissipation is
\[
 \|\phi''\|_2^2+2k^2\|\phi'\|_2^2+k^4\|\phi\|_2^2+\alpha|\phi'(0)|^2 .
\]
For the growing count, \eqref{eq:mode-count} gives the positive index of $E_2-sD-s^2I$ in this mode. For the damped count use \eqref{eq:stable-count-pencil} and $(BB^*\mathbf u,\mathbf u)_2=-E_2(\mathbf u)$, which follows from the weighted scalar inner product. In either case the positive index to be counted has the form
\begin{equation}\label{eq:mode-unified-form}
\begin{aligned}
 k^2\int_0^h C(z)|\phi|^2\,dz+R\|\phi'\|_2^2
 &-\nu s\bigl(\|\phi''\|_2^2+2k^2\|\phi'\|_2^2
                  +k^4\|\phi\|_2^2+\alpha|\phi'(0)|^2\bigr)\\
 &+\tau s^2\bigl(k^2\|\phi\|_2^2+\|\phi'\|_2^2\bigr),
\end{aligned}
\end{equation}
with $(C,R,\tau)=(c,-b,-1)$ for growing modes and $(C,R,\tau)=(-c,b,1)$ for damped modes. The positive index vanishes for $k>C_0/\sqrt{s}$, with $C_0$ fixed for $0<s\le1$. Indeed, for such $k$ the coefficients of both $\|\phi\|_2^2$ and $\|\phi'\|_2^2$ after bounding $C$ from above are nonpositive. The same cutoff applies to the local forms below.

Partition $(0,h)$ into $m$ intervals $I_j$ of equal length $\ell=h/m$ and set $C_j^-=\inf_{I_j}C$, $C_j^+=\sup_{I_j}C$. To bound the positive index from below, use clamped $H^2_0(I_j)$ functions extended by zero and coefficients $C_j^-$. These extensions give admissible fields and their wall friction term vanishes. For an upper bound, enlarge the space to the direct sum of the free $H^2(I_j)$ spaces, discard the friction term, and use $C_j^+$. On each interval the comparison space $H^2(I_j)\cap H^1_0(I_j)$ has codimension two in the free space, and its clamped subspace has codimension two within it. The positive index changes by at most the corresponding codimension upon restriction. Hence it suffices to sum the Dirichlet sine counts for $C_j^\pm$, at a total error $O(ms^{-1/2})$ before multiplication by $s$.

\medskip
\noindent\textit{The lattice limit.}

On each comparison interval of length $\ell$, the sine basis diagonalizes the constant-coefficient form, with vertical wave numbers $n\pi/\ell$, $n\ge1$. Set $\xi=\sqrt{s}\,k$ and $\eta=\sqrt{s}\,n\pi/\ell$. The positive-index inequality becomes
\[
 \nu(\xi^2+\eta^2)^2
 < C_j^\pm\xi^2+R\eta^2+\tau s^2(\xi^2+\eta^2),\qquad \xi\le K .
\]
The shaded regions in \autoref{fig:weyl} cut this inequality at $\xi=K$. Each lattice point accounts for a cell of area $\pi s/\ell$. For $C_j^\pm$ and $R$ in bounded sets, the corresponding planar regions have bounded diameter. In polar coordinates the outer radius is the square root of
\[
 \nu^{-1}\bigl(C_j^\pm\cos^2\vartheta+R\sin^2\vartheta+\tau s^2\bigr)_+.
\]
This is a bounded monotone function of $\vartheta\in[0,\pi/2]$. Its polar arc has uniformly bounded length by its total radial variation plus the angular length times the largest radius. The coordinate axes and the cutoff line $\xi=K$ contribute only bounded line segments inside the bounding disk. The area error between the region and the union of lattice cells counted is bounded by the area of a boundary neighborhood of width $O(\sqrt{s}+\sqrt{s}/\ell)$, uniformly in $K$. Covering each boundary arc by disks along its arclength gives a bound of that order when the mesh is small. Replacing the radius squared by its value at $\tau s^2=0$ costs $O(s^2)$ in area, uniformly in $K$. Consequently, $s$ times the sine count on the interval differs from $\ell\mathcal A_K(C_j^\pm,R)/\pi$ by at most $C(\ell\sqrt{s}+\sqrt{s}+\ell s^2)$.

\begin{figure}[!htbp]
\centering
\begin{tikzpicture}[x=2.32cm,y=2.32cm]
  \foreach \R/\shift/\title/\col in {-.7/0/{(a) Growing modes}/figOrange,.7/2.30/{(b) Slowly damped modes}/figTeal}{
    \begin{scope}[shift={(\shift,0)}]
      \node[figPanel] at (.86,1.66) {\title};
      \begin{scope}
        \clip (0,0) rectangle (.94,1.39);
        \fill[\col!17] (0,0)-- plot[variable=\q,domain=0:90,samples=101]
          ({cos(\q)*sqrt(max(0,1.7*cos(\q)^2+\R*sin(\q)^2))},
           {sin(\q)*sqrt(max(0,1.7*cos(\q)^2+\R*sin(\q)^2))}) --cycle;
      \end{scope}
      \draw[\col,line width=.9pt] (1.30384,0) --
        plot[variable=\q,domain=0:90,samples=101]
          ({cos(\q)*sqrt(max(0,1.7*cos(\q)^2+\R*sin(\q)^2))},
           {sin(\q)*sqrt(max(0,1.7*cos(\q)^2+\R*sin(\q)^2))});
      \foreach \k in {1,...,8}{
        \foreach \n in {1,...,8}{
          \pgfmathsetmacro{\xx}{.115*\k}
          \pgfmathsetmacro{\yy}{.115*\n}
          \pgfmathtruncatemacro{\inregion}{(\xx^2+\yy^2)^2 < 1.7*\xx^2+\R*\yy^2}
          \ifnum\inregion=1 \fill[\col!85!figInk] (\xx,\yy) circle[radius=.019]; \fi
        }
      }
      \draw[figArrow] (0,0)--(1.55,0) node[right] {$\xi$};
      \draw[figArrow] (0,0)--(0,1.31) node[left] {$\eta$};
      \draw[figInk,dashed] (.94,0)--(.94,1.23);
      \node[below] at (.94,0) {$K$};
    \end{scope}
  }
\end{tikzpicture}
\caption{Rescaled mode regions and the horizontal cutoff in \autoref{thm:vertical-weyl}. The curves are the exact polar boundaries for the illustrative coefficients $\nu=1$, $C=1.7$, $R=-0.7$ and $R=0.7$, respectively, in the criterion $\nu(\xi^2+\eta^2)^2<C\xi^2+R\eta^2$. The shaded parts are counted when $\xi\le K$; the marked sample lattice points represent the bracketing count.}
\label{fig:weyl}
\end{figure}

In this estimate, intersecting with $\xi\le K$ replaces the radius squared at angle $\vartheta$ by its minimum with $K^2/\cos^2\vartheta$. In particular $\mathcal A_K(p,R)$ is Lipschitz in $p$ uniformly in $K$, by the polar area formula. The sums for the infimum and supremum profiles therefore differ from $\pi^{-1}\int_0^h\mathcal A_K(C(z),R)\,dz$ by at most $C\omega_C(h/m)$, with $\omega_C$ the modulus of continuity. Including the codimension errors, the total scaled error is at most $C(\omega_C(h/m)+m\sqrt{s}+s^2)$. Taking $m=\lceil s^{-1/4}\rceil$ proves both uniform limits and the Lipschitz remainder. For constant $C$, taking $m=1$ gives the stated better remainder.

Finally, at $K=\infty$ the area is one half of the integral of the radius squared over the first quadrant. In the stable case $-c\ge0$ and $b>0$, so integrating $\cos^2\vartheta$ and $\sin^2\vartheta$ gives \eqref{eq:weyl-total-negative}. In the growing case the same polar formula gives \eqref{eq:weyl-total-positive}. The asymptotics for the ordered eigenvalues follow by inversion of the total counting functions.

\medskip
\noindent\textit{The joint law of rates and horizontal wave numbers.}
Let $\mu_{s,\sigma}$ be the probability measure of $(r/s,k\sqrt{s})$ for the sampling in the theorem, and fix $u\ge1$ and $K\ge0$. The cutoff $k\sqrt{s}\le K$ is $k\sqrt{su}\le K\sqrt u$ at threshold $su$. Consequently
\[
 \mu_{s,\sigma}((u,\infty)\times[0,K])
  =\frac{N_\sigma(su,K\sqrt u)}{N_\sigma(s,\infty)}.
\]
The Weyl limits identify the limit of this expression with $1/(\pi C_\sigma)$ times the volume of
\[
 \{(z,\xi,\eta):\ 0<z<h,\ \xi,\eta>0,\ \xi\le K,               \xi^2+\eta^2<A_\sigma(z,\vartheta)/u\},
 \qquad \vartheta=\arg(\xi+i\eta).
\]
Indeed, scaling $\xi,\eta$ by $\sqrt u$ multiplies area by $u$. The limiting measure is therefore the image of the normalized phase volume $dz\,d\xi\,d\eta/(\pi C_\sigma)$ on the set with $u=1$ and $K=\infty$ under
\[
 (z,\xi,\eta)\longmapsto \bigl(A_\sigma(z,\vartheta)/(\xi^2+\eta^2),\ \xi\bigr).
\]
To obtain the quantitative estimate, set $u=1/v$ in the counting ratio. If $w=s/v$ lies below a fixed small threshold, multiplication of the numerator by $s=v w$ gives the coefficient $v C_\sigma F_\sigma(K/\sqrt v)$ with error $O(v w^\zeta)$. The denominator after multiplication by $s$ is $C_\sigma+O(s^\zeta)$ and is bounded below by $C_\sigma/2$ for sufficiently small $s$. Since $v w^\zeta=s^\zeta v^{1-\zeta}\le s^\zeta$, the estimate is uniform in $v,K$. If $w$ is above that fixed threshold, the numerator counts at most the finite number of rates above the fixed threshold, even without a horizontal cutoff, while $v=O(s)$. Both the ratio and the proposed limit are then $O(s)$. This proves \eqref{eq:weyl-rectangle-error}, including the part of the rectangle family whose rate threshold does not tend to zero.

For fixed $z,\vartheta$ the inverse image of a fixed rate consists of at most one radius, while a fixed horizontal coordinate cuts out a set of planar area zero. Thus the rectangles are continuity sets. Tightness follows from the finite horizontal cutoff in the proof and the tail limit $1/u$. The rectangle limits then imply weak convergence.

In polar coordinates, set $U=A_\sigma/r_{\rm p}^2$, with $r_{\rm p}=(\xi^2+\eta^2)^{1/2}$. Wherever the sampled phase volume is nonzero, $A_\sigma>0$. One has $r_{\rm p}|dr_{\rm p}|=\tfrac12 A_\sigma U^{-2}\,dU$, and $Q=\xi\sqrt U=\sqrt{A_\sigma}\cos\vartheta$. The continuous transformation $(U,\xi)\mapsto(U,\xi\sqrt U)$ yields the convergence in \eqref{eq:weyl-joint-pair}. The radial factor $U^{-2}\,dU$ and the angular factor $A_\sigma\,dz\,d\vartheta/(2\pi C_\sigma)$ are separately normalized and factorize, proving independence. For $-c=b$, the angular measure is uniform. Thus $Q/R_*=\cos\vartheta$ has the first law in \eqref{eq:weyl-isotropic-cdfs}. The coordinate $\xi/R_*$ of a uniform point in the quarter disk gives the second law by integrating its chord length $\sqrt{1-q^2}$. The tails and the two coordinate distributions are plotted in \autoref{fig:joint-spectral-law}.
\end{proof}

\paragraph{Constant-coefficient comparison.}
With constant $c$ and zero bottom friction, the comparison sine mode at $(k,n)$ has $\ell_n=n\pi/h$ and $\mu=k^2+\ell_n^2$. Its nonzero rates solve
\[
 \lambda^2+\nu\mu\lambda-\frac{ck^2-b\ell_n^2}{\mu}=0 .
\]
This follows by dividing the velocity pencil by the mode's kinetic mass. In the isotropic stable case $-c=b>0$, the slowly damped root for large $\mu$ has $a(\nu\mu-a)=b$, hence $a\sim b/(\nu\mu)$. The leading count is the number of first-quadrant lattice points in a disk of radius $\sqrt{b/(\nu s)}$, with lattice spacing $1$ in $k$ and $\pi/h$ in $\ell_n$. It gives $s N_-(s,\infty)\to hb/(4\nu)$, as in \eqref{eq:weyl-total-negative}. In the mixed-slip case the wall term at each $k$ is a multiple of $|\phi'(0)|^2$. The index changes by at most one at each $k$, by restriction to $\phi'(0)=0$, so the scaled discrepancy is $O(\sqrt{s})$.

\begin{proof}[Proof of \autoref{prop:inverse-stratification}]
For the angular integral in \eqref{eq:weyl-total-positive} at a height where $c>0$, set $\tan\vartheta=\sqrt{c/b}\,t$. The support of the positive part is then $0\le t<1$, and the angular integral equals
\[
 \frac{c^{3/2}}{\sqrt b}\int_0^1\frac{1-t^2}{(1+(c/b)t^2)^2}\,dt .
\]
The integral vanishes at all other heights. It follows that $C_+(b)$ depends on the distribution of $c_+$. For $b>\|c_+\|_\infty$, the power series of $(1+y)^{-2}$ is absolutely and uniformly convergent for $y=(c_+/b)t^2$. Integrating termwise yields \eqref{eq:inverse-moment-expansion}, using $\int_0^1(1-t^2)t^{2n}\,dt=2/((2n+1)(2n+3))$.

To compare two profiles with the same sampled coefficients, put $w=1/b$. The right-hand side of \eqref{eq:inverse-moment-expansion} is analytic at $w=0$. Equality along a sequence $w\downarrow0$ identifies all coefficients, hence all $M_{n+3/2}$. These are the ordinary polynomial moments of the measure $s^{3/2}(c_+)_\# dz(ds)$. Polynomials are dense in the continuous functions on its compact support, so the measure is determined. Dividing by $s^{3/2}$ on $s>0$ gives the positive part of the distribution, and the known height $h$ determines its mass at zero. \autoref{fig:inverse-stratification} illustrates the loss of vertical ordering.

Finally, $0\le1-(1+y)^{-2}\le2y$ for $y\ge0$. Substitution into the angular formula proves \eqref{eq:inverse-single-moment}, since $\int_0^1 (1-t^2)t^2\,dt=2/15$. The stated error bound then follows by the triangle inequality.
\end{proof}

\begin{proof}[Proof of \autoref{prop:fixed-mode-count}]
If $b>0$, the drive at fixed $k$ is
$k^2\int c|\phi|^2-b\|\phi'\|_2^2$ and has finite positive index on $H^1_0(0,h)$. Its positive index on $H^2\cap H^1_0$ is the same by density, by approximating a basis of any finite-dimensional positive subspace. The form in \eqref{eq:mode-unified-form} with $\tau=-1$ is bounded above by this drive. Conversely, any finite-dimensional subspace on which the drive is positive remains positive for \eqref{eq:mode-unified-form} for all sufficiently small $s$, by uniformity on its unit sphere. This proves the Dirichlet count by min--max.

Suppose $b=0$. Use the same vertical bracketing as in \autoref{thm:vertical-weyl}, now at fixed $k$ and with a fixed partition. The total codimension error is then bounded independently of $s$. On an interval of length $\ell$ and constant coefficient $C$, put $\eta=s^{1/4}n\pi/\ell$ in the sine mode. The positive-index criterion reads
\[
 k^2 C>\nu\eta^4+2\nu k^2s^{1/2}\eta^2
                  +\nu sk^4+s^{3/2}\eta^2+s^2 k^2 .
\]
For $C>0$ the allowed interval has endpoint tending to $(k^2 C/\nu)^{1/4}$. For $C\le0$ it is empty. Hence $s^{1/4}$ times the comparison count tends to $\ell(k^2[C]_+/\nu)^{1/4}/\pi$. Summing over the partition and refining it gives the second assertion.

For growing modes with $b<0$ and for slowly damped modes in the strictly stable regime, the coefficient $R$ of $\|\phi'\|_2^2$ in \eqref{eq:mode-unified-form} is strictly positive. Bracket $C(z)$ between its global minimum and maximum and compare to the sine problem as before. For each constant comparison coefficient put $\eta=\sqrt{s}\,n\pi/h$. After multiplying the sine inequality by $s$, the left side minus the right side converges uniformly on bounded intervals to $R\eta^2-\nu\eta^4$. All admissible $\eta$ are uniformly bounded because of the negative quartic term. The limiting polynomial is strictly positive on $(0,\sqrt{R/\nu})$ and strictly negative above that interval. Bounding the counting error on arbitrarily small neighborhoods of its two nonnegative roots gives the limit $h\sqrt{R/\nu}/\pi$. For the damped count, the fixed subtraction at $a_{\rm c}$ vanishes after multiplication by $\sqrt{s}$. This proves both remaining formulas.
\end{proof}

\begin{proof}[Proof of \autoref{prop:threshold-location}]
The substitution $u=D_z^{1/2}\phi$ takes the drive in \autoref{prop:fixed-mode-count}(1) to
\[
 k^2(D_z^{-1/2}cD_z^{-1/2}u,u)_2-b\|u\|_2^2 .
\]
The operator $K=D_z^{-1/2}cD_z^{-1/2}$ is positive, compact and injective on $L^2(0,h)$. A positive eigenvalue $\eta$ gives the Dirichlet equation $-\phi''=\eta^{-1}c\phi$ for $\phi=D_z^{-1/2}u$. Its solutions with $\phi(0)=0$ form a one-dimensional space, by uniqueness for the second-order equation with prescribed $\phi'(0)$, so every positive eigenvalue is simple. The count and first-threshold characterization follow from min--max. No boundary friction occurs in this index. Vertical reflection commutes with $D_z$ and conjugates multiplication by $c$ to multiplication by $c(h-\cdot)$, proving the exact reflection invariance.

For the first threshold, take a nonnegative normalized maximizer $\phi$ of the Rayleigh quotient. A real maximizer exists by compact embedding, and replacing it by its absolute value leaves the quotient unchanged. The positive weight and the Dirichlet equation imply that $\phi>0$ in $(0,h)$ and is strictly concave. For the centered rearrangement $\phi^\#$, the layer-cake formula gives $\int c\phi^2\le\int c^\#(\phi^\#)^2$, since centered superlevel intervals have the largest possible overlap. Also $\int|(\phi^\#)'|^2\le\int|\phi'|^2$. To see the latter here, for a regular level $t$ of $\phi$ the two inverse branches contribute $a_1,a_2$ to $-\frac{d}{dt}|\{\phi>t\}|$, where $a_i=1/|\phi'|$ on the branches. Their contribution to the Dirichlet integral after changing variables to $t$ is $1/a_1+1/a_2\ge4/(a_1+a_2)$. The centered rearrangement has two equal inverse slopes and attains the latter expression. This proves \eqref{eq:threshold-rearrangement}.

For $c_{\rm w}$, its centered decreasing rearrangement is $c_{\rm m}$, which has the same distribution by a half-period shift. For the maximizing $\phi$ of $c_{\rm w}$, a high superlevel set of $c_{\rm w}$ lies close to the two walls, while a high superlevel set of $\phi$ lies strictly inside the interval. Over a nonzero range of the two levels these sets are disjoint and both have positive measure. Centering them gives positive overlap, so the layer-cake inequality for the numerator is strict. This proves the strict threshold comparison. The identical Weyl responses follow from \autoref{prop:inverse-stratification}.

For the perturbative formula write $K(\epsilon)=c_0D_z^{-1}+\epsilon D_z^{-1/2}gD_z^{-1/2}$. Each simple eigenvalue at $\epsilon=0$ has a differentiable local branch, as follows by the norm-convergent Neumann expansion of its Riesz projection on a circle isolating that eigenvalue. Differentiating the eigenvalue pairing at zero and using the normalized eigenfunction $e_n=(2/h)^{1/2}\sin(n\pi z/h)$ gives
\[
 \left.\frac{d}{d\epsilon}B_{n,k}[c_0+\epsilon g]\right|_{\epsilon=0}
   =k^2(n\pi/h)^{-2}\int_0^h g|e_n|^2\,dz .
\]
Now $2\sin^2(n\pi z/h)=1-\cos(2n\pi z/h)$, which gives \eqref{eq:threshold-linearized-data}. The Riemann--Lebesgue lemma gives the mean from the large-$n$ limit, and completeness of the cosine basis gives \eqref{eq:threshold-symmetric-reconstruction} because reflection removes the odd cosine coefficients.
\end{proof}

\begin{proof}[Proof of \autoref{cor:threshold-envelope}]
Write $\lambda_m[c]$ for the $m$th Dirichlet eigenvalue of $-\phi''=\lambda c\phi$, so $B_{m,k}[c]=k^2/\lambda_m[c]$. By the Sturm oscillation theorem an eigenfunction at $\lambda_m[c]$ has $m$ nodal intervals. On each interval divide the absolute value of the eigenfunction by its maximum, and glue the resulting functions to form a nonnegative $w\in H^1_0(0,h)$. On each nodal interval $w$ is positive, is strictly concave and has maximum one. Testing the equation there with that piece of $w$ and summing gives
\[
 \int_0^h |w'|^2\,dz=\lambda_m[c]\int_0^h cw^2\,dz.
\]
For $0<t<1$ put $a_i(t)=|w'(z_i(t))|^{-1}$ at the $2m$ points with $w(z_i(t))=t$. There are two points on each nodal interval and $w'\ne0$ at these points. If $L(t)=|\{w>t\}|$, then $-L'(t)=\sum_i a_i(t)$. Changing variables on the monotone pieces gives
\[
 \int_0^h |w'|^2\,dz
 =\int_0^1\sum_{i=1}^{2m}a_i(t)^{-1}\,dt
 \ge 4m^2\int_0^1\frac{dt}{-L'(t)}
 =m^2\int_0^h |(w^\#)'|^2\,dz .
\]
Here $w^\#$ has two centered inverse branches of slope magnitude $-L'(t)/2$ in the level variable. The formula follows first on compact subintervals of $(0,1)$ and then on the whole interval by monotone convergence. It gives finite energy and zero endpoint traces for $w^\#$. By the layer-cake formula, $\int cw^2\le\int c^\#(w^\#)^2$, so the first Dirichlet Rayleigh quotient yields $\lambda_m[c]\ge m^2\lambda_1[c^\#]$.

The centered rearrangement $c^\#$ is continuous. Indeed the continuous profile $c$ takes all intermediate values, and the preimage of any open interval strictly inside its range has positive length. Thus its quantile has no jump, including the one-sided limits at the extremes. Adjacent copies of $c^\#$ have equal endpoint values, so $c^{\#,[m]}$ is continuous and has the same distribution as $c$. The positive first eigenfunction for $c^\#$ is symmetric by uniqueness up to scale. Rescaling it on each subinterval and alternating its sign matches first derivatives across the internal endpoints. The resulting eigenfunction has $m-1$ internal zeros and eigenvalue $m^2\lambda_1[c^\#]$. This proves \eqref{eq:threshold-envelope}. If the right-hand side of \eqref{eq:threshold-count-envelope} is $m\ge1$, then $b<m^{-2}B_{1,k}[c^\#]$. The profile $c^{\#,[m]}$ attains at least $m$ modes, while \eqref{eq:threshold-envelope} bounds the count from above by $m$ for every admissible profile. If the right-hand side is zero the first-threshold bound gives zero throughout the class.
\end{proof}

\FloatBarrier

\section{Nonlinear dynamics}\label{sec:nonlinear-proof}

\subsection{Global strong solutions}\label{sec:global-proof}

We differentiate the velocity equation in time at the Galerkin level and use $L^q$ force bounds with $q>2$ to construct the flow map and prove uniqueness.

\begin{proof}[Proof of \autoref{thm_nonlinear_regularity}]
\medskip
\noindent\textit{Approximation and uniform estimates.}
Let $\mathbb P_N$ denote the orthogonal projection in $\mathcal H$ onto the first $N$ eigenvectors of $A$. These eigenvectors belong to $H^3$ by the $H^1$ Stokes estimate in \autoref{lem_stokes_estimates} and hence are Lipschitz. Take $\mathbf u_N(0)=\mathbb P_N\mathbf u_0$ and smooth periodic scalar data of the required parities converging in $L^q$. Solve the two scalar equations with $\mathbf u_N$ as the transport field and solve
\begin{equation}\label{eq:galerkin}
 (\mathbf u_N)_t+\nu A\mathbf u_N
 =\mathbb P_N\mathbb P\bigl(\mathbf F_N-(\mathbf u_N\cdot\nabla)\mathbf u_N\bigr),
 \qquad \mathbf F_N=fv_N\eone-\rho_N\nabla\Psi .
\end{equation}
To prove local existence at fixed $N$, we use continuous curves of velocity coefficients as input, solve the scalar equations along their flows, and integrate \eqref{eq:galerkin} as an equation for the new coefficients. On a bounded ball of coefficient curves, the flow and its inverse depend Lipschitz continuously on the curves, with a bound $C_N T$ on $[0,T]$ by the flow equations. The smooth initial scalars and the bounded derivatives of the source coefficients therefore give local Lipschitz dependence of $\mathbf F_N$ in $L^2$, and time integration gives a contraction for small $T$. The same construction restarts from Lipschitz scalar traces at positive times.

The approximations satisfy the scalar estimates in \autoref{lem_transport_lp} and the energy identity \eqref{eq_relative_energy}: the velocity test lies in the Galerkin space, and the background pressure correction has zero pairing with it by impermeability on the representative cell. This controls the velocity coefficients uniformly in time. At fixed $N$ a bound for the coefficients controls the accumulated Lipschitz norm on any finite interval and hence the scalar Lipschitz norms there. Thus the approximate solutions extend globally.

We establish estimates uniform in $N$ directly from \eqref{eq:galerkin}. We suppress the Galerkin index in the estimates below. In particular, the undiffused scalars are bounded uniformly in $L^q$ and the velocity obeys \eqref{eq_velocity_energy_class}.

Write $\mathbf{F}=fv\mathbf e_1-\rho\nabla\Psi$. To estimate the acceleration, we first bound $\mathbf{F}_t$ in the dual of $H^1(\Omega)$. For a periodic test field $\boldsymbol{\varphi}$ the two scalar equations and zero normal velocity give
\begin{equation*}
\begin{aligned}
\langle \mathbf{F}_t,\boldsymbol{\varphi}\rangle
={}&\int_\Omega
  (fv\mathbf e_1-\rho\nabla\Psi)
       \cdot(\mathbf{u}\cdot\nabla)\boldsymbol{\varphi}\,dx\,dz\\
&-\int_\Omega \rho\,(\mathbf{u}\cdot\nabla)\nabla\Psi
                   \cdot\boldsymbol{\varphi}\,dx\,dz\\
&+\int_\Omega \left(
    (\mathbf{u}\cdot\nabla\rho_s)\nabla\Psi
    -f(f+a_0)u\mathbf e_1\right)\cdot\boldsymbol{\varphi}\,dx\,dz.
\end{aligned}
\end{equation*}
Set $s_q=2q/(q-2)$, so that $1/q+1/s_q=1/2$. By H\"older's inequality and the embedding $H^1(\Omega)\hookrightarrow L^{s_q}(\Omega)$,
\begin{equation*}
\begin{aligned}
|\langle \mathbf{F}_t,\boldsymbol{\varphi}\rangle|
&\le C\bigl(\|\rho\|_{L^q(\Omega)}+\|v\|_{L^q(\Omega)}\bigr)
         \|\mathbf{u}\|_{L^{s_q}(\Omega)}
         \bigl(\|\nabla\boldsymbol{\varphi}\|_{L^2(\Omega)}
                +\|\boldsymbol{\varphi}\|_{L^2(\Omega)}\bigr)\\
&\quad+C\|\mathbf{u}\|_{L^2(\Omega)}
               \|\boldsymbol{\varphi}\|_{L^2(\Omega)}
\le C\|\mathbf{u}\|_{H^1(\Omega)}\|\boldsymbol{\varphi}\|_{H^1(\Omega)}.
\end{aligned}
\end{equation*}
The initial acceleration is given by the Leray projection of the momentum equation,
\begin{equation*}
\mathbf u_t(0)=-\nu A\mathbf u(0)
 +\mathbb P_N\mathbb P\bigl(-(\mathbf u(0)\cdot\nabla)\mathbf u(0)
              +fv(0)\eone-\rho(0)\nabla\Psi\bigr)\in L^2(\Omega),
\end{equation*}
with a uniform bound because $\mathbf u(0)=\mathbb P_N\mathbf u_0$ is bounded in $D(A)$ and $H^2(\Omega)$ embeds into $L^\infty(\Omega)$.

Differentiate the momentum equation in time and test with $\mathbf{u}_t$.
The projections disappear in tests with vectors in the Galerkin space, and the transport term $(\mathbf u\cdot\nabla)\mathbf u_t$ has zero pairing with $\mathbf u_t$. The Navier dissipation controls $\|\mathbf{u}_t\|_{H^1(\Omega)}^2$ by \autoref{lem_stokes_estimates}.
For the remaining convection term we have
\begin{equation*}
\begin{aligned}
\left|\int_\Omega
(\mathbf{u}_t\cdot\nabla)\mathbf{u}\cdot\mathbf{u}_t\,dx\,dz\right|
&\le\|\nabla\mathbf{u}\|_{L^2(\Omega)}\|\mathbf{u}_t\|_{L^4(\Omega)}^2
\le C\|\nabla\mathbf{u}\|_{L^2(\Omega)}
      \|\mathbf{u}_t\|_{L^2(\Omega)}\|\mathbf{u}_t\|_{H^1(\Omega)}\\
&\le\varepsilon\|\mathbf{u}_t\|_{H^1(\Omega)}^2
   +C_\varepsilon\|\mathbf{u}\|_{H^1(\Omega)}^2
                   \|\mathbf{u}_t\|_{L^2(\Omega)}^2,\\
|\langle\mathbf{F}_t,\mathbf{u}_t\rangle|
&\le\varepsilon\|\mathbf{u}_t\|_{H^1(\Omega)}^2
   +C_\varepsilon\|\mathbf{u}\|_{H^1(\Omega)}^2.
\end{aligned}
\end{equation*}
It follows that
\begin{equation}\label{eq_nonlinear_acceleration_energy}
\frac{d}{dt}\|\mathbf{u}_t\|_{L^2(\Omega)}^2
 +c\|\mathbf{u}_t\|_{H^1(\Omega)}^2
\le C\|\mathbf{u}\|_{H^1(\Omega)}^2
       \bigl(1+\|\mathbf{u}_t\|_{L^2(\Omega)}^2\bigr).
\end{equation}
By \eqref{eq_velocity_energy_class}, Gronwall's inequality followed by integration of \eqref{eq_nonlinear_acceleration_energy} proves
\begin{equation}
\label{eq_nonlinear_acceleration_bounds}
\mathbf{u}_t\in L^\infty(0,\infty;L^2(\Omega))
                  \cap L^2(0,\infty;H^1(\Omega)).
\end{equation}

Testing the momentum equation with $\mathbf u$ first gives
$\nu\Diss(\mathbf u)=(\mathbf F-\mathbf u_t,\mathbf u)_2$. The Stokes graph estimate, $L^2$ contractivity of the projections, and the two-dimensional bound
$\|(\mathbf u\cdot\nabla)\mathbf u\|_2
 \le C\|\mathbf u\|_{H^1}^{3/2}\|\mathbf u\|_{H^2}^{1/2}$
then give, by coercivity and Young's inequality,
\begin{equation}\label{eq:pointwise-velocity}
\begin{aligned}
 \|\mathbf u\|_{H^1}^2
   &\le C\bigl(\|\mathbf F\|_2+\|\mathbf u_t\|_2\bigr)\|\mathbf u\|_2 ,\\
 \|\mathbf u\|_{H^2}
   &\le C\bigl(\|\mathbf F\|_2+\|\mathbf u_t\|_2+\|\mathbf u\|_{H^1}^3\bigr).
\end{aligned}
\end{equation}
The energy, acceleration and scalar bounds make both estimates uniform in time and in the Galerkin dimension. For the limiting solution the same test and the $L^2$ Stokes estimate also bound $\|p\|_{H^1}$ by the second right-hand side.
\medskip
\noindent\textit{Passage to the limit and strong regularity.}
The uniform $H^2$ bound and time-derivative bound yield, after extraction, strong convergence of $\mathbf u_N$ in $C([0,T];H^{2-\epsilon}(\Omega)^2)$ for every $T<\infty$ and $0<\epsilon<1$. Indeed, interpolation gives equicontinuity there from equicontinuity in $L^2$, and the compact embedding of $H^2$ gives pointwise compactness. The scalar sequences converge weakly-* in $L^\infty(0,T;L^q)$. In the conservative transport equations, the products $\mathbf u_N\rho_N$ and $\mathbf u_Nv_N$ converge in distributions by strong convergence of $\mathbf u_N$ in $C([0,T];L^\infty)$. The same convergence and strong convergence in $C([0,T];H^1)$ give strong convergence in $L^2$ of the velocity convection on finite intervals. Testing \eqref{eq:galerkin} against finite spans of the Stokes eigenvectors and then using density gives the projected momentum equation for the limit. The convergence in the time-integrated equations also gives the stated scalar initial data in distributions. The weak bounds retain the velocity wall and reflection conditions.

For almost every $t$, the stationary Stokes problem with force
$(\mathbf F-\mathbf u_t-(\mathbf u\cdot\nabla)\mathbf u)/\nu$ has velocity $\mathbf u(t)$ by the projected equation and defines the mean-zero pressure after rescaling by $\nu$. The $L^q$ Stokes estimate applies at this stage. In fact, $\mathbf F$ and $(\mathbf u\cdot\nabla)\mathbf u$ are bounded in $L^\infty(0,\infty;L^q)$ by the scalar bounds and $H^2$ velocity control, and $\mathbf u_t$ belongs to $L^2(0,\infty;L^q)$ by the acceleration estimate. This gives \eqref{eq:velocity-lq} and $p\in L^\infty(0,\infty;H^1)$ by the $L^2$ Stokes estimate. The embedding $W^{2,q}(\Omega)\hookrightarrow C^1(\overline\Omega)$ for $q>2$ gives an integrable $C^1$ norm on each finite time interval. The flow and its inverse are $C^1$ by the derivative equation, and the flow has determinant one by incompressibility. Tangency of the velocity at all four sides gives invariance of each closed half cell. Uniqueness for linear transport by this velocity gives the two characteristic formulas, by applying the transport equations for the total density and the shifted total meridional velocity from \autoref{lem_transport_lp}. In particular the joint invariants hold for the limit itself. Strong continuity of the scalars in $L^q$ follows from uniform continuity of the flow and inverse flow on compact time intervals and approximation of the initial scalars by continuous functions in $L^q$.

The projected time derivative of the momentum equation gives $\mathbf u_{tt}\in L^2(0,T;\mathcal V')$. To see this, the bound on $\mathbf F_t$ above uses only the uniform $L^q$ scalar bounds and the velocity energy norm, so it holds in distributions for the limit. The differentiated convection
\[
 (\mathbf u_t\cdot\nabla)\mathbf u+(\mathbf u\cdot\nabla)\mathbf u_t
\]
belongs to $L^2(0,T;L^2)$. For the first term use $\|\mathbf u_t\|_4\|\nabla\mathbf u\|_4$, and for the second use $\|\mathbf u\|_\infty\|\nabla\mathbf u_t\|_2$. Both bounds follow from $\mathbf u\in L^\infty H^2$ and $\mathbf u_t\in L^2 H^1$. The viscous term is controlled in $\mathcal V'$ by the form. The Hilbert triple $\mathcal V\subset\mathcal H\subset\mathcal V'$ then gives $\mathbf u_t\in C([0,T];\mathcal H)$. Also $\mathbf u\in C([0,T];H^1)$ and interpolation with the uniform $H^2$ bound gives continuity into $H^{2-\epsilon}$ for every $0<\epsilon<1$. Thus the stationary Stokes force is continuous into $L^2$ in time. The stationary estimate gives the claimed $D(A)$ and $H^1$ continuity of velocity and pressure, including at $t=0$. This proves \eqref{eq:strong-continuity}. The same Hilbert triple time-energy identity justifies testing the differentiated momentum equation with $\mathbf u_t$ for the limit, so \eqref{eq_nonlinear_acceleration_energy} holds for this solution.

Finally, for $2\le p<\infty$, interpolation gives
\begin{equation*}
\begin{aligned}
\int_0^\infty\|\mathbf{u}(t)\|_{W^{1,p}(\Omega)}^p\,dt
&\le C\int_0^\infty \|\mathbf{u}(t)\|_{H^1(\Omega)}^2
                            \|\mathbf{u}(t)\|_{H^2(\Omega)}^{p-2}\,dt\\
&\le C\|\mathbf{u}\|_{L^\infty(0,\infty;H^2(\Omega))}^{p-2}
        \|\mathbf{u}\|_{L^2(0,\infty;H^1(\Omega))}^2<\infty.
\end{aligned}
\end{equation*}
This proves \eqref{eq_nonlinear_global_regularity}. To estimate the strain, set $\theta=(q-2)/(2(q-1))$. The two-dimensional Gagliardo--Nirenberg inequality gives
\[
 \|\nabla\mathbf u\|_{L^\infty}
 \le C\|\mathbf u\|_{H^1}^{\theta}\|\mathbf u\|_{W^{2,q}}^{1-\theta}.
\]
For $0\le S<T$, H\"older's inequality, \eqref{eq:velocity-lq} and $T-S\le T$ give
\[
 \int_S^T\|\nabla\mathbf u(t)\|_{L^\infty}\,dt
 \le C(1+T)^{1-\theta/2}
       \left(\int_S^T\|\mathbf u(t)\|_{H^1}^2\,dt\right)^{\theta/2}.
\]
Take $S=0$ for the upper bound. For the little-$o$ estimate, divide by $T^{1-\theta/2}$, first let $T\to\infty$ and then send $S\to\infty$, using the finite $H^1$ dissipation. Since $\eta_q=1-\theta/2$, this proves \eqref{eq:sublinear-strain}.

\medskip
\noindent\textit{Propagation of Sobolev regularity.}
On each half cell, write $Z=(\rho+\rho_s,v+v_s+fx)$. From $Z(t)=Z(0)\circ X_t^{-1}$, measure preservation and the inverse-flow derivative bound, we get
\begin{equation}\label{eq:transport-gradient}
 \|\nabla Z(t)\|_{L^r(\Omega_\sigma)}
 \le \|\nabla Z(0)\|_{L^r(\Omega_\sigma)}
       \exp\!\left(\int_0^t\|\nabla\mathbf u(s)\|_{L^\infty}\,ds\right).
\end{equation}
The perturbations have global periodic $W^{1,r}$ regularity by their characteristic formulas on the periodic strip: the source terms $-\mathbf u\cdot\nabla\rho_s$ and $-(f+a_0)u$ are in $L^1(0,T;W^{1,\infty})$. This gives the scalar conclusion in \eqref{eq_nonlinear_finite_time_regularity} without using the lifted affine label across the periodic seam.

In particular, $\mathbf{F}\in L^\infty(0,T;H^1(\Omega))$. The product rule, the embedding $H^2(\Omega)\hookrightarrow L^\infty(\Omega)$, and the $L^4(\Omega)$ bound for $\nabla\mathbf{u}$ give
\begin{equation*}
\begin{aligned}
\|(\mathbf{u}\cdot\nabla)\mathbf{u}\|_{H^1(\Omega)}
&\le C\left(\|\mathbf{u}\|_{L^\infty(\Omega)}
                  \|\nabla\mathbf{u}\|_{H^1(\Omega)}
             +\|\nabla\mathbf{u}\|_{L^4(\Omega)}^2\right)
             \le C\|\mathbf{u}\|_{H^2(\Omega)}^2.
\end{aligned}
\end{equation*}
The steady $H^1$ Stokes estimate and \eqref{eq_nonlinear_acceleration_bounds}
therefore give $\mathbf{u}\in L^2(0,T;H^3(\Omega))$.
For $2<p<\infty$, interpolation between $H^2(\Omega)$ and $H^3(\Omega)$ yields
\begin{equation*}
\begin{aligned}
\int_0^T\|\mathbf{u}(t)\|_{W^{2,p}(\Omega)}^{\frac{2p}{p-2}}\,dt
&\le C\|\mathbf{u}\|_{L^\infty(0,T;H^2(\Omega))}^{\frac{4}{p-2}}
        \int_0^T\|\mathbf{u}(t)\|_{H^3(\Omega)}^2\,dt<\infty.
\end{aligned}
\end{equation*}
The endpoint $p=2$ is the uniform $H^2(\Omega)$ bound from part (1).

\medskip
\noindent\textit{Eulerian continuous dependence.}
For a scalar difference $d$ let $\phi=(I-\Delta_N)^{-1}d$, with periodicity in $x$ and Neumann conditions at the walls. Then $\|d\|_{-1,N}=\|\phi\|_{H^1}$. For $d=d_\rho$ or $d_v$, the difference equation has the form
\[
 d_t+\mathbf u^{(1)}\cdot\nabla d
   =-\nabla\cdot(Ub_2)+g ,
\]
where $b_2=\rho^{(2)}$ and $g=-U\cdot\nabla\rho_s$ in the density equation, and $b_2=v^{(2)}$ and $g=-(f+a_0)U_1$ in the meridional equation. Here $U_1$ is the horizontal component of $U$. In both cases $\|g\|_2\le C\|U\|_2$.

Testing with $\phi$, the transport contribution is
\[
 \int_\Omega d\,\mathbf u^{(1)}\cdot\nabla\phi\,dx\,dz
   =\sum_{i,j}\int_\Omega
       (\partial_i\phi)(\partial_i u^{(1)}_j)(\partial_j\phi)\,dx\,dz .
\]
Indeed, substitute $d=\phi-\Delta\phi$ and integrate by parts; the normal traces of $\mathbf u^{(1)}$ and $\nabla\phi$ give zero boundary terms. This test is legitimate since $d\in L^\infty(0,T;L^2)$, $d_t\in L^2(0,T;(H^1_{\rm per})')$, and the Neumann solution is in $H^2$ a.e. in time. The right-hand side of the difference equation pairs with $\phi$ with absolute value at most
\[
 \bigl(\|b_2\|_{L^q}\|U\|_{L^{s_q}}+C\|U\|_2\bigr)\|\phi\|_{H^1}.
\]
For $Y=\|d_\rho\|_{-1,N}^2+\|d_v\|_{-1,N}^2$, Young's inequality and $H^1\hookrightarrow L^{s_q}$ give, for any $\varepsilon>0$,
\[
 Y'\le \varepsilon\|U\|_{H^1}^2
      +C_\varepsilon(1+\|\nabla\mathbf u^{(1)}\|_{L^\infty})Y .
\]
The velocity difference energy satisfies
\[
 \frac12\frac{d}{dt}\|U\|_2^2+\nu\Diss(U)
 \le C\|\nabla\mathbf u^{(2)}\|_2\|U\|_2\|U\|_{H^1}
      +C Y^{1/2}\|U\|_{H^1}.
\]
To estimate the force term, use $H^1$ duality and the boundedness of multiplication by $\nabla\Psi$ on $H^1$. Absorb the $H^1$ terms using coercivity. Summation and Gronwall's inequality give \eqref{eq:euler-dependence}; the coefficient is integrable on $[0,T]$ with uniform control by \eqref{eq:velocity-lq} and the uniform energy estimates. In particular, the strong solution is unique.
\end{proof}

\subsection{Decay and scalar limit points}\label{sec:limits-proof}

We use the decay of the planar velocity and force residual to identify the balances satisfied by weak accumulation points. Conservation of the joint law then gives a criterion for strong scalar convergence.

\begin{proof}[Proof of \autoref{thm_nonlinear_large_time}]
Put $y=\|\mathbf u_t\|_2^2$. The uniform acceleration bound in \autoref{thm_nonlinear_regularity} reduces \eqref{eq_nonlinear_acceleration_energy} to $y'\le C\|\mathbf u\|_{H^1}^2$. Integrate from times $r_j\to\infty$ at which $y(r_j)\to0$, using the integrability of both $y$ and $\|\mathbf u\|_{H^1}^2$, to obtain $\mathbf u_t(t)\to0$ in $L^2$. Similarly, $\|\mathbf u\|_2^2$ and its derivative $2(\mathbf u,\mathbf u_t)_2$ are integrable, so $\mathbf u(t)\to0$ in $L^2$. Interpolation with the uniform $H^2$ bound gives decay in $H^1$. For $2\le s<\infty$, the two-dimensional interpolation inequality
\[
 \|\mathbf u\|_{W^{1,s}}\le C_s\|\mathbf u\|_{H^1}^{2/s}
                                  \|\mathbf u\|_{H^2}^{1-2/s}
\]
gives decay in $W^{1,s}$, and the case $1\le s<2$ follows by finite measure.

Write the force residual as $\mathbf u_t+\operatorname{div}(\mathbf u\otimes\mathbf u)-\nu\Delta\mathbf u$. Its $H^{-1}$ norm is bounded by $C(\|\mathbf u_t\|_2+\|\mathbf u\|_{L^4}^2+\nu\|\nabla\mathbf u\|_2)$, which converges to zero. The bound is square integrable in time since
\[
 \int_0^\infty\|\mathbf u(t)\|_{L^4}^4\,dt
 \le C\sup_{t\ge0}\|\mathbf u(t)\|_2^2
            \int_0^\infty\|\mathbf u(t)\|_{H^1}^2\,dt .
\]
The remaining two norms are square integrable by the energy and acceleration estimates. The global bounds give weak compactness of $(\rho,v,p)$ in $L^2\times L^2\times H^1$ along time sequences. Since the residual is linear in these three fields with fixed smooth coefficients, its weak limit equals the residual of the limit fields. This gives \eqref{eq_nonlinear_weak_balance}.

Finally, fix $\gamma>0$ and $\beta\in\mathbb R$ in \eqref{eq_relative_energy}. The dissipation integral there has a finite limit by its monotonicity and nonnegativity of the other terms, and the kinetic term tends to zero by \eqref{eq_nonlinear_decay}. Taking the limit gives \eqref{theorem-2-conc-4} with \eqref{eq_nonlinear_energy_constant}. Both terms on the left-hand side of \eqref{theorem-2-conc-4} are nonnegative and $\gamma>0$. Their sum tends to zero precisely when the two convergences in \eqref{eq_nonlinear_exact_profile} hold.
\end{proof}

\begin{proof}[Proof of \autoref{prop:limit-set}]
By the scalar bounds, the trajectory $Z(t)$ is bounded in $L^2(\Omega)^2$. This space embeds compactly into the product dual of $H^1_{\rm per}$ by the adjoint of the Rellich embedding. The closures of $\{Z(t):t\ge T\}$ in this dual norm are therefore nested nonempty compact sets. Each is connected because the trajectory is continuous in $L^q$ and hence in the dual norm. Their intersection is the set of limit points, which is compact and connected. At any such limit the scalar pair converges weakly in $L^2$ along a subsequence, and a further subsequence of pressures converges weakly in $H^1$. The balance then follows from \autoref{thm_nonlinear_large_time}. The gradient of the pressure is prescribed by the scalars in the balance, and its mean fixes the additive constant.

By \autoref{lem_transport_lp}, $\|Z(t)\|_{L^2(\Omega_\sigma)}=\|Z(0)\|_{L^2(\Omega_\sigma)}$ for either half cell. Expanding the square of the difference and using weak convergence after restriction to each half cell gives \eqref{eq:variance-defect}. If the joint laws at the limit equal the initial joint laws, their second moments agree, so this difference tends to zero. Conversely, strong convergence implies convergence in measure, hence convergence of the joint laws against all bounded continuous tests on each half cell, which gives equality with the initial joint laws. If $Z$ converges strongly, the force $\mathbf F$ converges in $L^2$ to the balanced force. Apply the $L^2$ Stokes estimate to the difference from the limiting equilibrium, with forcing $\mathbf F(t_j)-\mathbf F_\infty-\mathbf u_t(t_j)-(\mathbf u(t_j)\cdot\nabla)\mathbf u(t_j)$. Every term tends to zero in $L^2$ by strong scalar convergence and \eqref{eq_nonlinear_decay}. This proves the asserted velocity and pressure convergence.

For the last part, the normal velocity vanishes on the sides of each half cell. Integrating each component from a side where that component vanishes gives
$\|\mathbf u(t)\|_{L^\infty}\le C\|\nabla\mathbf u(t)\|_{L^\infty}$.
The flow maps are therefore uniformly Cauchy as $t\to\infty$. On every finite time interval the $L^2W^{2,q}$ estimate with $q>2$ gives an integrable $C^1$ velocity. The derivative equation $\partial_t DX_t=(D\mathbf u)(t,X_t)DX_t$ and the finite strain show that $DX_t$ is uniformly bounded and uniformly Cauchy. The flow and its inverse have uniform Lipschitz constants, and their determinants equal one. The inverses then converge in $C^1$ by the inverse derivative formula. Testing against continuous functions shows that the limiting maps preserve measure. The pullback formulas and approximation of $Z(0)$ in $L^q$ by continuous functions give the asserted strong limit. For $W^{1,q}$ initial labels, \eqref{eq:transport-gradient} and $\|\mathbf u\|_\infty\le C\|\nabla\mathbf u\|_\infty$ bound $\|\partial_t Z\|_{L^q}$ by a constant times $\|\nabla\mathbf u\|_\infty$. Integrating in time gives \eqref{eq:strain-tail}. The Stokes argument applies to the full trajectory once the scalar pair converges strongly.
\end{proof}

\FloatBarrier

\subsection{Convex order and balanced transport}\label{sec:transport-proof}

To obtain \eqref{eq:joint-lower-bound}, use the coupling of the initial and target laws induced by $(Z_t,Z_*)$ on each half cell. By \autoref{lem_transport_lp}, the first marginal is the initial joint law. The total cost is $\|Z_t-Z_*\|_2^2$, which tends to $C_{\gamma,\beta}$ by \eqref{theorem-2-conc-4}. The parity bound follows from $\|v+v_s\|_2^2=\|v+a_0x\|_2^2+|\Omega|a_1^2$.

\begin{proof}[Proof of \autoref{prop:convex-order}]
First consider one half cell equipped with normalized Lebesgue measure $m=\mu_\sigma$. For any sequence of equimeasurable fields converging weakly in $L^2$ to $W$, weak lower semicontinuity of convex integral functionals gives $W_\#m\preceq_{\rm cx}\lambda_0$. Indeed, a continuous convex function of at most quadratic growth has a continuous convex integral functional on $L^2$, hence a weakly lower semicontinuous one.

Conversely, Strassen's theorem \cite{Strassen1965} gives a probability kernel $K(w,dz)$ with second marginal $\lambda_0$ when averaged against $W_\#m$, and with mean $w$. This kernel defines a coupling $\Gamma(dx,dz)=m(dx)K(W(x),dz)$ between physical position and labels. Choose refining finite partitions $(E_i)$ of the half cell with mesh tending to zero. For each partition, choose a finite partition $(B_\ell)$ of label space and representatives $z_\ell$ so that $Q_n(z)=z_\ell$ on $B_\ell$ satisfies $\|Q_n-\mathrm{id}\|_{L^2(\lambda_0)}\to0$. Such quantizations are obtained by truncating the tails and using finer bins on bounded sets.

We partition each physical cell $E_i$ and each initial preimage $Z_0^{-1}(B_\ell)$ into pieces of masses $\Gamma(E_i\times B_\ell)$, indexed by the other index. Lebesgue measure has no atoms, and pairs of pieces of equal positive mass are isomorphic modulo null sets as Lebesgue probability spaces after normalization. Join the corresponding piecewise bijections to obtain a measure-preserving bijection $S_n$. The average of $Q_n(Z_0\circ S_n)$ over $E_i$ equals that of $x\mapsto\int Q_n(z)K(W(x),dz)$. The first quantization error is $\|Q_n-\mathrm{id}\|_{L^2(\lambda_0)}$, and Jensen's inequality bounds the $L^2(m)$ distance of the second field from $W$ by the same error. Testing against the piecewise averages of an arbitrary $L^2(m)$ test function now gives $Z_0\circ S_n\rightharpoonup W$. The piecewise averages converge strongly by density, and the rearranged fields have bounded $L^2$ norm. This proves the closure assertion. When the fields on the two half cells obey the prescribed reflection, the construction on the positive half cell can be reflected to the negative half cell.

For part (2), let $W$ be the weak limit. Pairing each component of $Z_{t_j}$ against each component of $W$ and using conservation of the full second moment matrix gives
\[
 \int (Z_{t_j}-W)\otimes(Z_{t_j}-W)\,d\mu_\sigma
       \longrightarrow D_\sigma .
\]
Weak convergence gives equality of the first moments, so this difference is also a difference of covariance matrices. It is positive semidefinite, and $\|e\cdot(Z_{t_j}-W)\|_{L^2(\mu_\sigma)}^2\to e^{\mathsf T}D_\sigma e$. Expand the target cost around $W$ and use the same weak convergence to eliminate its cross term. Equation \eqref{theorem-2-conc-4} then yields the first line of \eqref{eq:mixing-target-defect}, and the coupling induced by $(W,Z_*)$ gives the lower bound in the second line.

To compare with \eqref{eq:joint-lower-bound}, glue a martingale coupling $(w,z)$ from $\lambda_\sigma$ to $\lambda_{0,\sigma}$ to an optimal quadratic coupling $(w,y)$ from $\lambda_\sigma$ to $\lambda_{*,\sigma}$, conditionally independently given $w$. Since $\mathbb E[z-w\mid w]=0$, the resulting coupling $(z,y)$ has cost
\[
 \mathbb E|z-y|^2
 =m_2(\lambda_{0,\sigma})-m_2(\lambda_\sigma)
       +W_2^2(\lambda_\sigma,\lambda_{*,\sigma}).
\]
This is at least $W_2^2(\lambda_{0,\sigma},\lambda_{*,\sigma})$. For part (3), identical one-dimensional marginals give equal joint second moments. If the target is a weak limit, $D_\sigma$ therefore has zero trace, hence is zero. The sequence converges strongly on each half cell by the matrix identity, so its limiting joint law must equal the initial one. 
\end{proof}

\begin{proof}[Proof of \autoref{thm:balanced-transport-optimizer}]
\emph{Lower bound and the equality mechanism.}
Write $Y=Z_*$ in this proof. For a field admissible for the relaxed minimum, the convex-order condition gives a martingale kernel $K(W(x,z),dw)$ with initial law $\lambda_{0,\sigma}$ when averaged over $\mu_\sigma$ and conditional mean $W(x,z)$. Sample $(x,z)$ from $\mu_\sigma$ and $w$ from this kernel. The induced coupling $(Y,w)$ has cost
\[
 \mathbb E|Y-w|^2
  =\int |Y-W|^2\,d\mu_\sigma+
                  m_2(\lambda_{0,\sigma})-m_2(W_\#\mu_\sigma).
\]
Thus $\mathcal J(W)\ge\mathcal W$. Equality forces each of these couplings to be optimal. If the target-label laws are absolutely continuous, Brenier's theorem \cite{Brenier1991} implies $w=T_\sigma(Y)$ almost surely under each such coupling. Taking the conditional mean at the physical point yields $W=T_\sigma(Y)$, with the initial joint law. This proves uniqueness for the relaxed minimum in the absolutely continuous case once attainment is established. This coupling argument uses only convex order. It also proves uniqueness in the class with exact initial joint law without imposing balance, because the squared distance agrees with $\mathcal J$ on that class.

\medskip
\noindent\emph{Attainment for compact label laws.}
We first assume that the initial joint law on the positive half cell is compactly supported. Choose an optimal coupling $\pi_+(dy,dw)$ from $\lambda_{*,+}$ to $\lambda_{0,+}$. It is supported on $w\in\partial\varphi_+(y)$ for a globally Lipschitz convex $\varphi_+$. Indeed, for the equivalent maximization of $y\cdot w$, compact-support Kantorovich duality \cite[Chapter~5]{Villani2009OT} supplies a dual function on the compact terminal support. Taking the supremum of its affine functions with slopes in that support gives a globally Lipschitz convex potential, with equality on the optimal plan. Define $W(x,z)$ on the positive half cell to be the conditional mean of $w$ given $y=Y(x,z)$. This field belongs to the convex-order class by Jensen's inequality. The conditional mean belongs to the closed convex set $\partial\varphi_+(Y(x,z))$. Expanding the optimal cost by conditioning on $y$ gives $\mathcal J(W)=\mathcal W$ on this half cell.

The function $\varphi_+\circ Y$ is Lipschitz on the closed half cell, even when $\lambda_{*,+}$ is singular. At any point of the half cell where it is differentiable, the subgradient inequality for any $w\in\partial\varphi_+(Y)$ applied to positive and negative increments gives $\nabla(\varphi_+\circ Y)=DY^{\mathsf T}w$. Hence with $W=(R,\sqrt\gamma M)$,
\begin{equation}\label{eq:optimal-pressure-chain}
 \nabla\bigl(\gamma^{-1}\varphi_+(Z_*)\bigr)
             =fM\mathbf e_1-R\nabla\Psi
 \quad\text{almost everywhere on the positive half cell}.
\end{equation}
The identity involves the full gradient of $\Psi$. Subtract $f(\beta_f x^2/2+a_1x)$ and $p_s$ from $\gamma^{-1}\varphi_+(Z_*)$. The resulting potential has gradient $fv\mathbf e_1-\rho\nabla\Psi$, where $(\rho,v)$ is reconstructed from $W$. Its even reflection across the vertical sides gives a periodic $H^1$ pressure after subtracting the mean.

To verify the reflection identities for the cost and the labels, set $a=\sqrt\gamma a_1$. The initial labels on the negative half cell are obtained from the positive labels by $J_0(w)=(w_1,2a-w_2)$ after reflecting $x$. For target labels the corresponding map is $J_*(y)=(y_1,-y_2)$, since $\Psi$ is even in $x$. The transformed quadratic cost differs from $|y-w|^2$ by $4a(y_2-w_2)+4a^2$, whose integral is fixed by the marginals. Hence the reflected coupling is optimal. Reflect $W$ with $J_0$ to preserve its required parity and convex-order constraint. This yields an admissible field attaining $\mathcal W$ on the full cell.

\medskip
\noindent\emph{Attainment for finite moments.}
Clip the initial labels on the positive half cell to the Euclidean ball of radius $n$ and use the reflected truncation on the negative half cell. The resulting laws $\lambda_{0,\sigma}^{(n)}$ converge to $\lambda_{0,\sigma}$ in $W_2$ and against continuous convex functions of at most quadratic growth by dominated convergence. Their balanced barycentric optimizers $W_n$ have a uniform $L^2$ bound by convex order. Pass to an $L^2$ weak subsequence. The reflection passes to the limit as an affine constraint in $L^2$. The reconstructed force depends affinely and continuously on $W_n$ in $L^2$. Its Helmholtz projection vanishes for every $n$, so the weak limit also has zero projected force and the periodic mean-zero pressure is recovered by \autoref{lem_stokes_estimates}. Weak lower semicontinuity of the convex tests gives the original convex-order constraint. In $\mathcal J(W_n)$ with the truncated initial laws, expansion of the square cancels the second moment of $W_n$. The remaining field-dependent term is $-2\langle W_n,Z_*\rangle$, which passes to the weak limit. The optimal transport costs for the truncated laws tend to $\mathcal W$ by the triangle inequality for $W_2$. The limit is therefore admissible and attains $\mathcal W$. Every admissible field inherits the initial $L^q$ integrability. Indeed, take the convex radial functions that equal $r^q$ for $r=|w|\le n$ and equal $n^q+q n^{q-1}(r-n)$ for $r>n$. Testing convex order and taking the increasing limit gives $\int |w|^q\,d(W_\#\mu_\sigma)\le\int |w|^q\,d\lambda_{0,\sigma}$. Together with its pressure, the field defines a stationary solution with zero planar velocity. For the given trajectory, \autoref{thm_nonlinear_large_time} and \autoref{prop:convex-order} make every scalar accumulation point admissible, with $\mathcal J(W)=C_{\gamma,\beta}$ by \eqref{eq:mixing-target-defect}. This proves the equality criterion for the relaxed minimum and completes part (1).

\medskip
\noindent\emph{Exact laws and convergence.}
On each half cell $\det DZ_*=\gamma^{3/2}f\,\partial_z\Psi$. If it is nonzero almost everywhere, the area formula implies that $\lambda_{*,\sigma}$ is absolutely continuous. Under absolute continuity, the equality mechanism above gives the unique admissible optimizer $Z_{\rm opt}=T_\sigma(Z_*)$ and its exact initial laws. It inherits the original $L^q$ scalar integrability by equality of laws. With zero planar velocity the transport equations hold and its pressure gives the momentum balance, so this is a stationary solution.

If $C_{\gamma,\beta}=\mathcal W$, every weak scalar accumulation point of the flow is admissible by \autoref{thm_nonlinear_large_time} and \autoref{prop:convex-order}. Its relaxed cost equals $C_{\gamma,\beta}$ by \eqref{eq:mixing-target-defect}, so it is $Z_{\rm opt}$. The uniformly bounded trajectory therefore converges weakly in $L^2$ to $Z_{\rm opt}$. Since the optimizer has the initial joint laws, the norms are equal and convergence is strong. Conversely, strong convergence to $Z_{\rm opt}$ gives equality in the limit of the cost. The Stokes argument of \autoref{prop:limit-set} gives the velocity and pressure limits.

For any convex $C^1$ function $H$ with $L$-Lipschitz gradient,
\begin{equation}\label{eq:fenchel-leaf}
 H(y)+H^*(w)-y\cdot w
 \ge \frac{|w-\nabla H(y)|^2}{2L}.
\end{equation}
Indeed, evaluate the supremum defining $H^*(w)$ at $y+(w-\nabla H(y))/L$ and use the upper quadratic Taylor bound for $H$. Apply this estimate to the assumed global extension of $\varphi_\sigma$. The conjugate is integrable for the initial law by Fenchel equality on the optimal coupling. Since the marginal laws are fixed, twice the integral of this gap with $y=Z_*$ and $w=Z_t$ is the excess cost $\|Z_t-Z_*\|_2^2-\mathcal W$. This proves \eqref{eq:transport-excess-control}. The relative energy identity gives $\mathcal E_{\gamma,\beta}(t)-C_{\gamma,\beta}=2\gamma\nu\int_t^\infty\Diss(\mathbf u(s))\,ds$, which bounds the scalar cost excess when $C_{\gamma,\beta}=\mathcal W$.
\end{proof}

\subsubsection{Examples of balanced optimizers and a joint obstruction}

\paragraph{A coupled balanced configuration attaining the bound.}
For instance, take $\Psi(x,z)=z+\epsilon_0\cos x$, $\rho_s=\beta-\gamma\Psi$ and $a_0=a_1=0$. Choose $H$ smooth and compactly supported inside the positive target region at $\epsilon_0=0$, with mixed derivative not identically zero. Its support remains inside that region for sufficiently small $|\epsilon_0|$. For sufficiently small nonzero $\epsilon$, the function $\varphi(y)=|y|^2/2+\epsilon H(y)$ is globally uniformly convex. Prescribe $Z_0=\nabla\varphi(Z_*)$ on that half cell and extend by the reflection of the scalar pair. Since a convex gradient is the quadratic optimal map to its pushforward, these data give $Z_{\rm opt}=Z_0$. At $\epsilon_0=0$, the mixed derivative of $H$ makes the density perturbation vary horizontally and the meridional perturbation vary vertically. Both variations persist for sufficiently small nonzero $\epsilon_0$. Their forces still satisfy the joint balance by \eqref{eq:optimal-pressure-chain}. Taking zero initial planar velocity makes the solution stationary and gives $C_{\gamma,\beta}=\mathcal W=\|Z_0-Z_*\|_2^2>0$.

\paragraph{Coupled nonuniqueness for a two-dimensional potential.}
Fix $I=(a,b)\Subset(0,h)$ and let $A\ge0$ be smooth, zero exactly on $[a,b]$, and nonconstant outside it. For instance, let $A(z)=\chi(a-z)+\chi(z-b)$ with $\chi(r)=e^{-1/r^2}$ for $r>0$ and $\chi(r)=0$ otherwise. Take
\[
 \Psi(x,z)=\cos x+A(z),\qquad
 \rho_s=r_*=\beta-\gamma\Psi,\qquad a_0=a_1=0 .
\]
Choose two distinct translates $B_0,B_1$ of a nonzero smooth bump with compact support in $I$, with disjoint supports in $I$ and $0\le B_i\le1$. Put $s_i=-1+B_i$ and fix $0<\epsilon<f^2$. Define
\begin{equation}\label{eq:coupled-optimal-counterexample}
 \rho_i=\epsilon s_i(z),\qquad
 v_i=-\frac{\epsilon}{f}s_i(z)\sin x,\qquad
 \mathbf u_i=0,\qquad p_i=\epsilon(A-\overline A),
 \qquad i\in\{0,1\},
\end{equation}
where $\overline A=h^{-1}\int_0^h A(z)\,dz$. These fields are smooth and satisfy the reflection conditions. The potential depends on both coordinates, and both density and meridional velocity differ between the two fields. The force balance follows from
\[
 fv_i\mathbf e_1-\rho_i\nabla\Psi
    =(0,-\epsilon s_i A')=(0,\epsilon A')=\nabla p_i .
\]
Here $s_i=-1$ outside the plateau and $A'=0$ on it. The two fields therefore give the same nonconstant normalized pressure. Their total labels are
\[
 Z^{[i]}=Z_*+\epsilon s_i(z)\bigl(1,-\sqrt\gamma\sin x/f\bigr),
 \qquad Z_*=(\beta-\gamma\Psi,\sqrt\gamma f x).
\]
The labels agree outside the plateau. On the plateau, conditional on $x$, the two translates of the bump have identical distributions in $z$, so $Z^{[0]}$ and $Z^{[1]}$ have the same joint law on each half cell.

To verify optimality, define the following function on label space
\[
 G(y)=y_1-\beta+\gamma\cos\!\bigl(y_2/(\sqrt\gamma f)\bigr),
 \qquad \varphi(y)=\tfrac12|y|^2+\epsilon |G(y)|.
\]
The potential $\varphi$ is convex because it is the maximum of $\tfrac12|y|^2+\epsilon G(y)$ and $\tfrac12|y|^2-\epsilon G(y)$, both convex for $\epsilon<f^2$. Since $G(Z_*)=-\gamma A(z)$, off the plateau there is one gradient and $s_i=-1$ selects it. On the plateau $G(Z_*)=0$, and each $s_i\in[-1,0]$ selects a subgradient in the segment between the two branch gradients. Thus $Z^{[i]}\in\partial\varphi(Z_*)$ everywhere. Equality in Fenchel's inequality on the coupling $(Z_*,Z^{[i]})$ proves quadratic optimality for the common terminal law. With $Z^{[0]}$ as the initial field, both balanced states therefore attain the same strictly positive cost
\[
 \mathcal W=\|Z^{[i]}-Z_*\|_2^2
      =2\pi\epsilon^2\Bigl(1+\frac{\gamma}{2f^2}\Bigr)
                  \int_0^h s_i(z)^2\,dz .
\]
The target-label law has a singular part from the plateau and an absolutely continuous part where $A'\ne0$. \autoref{fig:coupled-optimal}(a) shows the displaced profiles, and panel (b) shows their common subgradient segment at a plateau label.

\begin{figure}[!htbp]
\centering
\begin{tikzpicture}[x=1cm,y=1cm]
 \node[figPanel] at (2.47,3.21) {(a) Two equal-law profiles};
 \node[figPanel] at (8.25,3.21) {(b) One subgradient segment};
 \begin{scope}[x=4.3cm,y=1.42cm,shift={(.04,.37)}]
   \fill[figTeal!12] (.20,0) rectangle (.80,1.24);
   \draw[figInk!48,dashed] (.20,0)--(.20,1.24) (.80,0)--(.80,1.24);
   \draw[figArrow] (0,0)--(1.10,0) node[right] {$z/h$};
   \draw[figArrow] (0,0)--(0,1.47) node[above] {$B_i$};
   \draw[figBlue,line width=1.05pt,domain=-.98:.98,samples=81,variable=\q]
      plot({.36+.112*\q},{exp(1-1/(1-\q*\q))});
   \draw[figOrange,densely dashed,line width=1pt,domain=-.98:.98,samples=81,variable=\q]
      plot({.64+.112*\q},{exp(1-1/(1-\q*\q))});
   \node[figBlue,above] at (.36,1) {$B_0$};
   \node[figOrange,above] at (.64,1) {$B_1$};
   \draw[figInk,<->,>=Latex] (.20,-.14)--(.80,-.14) node[midway,below] {$A=0$};
 \end{scope}
 \coordinate (minus) at (6.30,.42);
 \coordinate (zero) at (8.10,1.32);
 \coordinate (plus) at (9.90,2.22);
 \draw[figInk!64,line width=.85pt] (minus)--(plus);
 \draw[figBlue,line width=1.3pt] (minus)--(zero);
 \fill[figInk] (minus) circle[radius=1.85pt] node[below left] {$Z^{[1]}$};
 \fill[figInk] (zero) circle[radius=1.85pt] node[below right] {$Z_*$};
 \fill[figInk] (plus) circle[radius=1.85pt] node[above right] {$Z_*+\varepsilon\nabla G$};
 \fill[figTeal] (7.31,.925) circle[radius=2pt] node[above left=1pt and 3pt] {$Z^{[0]}$};
 \node[align=center] at (8.56,-.22) {$f\Delta v=(\Delta\rho)\Psi_x,\qquad \Psi_z=0$};
\end{tikzpicture}
\caption{Coupled nonuniqueness in \eqref{eq:coupled-optimal-counterexample}. (a) One choice of the two smooth translates on the plateau, with $s_i=-1+B_i$. (b) At a fixed target label on the plateau, $\partial\varphi(Z_*)$ is the full segment and the constructed labels use its highlighted half, since $s_i\in[-1,0]$. The labels shown correspond to a point on the shoulder of the first bump, where $0<B_0<1$ and $B_1=0$. The segment direction changes both labels while leaving the force unchanged. Segment orientation is schematic.}
\label{fig:coupled-optimal}
\end{figure}

\paragraph{Pressure uniqueness in this class.}
More generally, for compactly supported exact joint laws a single optimal Lipschitz convex potential can be chosen on each half cell for every exact-law minimizer. Indeed, every optimal coupling saturates the same dual inequality almost everywhere. At points where the pulled-back potential is differentiable, \eqref{eq:optimal-pressure-chain} therefore gives the same force for every such minimizer. Writing differences of the reconstructed fields with $\Delta$, one obtains
\[
 (\Delta\rho)\,\partial_z\Psi=0,\qquad
 f\,\Delta v=(\Delta\rho)\,\partial_x\Psi .
\]
Thus the freedom in \eqref{eq:coupled-optimal-counterexample} lies where the target coordinates lose rank and the two scalar contributions to the force cancel.

\paragraph{Example of a purely joint obstruction.}
Let $\Psi=\cos x$, $a_0=a_1=0$ and $\rho_s=\beta-\gamma\cos x$. Choose total initial density $R_0=\beta+\gamma\cos x$ and $v_0=0$. On each half cell the initial density and the target density have the same distribution, by the reflection of $x$ about the midpoint of that half interval. The shifted meridional marginal is also the same, since $v_0+v_s+fx=fx$. The joint laws differ, so \autoref{prop:convex-order}(3) rules out even weak convergence to this target. In particular, the target is supported on the graph of
\[
 G_*(b)=\beta-\gamma\cos\bigl(b/(\sqrt\gamma f)\bigr)
\]
in the coordinates $(G_*(b),b)$. The Lipschitz constant of $G_*$ is at most $\sqrt\gamma/f$. For any $(r,b)$ and any point $(G_*(c),c)$ on the graph, the squared distance is at least
\[
 \frac{|r-G_*(b)|^2}{1+\gamma/f^2}.
\]
This follows from $|r-G_*(b)|\le |r-G_*(c)|+(\sqrt\gamma/f)|c-b|$ and Cauchy--Schwarz. Applying the lower bound to every coupling in \eqref{eq:joint-lower-bound} gives $C_{\gamma,\beta}\ge 2\gamma^2|\Omega|/(1+\gamma/f^2)>0$ for the global solution with these data. Taking $\mathbf u_0=0$ gives an explicit steady solution: $-\rho\nabla\Psi$ is a periodic gradient because $\rho=2\gamma\cos x$ and $\Psi=\cos x$. Its actual limiting defect is $2\gamma^2|\Omega|$.

In the normalized coordinates of \autoref{fig:joint-obstruction-graphs}, reflection about $q=1/2$ interchanges the two support curves and preserves the law of $q$. The distinct density values at fixed $q$ exhibit the joint obstruction directly.

\begin{figure}[!htbp]
\centering
\begin{tikzpicture}[x=7.18cm,y=1.24cm]
 \draw[figArrow] (0,-1.1)--(0,1.17) node[above] {$r$};
 \draw[figArrow] (0,0)--(1.10,0) node[right] {$q$};
 \draw[figInk!28,densely dotted] (0,1)--(1.01,1) (0,-1)--(1.01,-1);
 \draw[figBlue,line width=.96pt,domain=0:1,samples=91] plot(\x,{cos(180*\x)});
 \draw[figOrange,line width=1pt,densely dashed,domain=0:1,samples=91] plot(\x,{-cos(180*\x)});
 \foreach \y/\val in {1/1,-1/{-1}}{\node[left=2pt] at (0,\y) {$\val$};}
 \node[below=2pt] at (1,0) {$1$};
 \node[figBlue,above] at (.24,.83) {initial};
 \node[figOrange,above] at (.77,.82) {target};
 \draw[figInk!85,<->,>=Latex] ({1/5},{-cos(36)})--({1/5},{cos(36)});
\end{tikzpicture}
\caption{Exact support curves for the purely joint obstruction example on $\Omega_+$. In these coordinates $q=x/\pi$ and $r=(R-\beta)/\gamma$, where $R$ is the total density. The initial and target laws are generated by the same uniform $q$ and by opposite cosines. Their one-dimensional marginals agree, while at a fixed $q\ne1/2$ their density labels differ. The vertical arrow shows this difference at $q=1/5$.}
\label{fig:joint-obstruction-graphs}
\end{figure}

\FloatBarrier

\subsection{Energy stability and nonlinear escape}\label{sec:stability-proof}

\subsubsection{Mechanical energy on the joint rearrangement class}

\begin{proof}[Proof of \autoref{prop:joint-second-variation}]
Write $D_{\mathbf W}=\mathbf W\cdot\nabla$ and $M_s=v_s+fx$. Each $X_\eta$ preserves measure and leaves each half cell invariant, so $\|M_s\circ X_{-\eta}\|_2^2$ is constant. Expanding the square in $\mathcal P(\eta)$ gives
\[
 \mathcal P(\eta)=
 -f\int_\Omega x(M_s\circ X_{-\eta})\,dx\,dz
 +\int_\Omega\Psi(\rho_s\circ X_{-\eta})\,dx\,dz+\text{constant}.
\]
At zero, the derivatives of the two pullbacks are $-D_{\mathbf W}$ and $D_{\mathbf W}^2$ applied to the initial fields. These differentiations are valid on each closed half cell for the stated regularity. The first derivative is
\[
 f(f+a_0)\int_\Omega xW_1\,dx\,dz
   -\int_\Omega \Psi\,D_{\mathbf W}\rho_s\,dx\,dz=0.
\]
The first integral vanishes by testing with $x^2/2$. The integral $\int_\Omega\Psi D_{\mathbf W}\rho_s$ equals $-\int_\Omega\rho_s D_{\mathbf W}\Psi$ and vanishes by the background hydrostatic balance. All these integrations by parts use the zero flux on the half-cell boundaries. For the second derivative,
\[
\begin{aligned}
 \mathcal P''(0)
 &= -f\int_\Omega x D_{\mathbf W}^2 M_s\,dx\,dz
      +\int_\Omega\Psi D_{\mathbf W}^2\rho_s\,dx\,dz\\
 &= f\int_\Omega (D_{\mathbf W}x)(D_{\mathbf W} M_s)\,dx\,dz
      -\int_\Omega(D_{\mathbf W}\Psi)(D_{\mathbf W}\rho_s)\,dx\,dz\\
 &= f(f+a_0)\int_\Omega W_1^2\,dx\,dz
    -\int_\Omega\delta|\mathbf W\cdot\nabla\Psi|^2\,dx\,dz .
\end{aligned}
\]
This is \eqref{eq:joint-second-variation}. The sufficiency of a negative second variation follows from \autoref{thm_linear_stability}. Conversely the oscillatory trial fields in the proof of the pointwise criterion are smooth, so whenever a growing mode exists they provide smooth directions with positive drive.
\end{proof}

\subsubsection{Nonlinear stability for decreasing density laws}

\begin{proof}[Proof of \autoref{prop:nonlinear-stability}]
It suffices to restrict $I$ to the compact range $[\min\Psi,\max\Psi]$. On the image of this interval under $\mathfrak r$ take $G'=-\mathfrak r^{-1}$ and extend $G''$ as a constant beyond each endpoint before integrating. If $\Psi$ is constant, take any constant $G''\in[M^{-1},m^{-1}]$ and fix $G'(\rho_s)=-\Psi$. The bounds on $G''$ give
\[
 \tfrac{1}{2M}|\rho|^2
 \le G(R)-G(\rho_s)-G'(\rho_s)(R-\rho_s)
 \le \tfrac{1}{2m}|\rho|^2 .
\]
Using $G'(\rho_s)=-\Psi$ and putting $M=v_s+v+fx$ and $M_s=v_s+fx$, we obtain
\[
 \mathcal B-(E_{\rm mech}-E_s)
   =\int_\Omega\bigl(G(R)-G(\rho_s)\bigr)\,dx\,dz
      -\frac{a_0}{2(f+a_0)}\bigl(\|M\|_2^2-\|M_s\|_2^2\bigr).
\]
Indeed, the meridional terms agree after expanding
$\|M\|_2^2-\|M_s\|_2^2=\|v\|_2^2+2(f+a_0)\int_\Omega xv$, since $v$ is odd. Both differences on the right are constant by \autoref{lem_transport_lp}. The mechanical dissipation identity from \autoref{lem_relative_energy} thus gives \eqref{eq:nonlinear-lyapunov}. The quadratic bounds give Lyapunov stability. On the reference joint rearrangement class the two differences vanish, proving \eqref{eq:lyapunov-mechanical} and the strict minimum.
\end{proof}

\subsubsection{Nonlinear escape and mechanical energy release}

\begin{proof}[Proof of \autoref{thm:nonlinear-escape}]
Use the unweighted product Hilbert space of the linear instability proof, with generator $\mathcal L_{\rm u}$. Since $-\nu A$ generates an analytic semigroup and the coupling is bounded, $\mathcal L_{\rm u}$ generates an analytic semigroup $S_{\rm u}(t)$ by the bounded perturbation theorem. The spectral characterization in \autoref{thm_linear_stability} gives spectral bound $\lambda_*$. The spectral mapping theorem for analytic semigroups \cite[Corollary~IV.3.12]{EngelNagel2000} then yields
\begin{equation}\label{eq:unstable-semigroup}
 \|S_{\rm u}(t)\|\le C_\sigma e^{\sigma t},
 \qquad t\ge0,\qquad \lambda_*<\sigma<2\lambda_* .
\end{equation}
The velocity of the growing mode lies in $D(A)$ and is nonzero. The scalar reconstruction \eqref{eq_linear_mode_scalars} gives $\rho_*,v_*\in H^1$, since the source is expressed through the smooth background gradient. Applying the $H^1$ stationary Stokes estimate to the mode equation yields $\mathbf u_*\in H^3$, hence the flow $X_s^*$ is well defined and $C^1$. It preserves volume and commutes with reflection. In particular, the scalar perturbations in \eqref{eq:escape-invariant-data} have the required parity. The meridional perturbation equals $(f+a_0)(X_{-\epsilon,1}^*-x)$, where the horizontal flow coordinate is lifted to the representative strip. It is periodic because the flow displacement is periodic. This also gives $H^1$ across the vertical seams.

Let $Z_s=(\rho_s,v_s+fx)$ and $\mathbf W=\mathbf u_*/\lambda_*$. Since
$-D_{\mathbf W}Z_s=(\rho_*,v_*)$, differentiation of the pullback in \eqref{eq:escape-invariant-data} gives
\begin{equation}\label{eq:escape-initial-expansion}
 \|Y^\epsilon(0)-\epsilon Y_*\|_2\le C\epsilon^2,
 \qquad \|(\rho^\epsilon(0),v^\epsilon(0))\|_{H^1}\le C\epsilon .
\end{equation}
To estimate the quadratic remainder, use $D_{\mathbf W}^2 Z_s\in L^2$ on each half cell, and for the $H^1$ bound integrate the first derivative $-(D_{\mathbf W} Z_s)\circ X_{-s}^*$ using the bi-Lipschitz bounds. Thus \autoref{thm_nonlinear_regularity} applies with $q=4$, and its finite-time Sobolev bounds justify the estimates below.

\medskip

\noindent\textit{Estimates while the disturbance is small.}
Write $Y=(\mathbf u,\rho,v)=Y^\epsilon$ in this proof and put $d(t)=\epsilon e^{\lambda_*t}$. Work on an interval $[0,T]$ on which $d(t)\le\theta\le1$ and $\|Y(t)\|_2\le2d(t)$. Constants denoted by $C$ are independent of $\epsilon$, $T$ and $\theta\le1$. 

We test the three perturbation equations with $\mathbf u,\rho,v$, respectively. Convection has zero pairing by incompressibility and impermeability. All linear couplings are bounded in the product $L^2$ norm. Coercivity of the Stokes form then gives
\begin{equation}\label{eq:escape-energy-estimate}
 \int_0^t\|\mathbf u(s)\|_{H^1}^2\,ds
 \le C\left(\epsilon^2+\int_0^t\|Y(s)\|_2^2\,ds\right)
 \le C d(t)^2 .
\end{equation}
The uniform $L^4$ bounds of the scalars for this family follow from \autoref{lem_transport_lp} and the uniformly bounded initial data. Therefore the constant in \eqref{eq_nonlinear_acceleration_energy} is uniform in $\epsilon$. The projected momentum equation at $t=0$ gives $\|\mathbf u_t(0)\|_2\le C\epsilon$ because $Y_*$ is fixed in $D(A)\times H^1\times H^1$. By \eqref{eq:escape-energy-estimate}, Gronwall's inequality in \eqref{eq_nonlinear_acceleration_energy} and integration give
\begin{equation}\label{eq:escape-acceleration}
 \|\mathbf u_t(t)\|_2^2+\int_0^t\|\mathbf u_t(s)\|_{H^1}^2\,ds
 \le C d(t)^2 .
\end{equation}
Gronwall's inequality introduces an exponential factor bounded by $\exp(Cd(t)^2)\le \exp(C)$.

Since $\|\mathbf F\|_2+\|\mathbf u_t\|_2+\|\mathbf u\|_2\le C d(t)$, \eqref{eq:pointwise-velocity} gives $\|\mathbf u\|_{H^1}\le Cd(t)$ and an $H^2$ bound by $C(d(t)+d(t)^3)$. Here $d(t)\le1$. It follows that $\|\mathbf u(t)\|_{H^2}\le C d(t)$.

The inhomogeneous perturbation transport equations along the measure-preserving flow imply
\[
 \|(\rho,v)(t)\|_{L^4}
 \le C\epsilon+C\int_0^t\|\mathbf u(s)\|_{L^4}\,ds
 \le C d(t).
\]
Their source terms are products of the planar velocity with fixed bounded coefficients. Thus the $L^4$ Stokes estimate gives, almost everywhere,
\[
 \|\nabla\mathbf u\|_{L^\infty}
 \le C\bigl(\|\mathbf u_t\|_{H^1}+d(t)\bigr),
\]
where convection is at most $C\|\mathbf u\|_{H^2}^2$ in $L^4$. To integrate the acceleration term in this bound, divide $[0,t]$ into intervals of length at most one, counting backwards from $t$. On the interval with right endpoint $t-j$, $j=0,1,\dots,\lfloor t\rfloor$, Cauchy--Schwarz and \eqref{eq:escape-acceleration} bound its integral by $C d(t-j)$. The sum is geometric. Consequently
\begin{equation}\label{eq:escape-strain}
 \int_0^t\|\nabla\mathbf u(s)\|_{L^\infty}\,ds\le C d(t).
\end{equation}
Finally, differentiating the scalar transport equations yields an $H^1$ bound for $(\rho,v)$. Their source gradients are bounded in $L^2$ by $C\|\mathbf u\|_{H^1}$, and their gradient stretching term is bounded by $\|\nabla\mathbf u\|_{L^\infty}\|(\nabla\rho,\nabla v)\|_2$. Integrating with \eqref{eq:escape-strain} yields
\[
 \|(\rho,v)(t)\|_{H^1}
 \le C e^{Cd(t)}\left(\epsilon+\int_0^t\|\mathbf u(s)\|_{H^1}\,ds\right)
 \le C d(t).
\]
These transport estimates can also be obtained from the composition formulas for the inhomogeneous scalar equations and their $H^1$ derivatives. In particular, they apply at the regularity of \autoref{thm_nonlinear_regularity}.

\medskip
\noindent\textit{Comparison with the growing mode.}
The nonlinear term in the projected system is
\[
 \mathcal N(Y)=
 -\bigl(\mathbb P(\mathbf u\cdot\nabla\mathbf u),
        \mathbf u\cdot\nabla\rho,\mathbf u\cdot\nabla v\bigr).
\]
The preceding estimates give $\|\mathcal N(Y(t))\|_2\le C d(t)^2$. The variation-of-constants formula in the product $L^2$ space and \eqref{eq:unstable-semigroup} imply
\[
 \|Y(t)-d(t)Y_*\|_2
 \le C\epsilon^2 e^{\sigma t}+C\int_0^t e^{\sigma(t-s)}d(s)^2\,ds
 \le C d(t)^2 .
\]
Choose $\theta$ so small that $C\theta<\frac12\min(1,\|\mathbf u_*\|_2)$. The resulting estimate $\|Y(t)\|_2\le\frac32d(t)$ improves the bootstrap bound. By continuity and global existence, the bound holds until $T_\epsilon$, proving \eqref{eq:escape-approximation}. It also gives
\[
 \|\mathbf u^\epsilon(t)\|_2\ge \tfrac12\|\mathbf u_*\|_2\,d(t)
 \qquad(0\le t\le T_\epsilon).
\]
At $t=T_\epsilon$ this yields the stated escape. Moreover, the first time at which the velocity norm reaches $\theta\|\mathbf u_*\|_2/4$ differs from $\lambda_*^{-1}\log(1/\epsilon)$ by an amount bounded independently of $\epsilon$, by the upper bound $\|Y(t)\|_2\le2d(t)$ and the lower bound at $T_\epsilon$.

\medskip
\noindent\textit{Energy release and the limiting states.}
Apply \autoref{prop:joint-second-variation} with $\mathbf W=\mathbf u_*/\lambda_*$. Including the initial planar kinetic energy gives
\[
 E_{\rm mech}^\epsilon(0)
 = E_s+\frac{\epsilon^2}{2}
       \left(\|\mathbf u_*\|_2^2-\lambda_*^{-2}E_2(\mathbf u_*)\right)
       +o(\epsilon^2).
\]
The eigenmode energy identity \eqref{eq_linear_mode_identity} yields \eqref{eq:escape-initial-energy}. By coercivity, the velocity lower bound on $[T_\epsilon-1,T_\epsilon]$ implies
\[
 \nu\int_0^{T_\epsilon}\Diss(\mathbf u^\epsilon(s))\,ds
 \ge c\theta^2\int_0^1 e^{-2\lambda_* r}\,dr
\]
when $T_\epsilon\ge1$. The right-hand side is strictly positive and independent of $\epsilon$. As explained in \autoref{sec:transport}, the relative energy identity \eqref{eq_relative_energy} is equivalent to
\[
 E_{\rm mech}^\epsilon(t)
       +\nu\int_0^t\Diss(\mathbf u^\epsilon(s))\,ds
       =E_{\rm mech}^\epsilon(0).
\]
For small $\epsilon$, \eqref{eq:escape-initial-energy} gives $E_{\rm mech}^\epsilon(0)\le E_s$. Choosing $2d_*$ no larger than the uniform dissipation lower bound proves \eqref{eq:energy-release}. For a weak scalar accumulation point, the potential-energy term passes to the limit by weak convergence and the squared meridional norm is weakly lower semicontinuous. This gives \eqref{eq:energy-limit-gap}; the limiting state is balanced by \autoref{thm_nonlinear_large_time}. Finally the scalar norms remain bounded uniformly in $\epsilon$ by transport. The functional in \eqref{eq:energy-limit-gap} is Lipschitz on the resulting bounded product $L^2$ ball, so the energy gap gives a uniform positive distance from the original equilibrium.
\end{proof}

\FloatBarrier

\subsubsection{Nonlinear classification on the equilibrium rearrangement class}

\begin{proof}[Proof of \autoref{thm:vertical-nonlinear-dichotomy}]
The pointwise matrix in \eqref{eq:pointwise-drive} is diagonal with entries $-f\beta_f$ and $c(z)$. Hence part (1) follows from \autoref{thm:nonlinear-escape}. Assume now that $\beta_f\ge0$ and $c\le0$.

Because $\partial_x\rho_s=0$ and $\rho_s'=\delta\Psi'$, monotonicity of $\Psi$ gives
\[
 |\rho_s(z_2)-\rho_s(z_1)|
 \le L|\Psi(z_2)-\Psi(z_1)| .
\]
Also, wherever $\Psi'\ne0$, $\rho_s'$ has the opposite weak sign to $\Psi'$. Thus there is a nonincreasing, $L$-Lipschitz law $\mathfrak r$ on the range of $\Psi$ with $\rho_s=\mathfrak r(\Psi)$, including when $\Psi$ has plateaus. Extend $\mathfrak r$ by constants outside that range and define the convex function
\[
 H(x)=-\int_0^x\mathfrak r(y)\,dy,\qquad
 H^*(z)=\sup_{x\in\mathbb R}\{zx-H(x)\}.
\]
The derivative $H'$ is $L$-Lipschitz. Applying \eqref{eq:fenchel-leaf} with $y=\Psi$ and $w=-R$, where $R=\rho+\rho_s$, and subtracting Fenchel equality at $w=-\rho_s$, gives
\[
 H^*(-R)-H^*(-\rho_s)+\Psi\rho\ge \frac{|\rho|^2}{2L}.
\]
The conjugate is bounded on the slope range of $H$, since $H$ is affine outside a compact interval. The first two terms cancel on integration by equality of the density laws. Hence $\int_\Omega\Psi\rho\ge\|\rho\|_2^2/(2L)$.

For $\beta_f>0$, conservation of the reference second moment of $M=v+v_s+fx$ gives $\|v\|_2^2+2\beta_f\int_\Omega xv=0$, since $v$ is odd. Expanding the kinetic-energy difference and eliminating $\int_\Omega xv$ yields
\[
 \tfrac12\|v+v_s\|_2^2-\tfrac12\|v_s\|_2^2
 =\tfrac12\|v\|_2^2+a_0\int_\Omega xv
 =\tfrac{f}{2\beta_f}\|v\|_2^2.
\] If $\beta_f=0$, the equilibrium shifted meridional variable is the constant $a_1$, and equality of distributions forces $v\equiv0$. Adding the planar kinetic energy proves the lower bounds and strict minimality.

The mechanical energy is nonincreasing by \eqref{eq_relative_energy}. At $t=0$ its difference from $E_s$ tends to zero with the product $L^2$ norm because the kinetic differences are quadratic in that norm plus bounded linear terms, and the potential-energy difference is linear. The bounds just proved therefore give Lyapunov stability for all time in the specified class.
\end{proof}

\subsection{Rigidity and equilibrium selection}\label{sec:rigidity-proof}

\begin{proof}[Proof of \autoref{thm:rigid-balance}]
A vertical reflection in the following stationary argument allows us to take $\Psi'\ge0$. The argument uses the balance and the side-preserving map on a rectangle. In particular, it also applies with the vertical coordinate reversed. By the proof of \autoref{thm:vertical-nonlinear-dichotomy}, $\rho_s=\mathfrak r(\Psi)$ for a nonincreasing Lipschitz function $\mathfrak r$, and $\rho_s$ is nonincreasing in $z$. Let $Q$ be either closed half cell and write $S(x,z)=(\xi(x,z),\eta(x,z))$ on $Q$.

\medskip
\noindent\textit{Positive effective Coriolis coupling.}
Let $b=f\beta_f>0$ and put $G(x,z)=-\rho_s(\eta(x,z))$. The balance and the representation through $S$ give
\begin{equation}\label{eq:rigid-gradient}
  p_x=b(\xi-x),\qquad p_z=\Psi'(z)(G+\rho_s(z)).
\end{equation}
Consequently,
\[
 b\,\xi_z=\Psi'(z)G_x=\Psi'(z)g\eta_x,\qquad
 g(x,z)=-\rho_s'(\eta(x,z))\ge0 .
\]
The Jacobian identity $\det DS=1$ gives
\begin{equation}\label{eq:rigid-jacobian}
 \xi_x\eta_z=1+b^{-1}\Psi'(z)g\eta_x^2 .
\end{equation}
Along the left vertical side of $Q$, $\xi$ is constant in $z$ and equals the left endpoint of the horizontal interval. The derivative $\xi_x$ there is nonnegative because the diffeomorphism maps the rectangle into itself. It is nonzero by \eqref{eq:rigid-jacobian}. By continuity, connectedness and the same identity, $\xi_x>0$ and $\eta_z>0$ on all of $Q$.

We next establish the optimality property of the coordinate pair
\[
 Y=(x,\Psi(z)),\qquad T=(b\xi,G).
\]
The coordinate $Y$ may collapse vertical intervals. For $\varepsilon>0$ use $Y_\varepsilon=(x,\Psi(z)+\varepsilon z)$, a smooth coordinate on the rectangle, and view $T_\varepsilon(Y_\varepsilon)=T(x,z)$ as a vector field on its image rectangle. Write $r_\varepsilon=\Psi'/(\Psi'+\varepsilon)$. The symmetric part of $D_{Y_\varepsilon}T_\varepsilon$ has diagonal entries
$b\xi_x$ and $g\eta_z/(\Psi'+\varepsilon)$ and off-diagonal entry
$g\eta_x(1+r_\varepsilon)/2$. Its determinant, by \eqref{eq:rigid-jacobian}, is
\begin{equation}\label{eq:rigid-monotonicity}
 \frac{b g}{\Psi'+\varepsilon}
   -\frac{g^2\eta_x^2}{4}(1-r_\varepsilon)^2
 =\frac{g}{\Psi'+\varepsilon}
   \left(b-\frac{g\eta_x^2\varepsilon^2}{4(\Psi'+\varepsilon)}\right).
\end{equation}
The last bracket is positive for sufficiently small $\varepsilon$, since $g\eta_x^2$ is bounded and $\Psi'\ge0$. Thus $T_\varepsilon$ is monotone on the rectangle, by integrating the derivative pairing along line segments.

To preserve the gradient information in \eqref{eq:rigid-gradient}, define
\[
 P(x,z)=p(x,z)+\tfrac b2 x^2-\int_0^z \rho_s(r)\Psi'(r)\,dr,\qquad
 P_\varepsilon(x,z)=P(x,z)+\varepsilon\int_0^zG(x,r)\,dr.
\]
Then $P_x=b\xi$, $P_z=\Psi'G$, and the gradient of $P_\varepsilon$ in the $Y_\varepsilon$ coordinates is $T_\varepsilon$ plus a horizontal error uniformly bounded by $C\varepsilon$. For points $X,X'\in Q$, integrate along the segment from $Y_\varepsilon(X)$ to $Y_\varepsilon(X')$ and use monotonicity to obtain
\[
 P_\varepsilon(X')-P_\varepsilon(X)
 \ge T(X)\cdot\big(Y_\varepsilon(X')-Y_\varepsilon(X)\big)
       - C\varepsilon|Y_\varepsilon(X')-Y_\varepsilon(X)|.
\]
Since $P_z=\Psi'G$, the function $P$ is constant on every vertical interval collapsed by $Y$ at fixed $x$. It defines a continuous function $\Phi$ on the closed rectangle $Y(Q)$, because $Y$ is a continuous surjection from the compact rectangle $Q$. This includes the case of zero height. Letting $\varepsilon\downarrow0$ gives a supporting affine function at every point:
\begin{equation}\label{eq:rigid-support}
 \Phi(Y(X'))\ \ge\ \Phi(Y(X))+T(X)\cdot(Y(X')-Y(X)).
\end{equation}
In particular $\Phi$ is convex on $Y(Q)$, and, extended by $+\infty$ off this rectangle, it has $T(X)\in\partial\Phi(Y(X))$. This reasoning accommodates a potential with intervals of constancy.

The background target is $T_0(X)=(b x,-\rho_s(z))$. The targets $T=T_0\circ S$ and $T_0$ have the same joint distribution on $Q$. By Fenchel's inequality and its equality in \eqref{eq:rigid-support},
\[
 \int_Q Y\cdot T_0
 \le\int_Q\big(\Phi(Y)+\Phi^*(T_0)\big)
 =\int_Q\big(\Phi(Y)+\Phi^*(T)\big)=\int_Q Y\cdot T .
\]
Conversely, the differentiable convex function
\[
 \Phi_0(x,y)=\tfrac b2 x^2-\int_0^y\mathfrak r(s)\,ds
\]
has $\nabla\Phi_0(Y)=T_0$. Here $\mathfrak r$ is extended by constants as in \autoref{thm:vertical-nonlinear-dichotomy}. Applying Fenchel's inequality with $\Phi_0$ gives the reverse inequality. The integrals are finite: for $\Phi$ the domain is compact, and for $\Phi_0^*$ the targets lie in the gradient image of the compact source rectangle, where the conjugate is bounded by Fenchel equality. We conclude that $T\in\partial\Phi_0(Y)$ almost everywhere. Differentiability gives $T=T_0$, first almost everywhere and then everywhere by continuity. Thus $\xi=x$, $\rho_s(\eta)=\rho_s(z)$ and $\rho=v=0$. Moreover, $\det DS=1$ now gives $\eta_z=1$, and $\eta(x,0)=0$ implies $S=\mathrm{id}$.

\medskip

\noindent\textit{Zero effective Coriolis coupling.}
When $\beta_f=0$, the transported meridional quantity is constant for a rearranged equilibrium, so $v=0$. Write $R=\rho_s\circ\eta$. At every height $z_0\in(0,h)$ with $\Psi'(z_0)>0$, the curl of the force balance gives $R_x=0$, so $R(x,z_0)=r_0$ for all $x$. The set $\{R<r_0\}$ is the preimage under $S$ of an upper horizontal substrip of $Q$. When nonempty, it is path connected up to the top side and cannot cross the horizontal section at $z_0$. It must lie above that section. The same argument for $\{R>r_0\}$ places it below. Hence any two points with strictly ordered values of $\Psi$ have oppositely ordered values of $R$, by choosing a height between them at which $\Psi'>0$.

Since $\rho_s=\mathfrak r(\Psi)$, for each level $r$ membership in $\{\rho_s>r\}$ depends only on the potential. Suppose $\{R>r\}$ and $\{\rho_s>r\}$ differ on a set of positive measure. Equimeasurability gives points in both differences. The point in $\{\rho_s>r\}\setminus\{R>r\}$ has strictly smaller potential but strictly smaller $R$ than the point in $\{R>r\}\setminus\{\rho_s>r\}$, contradicting the preceding order property. Taking rational $r$ gives $R=\rho_s$. If $\Psi$ is constant, $\rho_s$ is constant by compatibility, so the same conclusion follows directly. This proves part (1).

\medskip
\noindent\textit{Selection along a trajectory.}
Pass to a common subsequence of times along which the flow maps converge in $C^1$ on both half cells, to $X_\infty$. Their determinants are one, and the derivative bounds along the subsequence give uniform derivative bounds for the inverses. Since the target half cells are convex, the inverse Lipschitz constants are uniformly bounded. Thus $X_\infty$ is invertible and its inverses along the subsequence converge in $C^1$ by the inverse derivative formula. Each map preserves the four sides, so the same holds for the limit. The pullback $S_\infty=S_0\circ X_\infty^{-1}$ satisfies the hypotheses of part (1): its scalar pair is a strong accumulation point by transport and is balanced by \autoref{thm_nonlinear_large_time}. Part (1) identifies this pair with $Z_s$.

Along this subsequence the velocity also tends to zero, so $E_{\rm mech}(t_j)\to E_s$. Monotonicity identifies this value as the limit of the mechanical energy along the full trajectory. The scalar coercivity in \autoref{thm:vertical-nonlinear-dichotomy} now yields convergence of $\rho,v$ to zero in $L^2$ for all times tending to infinity. Both scalars are uniformly bounded by \autoref{lem_transport_lp}, since the initial pair is bounded. Interpolation gives convergence in every finite $L^q$. The $H^2$ and $H^1$ convergence of velocity and normalized pressure follows by the Stokes argument in \autoref{prop:limit-set}.

For part (3), finite strain gives $C^1$ convergence of the full trajectory of flow maps and inverses by \autoref{prop:limit-set}. The transport representation then gives the scalar $C^1$ convergence. If $\beta_f>0$, $S_\infty=\mathrm{id}$ by part (1), so $X_\infty=S_0$. Estimate \eqref{eq:strain-tail} applies for each finite $q$ because the initial rearranged scalars are $C^1$ on each closed half cell.

For part (4), take a subsequence on which both maps and inverses converge uniformly on both closed half cells, by the Arzela--Ascoli theorem. Passing to the limit in both inverse identities gives mutually inverse, side-preserving homeomorphisms. Testing continuous functions shows that they preserve area, so the resulting limit label $S_\infty$ is also an area-preserving homeomorphism. The scalar pairs converge uniformly on each half cell along this sequence and their limit is balanced by \autoref{thm_nonlinear_large_time}. \autoref{lem:rectangular-force} applies to a continuous balanced label homeomorphism and gives $S_\infty=\mathrm{id}$. The mechanical-energy argument in part (2) then gives the asserted convergence along the full trajectory.
\end{proof}

\begin{proof}[Proof of \autoref{prop:folded-equilibria}]
For the vertical construction, choose an interval $J=(a,a+\ell)\Subset(0,h)$ on which $\rho_s'\ne0$. On any interval $J=(a,a+\ell)$ the fold is the identity outside $J$ and is given inside by
\[
 s_J(a+\ell y)=a+\ell\sigma(y),\qquad
 \sigma(y)=
 \begin{cases}
  3y,&0\le y\le1/3,\\
  2-3y,&1/3\le y\le2/3,\\
  3y-2,&2/3\le y\le1 .
 \end{cases}
\]
The three affine branches each cover $[0,1]$ with slope of magnitude three, so $s_J$ preserves Lebesgue measure. The three preimages in \autoref{fig:fold}(a) each have inverse slope of magnitude $1/3$, giving the three measure contributions shown in panel (b). For nonconstant $\rho_s$, take $\rho=\rho_s(s_J(z))-\rho_s(z)$ and $v=0$. This is a stationary perturbation with $\mathbf u=0$ and pressure gradient $(0,-\rho\Psi')$. The joint distribution is preserved on each half cell by the product form. The perturbation is nonzero because $\rho_s$ is strictly monotone on $J$, and its norms satisfy $\|\rho\|_{L^\infty}\le C\ell$ and $\|\rho\|_{W^{1,r}}\le C_r\ell^{1/r}$.

\begin{figure}[!htbp]
\centering
\begin{tikzpicture}[x=3.0cm,y=2.73cm]
  \begin{scope}
    \node[figPanel] at (.55,1.40) {(a) The normalized fold};
    \draw[figArrow] (0,0)--(1.15,0) node[right] {$y$};
    \draw[figArrow] (0,0)--(0,1.10) node[above] {$\sigma(y)$};
    \draw[figInk!46,densely dotted] (0,0)--(1,1);
    \def\level{.44}
    \draw[figTeal,line width=1.15pt] (0,0)--({1/3},1)--({2/3},0)--(1,1);
    \draw[figOrange,dashed] (0,\level)--(1.02,\level);
    \foreach \preimage in {\level/3,(2-\level)/3,(2+\level)/3}{
      \draw[figInk!52,dotted] ({\preimage},0)--({\preimage},\level);
      \fill[figOrange] ({\preimage},\level) circle[radius=.023];
    }
    \foreach \x in {1,2}{
      \draw ({\x/3},.018)--({\x/3},-.018) node[below] {$\x/3$};
    }
    \draw (1,.018)--(1,-.018) node[below] {$1$};
    \node[left] at (0,\level) {$r$};
    \node[left] at (0,1) {$1$};
  \end{scope}
  \begin{scope}[shift={(1.87,0)}]
    \node[figPanel] at (.63,1.40) {(b) Pushforward measure};
    \foreach \y/\lab/\col in {.88/1/figBlue,.62/2/figOrange,.36/3/figTeal}{
      \draw[\col,line width=2.4pt] (0,\y)--(.32,\y);
      \node[left,\col] at (0,\y) {\lab};
      \draw[figArrow,draw=\col] (.39,\y)--(.77,.62);
    }
    \draw[figInk,line width=1.5pt] (.82,.62)--(1.33,.62);
    \node[below=6pt] at (1.09,.62) {$[0,1]$};
    \node[below,align=center] at (.63,.25) {$\tfrac13+\tfrac13+\tfrac13=1$};
  \end{scope}
\end{tikzpicture}
\caption{A stationary rearrangement outside the diffeomorphism class. The three branches of $\sigma$ have slopes $3,-3,3$. Each contributes one third of Lebesgue measure to the pushed-forward measure, so the reference joint law is preserved when the fold is applied in a single coordinate.}
\label{fig:fold}
\end{figure}

If $\beta_f>0$, one can also construct a horizontal fold. Choose a short interval in $(0,\pi)$, apply the same fold to the horizontal coordinate there, and extend the resulting map to $(-\pi,\pi)$ by odd reflection and by the identity outside the two folded intervals. Call the map $s_h$. Take $\rho=0$ and $v=\beta_f(s_h(x)-x)$. Then the initial label pair is $(\rho_s,\beta_f s_h+a_1)$ and has the required joint distribution on each half cell. The pressure is an even periodic primitive of $fv$ in $x$, so this is stationary with zero planar velocity. The same estimates with the interval length $\ell$ hold for $v$. These two constructions cover every nonconstant reference pair.

Finally, if $\rho_s$ is constant and $\beta_f=0$, the joint distribution is a point mass, so $\rho=v=0$. Any stationary planar velocity then vanishes by the mechanical-energy identity and coercivity.
\end{proof}

\FloatBarrier

\subsection{Quantitative relaxation and lower bounds}\label{sec:rates-proof}

\begin{lemma}\label{lem:rectangular-force}
In the strictly stable setting of \autoref{thm:quantitative-selection}, let $S$ be an area-preserving homeomorphism of each closed half cell that maps each side onto itself, and suppose its rearranged scalars have the prescribed periodicity and parity. A balanced force implies $S=\mathrm{id}$. If $[S]_{C^{0,\theta}}$ and $[S^{-1}]_{C^{0,\kappa}}$ are at most $L$ on each half cell, the force of the rearranged pair obeys
\begin{equation}\label{eq:lip-force-static}
 \|S-\mathrm{id}\|_2\le C\|\mathbb P\mathbf F\|_{L^\infty}^{\kappa(1+\theta)/2}.
\end{equation}
With $g=\|\mathbb P\mathbf F\|_{\mathcal V'}$ it satisfies
\begin{equation}\label{eq:lip-force-negative}
 \|S-\mathrm{id}\|_2\le Cg^{\gamma_{\theta,\kappa}},
\end{equation}
where $\gamma_{\theta,\kappa}$ is given by \eqref{eq:holder-rate} and the constants depend on $L$, the fixed background and the exponents.
\end{lemma}

\begin{proof}
Write $S=(\xi,\eta)$, $b=f\beta_f$, $Q$ for a half cell, $Y=(x,\Psi(z))$ and $T=(b\xi,-\rho_s(\eta))$. The strict inequalities imply that $Y$ and $T_0=(bx,-\rho_s(z))$ are smooth bi-Lipschitz coordinates onto rectangles. In $Y$ coordinates, $T=T_0\circ S$ is a homeomorphism from the source rectangle $Y(Q)$ onto the target rectangle $T_0(Q)$. Corresponding sides are mapped onto corresponding sides because $-\rho_s$ is strictly increasing as a function of $\Psi$. Denote these rectangles by $U$ and $V$. The potential part of the force is $\nabla\pi=\mathbf F-\mathbf G$, where $\mathbf G=\mathbb P\mathbf F$ and $\mathbf F=(f\beta_f(\xi-x),-\Psi'(\rho_s(\eta)-\rho_s(z)))$. In $Y$ coordinates, set
\[
 \Phi(Y(x,z))=\pi(x,z)+\tfrac b2 x^2-\int_0^z\rho_s(r)\Psi'(r)\,dr .
\]
Then $\nabla_Y\Phi=T-e$, with $e=(G_1,G_2/\Psi')$ in these coordinates. We use the continuous representative of $\mathbf G$ given by the frequency estimate below. In the balanced case $\mathbf G=0$ the rearranged force is continuous, so $\Phi$ has a continuous gradient also for continuous labels.

Fix $Y_0\in U$, $y=T(Y_0)$ and let $Y_*$ minimize $\Phi(Y)-y\cdot Y$ over $U$. For an interior coordinate of $Y_*$, the corresponding component of $T(Y_*)-y$ equals that of $e(Y_*)$. For a coordinate on the lower side, $T_i(Y_*)=\min_{w\in V}w_i\le y_i$ while $\partial_i\Phi(Y_*)-y_i\ge0$. The reverse inequalities hold on the upper side. These signs, including at corners, give
\[
 |T(Y_*)-y|\le C\|e\|_\infty .
\]
When $e=0$ the injectivity of $T$ gives $Y_*=Y_0$. For H\"older labels, the inverse bound gives $|Y_*-Y_0|\le C\|e\|_\infty^\kappa$. If $\varepsilon=\|e\|_\infty\le1$, then along the segment from $Y_0$ to $Y_*$ the gradient of $\Phi-y\cdot Y$ is bounded by $C|Y_*-Y_0|^\theta+\varepsilon$. Integrating along the segment gives a bound by $C(\varepsilon^{\kappa(1+\theta)}+\varepsilon^{1+\kappa})$, hence
\begin{equation}\label{eq:approximate-support}
 0\le \Phi(Y_0)+\Phi^*(y)-Y_0\cdot y\le C\|e\|_\infty^{\kappa(1+\theta)},
 \qquad \Phi^*(y):=\max_{Y\in U}(Y\cdot y-\Phi(Y)).
\end{equation}
For $e=0$ the defect is identically zero. If $\|e\|_\infty>1$, the displacement estimate follows from boundedness of the half cell. This definition of $\Phi^*$ only uses a supremum and does not require convexity of $\Phi$.

Because $T$ and $T_0$ have the same law on the physical half cell, the lower bound in Fenchel's inequality at $(Y,T_0)$ and \eqref{eq:approximate-support} at $(Y,T)$ give
\[
 \int_Q Y\cdot(T_0-T)\,dx\,dz\le C\|e\|_\infty^{\kappa(1+\theta)} .
\]
For the strictly convex reference potential $\Phi_0$ in the rigidity proof, its conjugate has Hessian uniformly positive definite in $V$. Equivalently, uniform convexity there follows from the bounded positive derivative of the explicit inverse of $T_0$ in $Y$ coordinates. Thus
\[
 c|T-T_0|^2
 \le \Phi_0(Y)+\Phi_0^*(T)-Y\cdot T .
\]
Integrate this and use the equality of laws again, together with equality in Fenchel's inequality at $(Y,T_0)$. This bounds $\|T-T_0\|_2^2$ by $C\|e\|_\infty^{\kappa(1+\theta)}$. The strict stratification and positive coupling give $|S-\mathrm{id}|\le C|T-T_0|$. Summing over the half cells proves \eqref{eq:lip-force-static}, and the same argument when $e=0$ proves the balanced-homeomorphism assertion.

We next estimate the projected force in the dissipation norm. Write $g=\|\mathbf G\|_{\mathcal V'}$. Extend $\mathbf F$ to the vertically doubled torus by even reflection of its horizontal component and odd reflection of its vertical component. The extension is uniformly $C^{0,\theta}$. In fact, the vertical force vanishes at the two horizontal walls since $S$ preserves each side, and the parity and matching traces at vertical seams give a globally $C^{0,\theta}$ force on the original periodic cell. On the doubled torus the torus Helmholtz projection $\widetilde{\mathbf G}$ restricts to $\mathbf G$: the projection respects the reflection parities and its gradient part restricts to a periodic Neumann potential. Moreover
\[
 \|\widetilde{\mathbf G}\|_{H^{-1}(\mathbb T\times(\mathbb R/2h\mathbb Z))}
 \le Cg .
\]
Indeed, symmetrize an arbitrary $H^1$ test field to the same parities and apply the torus Helmholtz projection. Its restriction is in $\mathcal V$, with $H^1$ norm at most a constant times the original test norm. Pairing on the torus is twice the pairing on the strip.

For a smooth dyadic block of frequencies of size $2^j$, $j\ge1$, the annular Helmholtz multiplier has a periodized kernel of mean zero and absolute moment of order $\theta$ bounded by $C2^{-j\theta}$, by rescaling a smooth annular symbol. Subtracting the value of the extended force at the convolution point and using its H\"older seminorm gives the uniform block bound $C2^{-j\theta}$ in $L^\infty$. For frequencies at most a constant times $2^J$, Cauchy--Schwarz in the Fourier sum bounds the $L^\infty$ norm of the low-frequency sum by $C2^{2J}g$ in two dimensions. Summing the high-frequency block bounds gives
\[
 \|\mathbf G\|_\infty\le C(2^{2J}g+2^{-J\theta}),\qquad J\ge0.
\]
The high-frequency series therefore converges uniformly. For $0<g\le1$, choose $2^J$ comparable to $g^{-1/(2+\theta)}$. For $g\ge1$, use the bounded low-frequency part and the uniform block bounds, since the force itself is bounded. The case $g=0$ follows by density of $\mathcal V$ in $\mathcal H$. We obtain $\|\mathbf G\|_\infty\le C g^{\theta/(2+\theta)}$. Substitution in \eqref{eq:lip-force-static} proves \eqref{eq:lip-force-negative}. 
\end{proof}

\begin{proof}[Proof of \autoref{thm:quantitative-selection}]
The scalar pair satisfies $(\rho+\rho_s,v+v_s+fx)=Z_s\circ S_t$ by transport. Thus $v=\beta_f d_1$ and
$\rho=\rho_s(z+d_2)-\rho_s(z)$. The strict sign condition bounds $|\rho_s'|$ and $|\Psi'|$ from below and makes their signs opposite. In particular, $\|(\rho,v)\|_2$ is equivalent to $\|d\|_2$, and
\begin{equation}\label{eq:quant-energy-equivalence}
 \mathcal E\asymp\|\mathbf u\|_2^2+\|d\|_2^2,\qquad
 \mathcal E'=-\nu\Diss(\mathbf u).
\end{equation}
The equivalence follows from \eqref{eq:lyapunov-mechanical} with the decreasing density law $\rho_s=\mathfrak r(\Psi)$, whose derivative is bounded between two negative constants. The two components of $d$ have the velocity parities, since the scalar perturbations have their prescribed parities and $\rho_s$ is strictly monotone. The normal component of $d$ vanishes on the four sides of each half cell. In particular $d$ extends to a periodic $H^1$ vector field on the full strip: the odd component has zero trace on the vertical sides and the even component has matching traces there.

\medskip
\noindent\textit{Coercivity of the projected force.}
As $\det DS_t=1$, the two-dimensional identity for a displacement yields
\begin{equation}\label{eq:quant-null}
 \nabla\cdot d=-\det Dd,\qquad
 \det Dd=\partial_x(d_1\partial_z d_2)-\partial_z(d_1\partial_x d_2).
\end{equation}
The vector field in the latter divergence has zero normal flux: $d_1=0$ on the vertical sides, and $\partial_x d_2=0$ on the horizontal sides. Thus for every periodic $H^1$ test function $\psi$,
\[
 \left|\int_\Omega d\cdot\nabla\psi\right|
    \le C\|Dd\|_{L^\infty}\|d\|_2\|\nabla\psi\|_2 .
\]
For smooth functions, we first integrate on each half cell and use the boundary fluxes in \eqref{eq:quant-null}, and pass to $C^1$ displacements and $H^1$ tests by approximation. Equivalently this identity is the weak divergence formula for the determinant. In the Helmholtz decomposition $d=\mathbb P d+\nabla\phi$ with periodic mean-zero Neumann potential, this implies
\begin{equation}\label{eq:quant-gradient-piece}
 \|d-\mathbb P d\|_2\le C\|Dd\|_{L^\infty}\|d\|_2,\qquad
 \|\mathbb P d\|_{H^1}\le C\|d\|_{H^1}.
\end{equation}
For the second estimate, apply the $H^2$ Neumann bound with source $\nabla\cdot d$ and homogeneous boundary data. The potential is even by reflection, so its gradient has velocity parity and this is the projection onto the solenoidal space used in the paper.

Write $\mathbf G=\mathbb P\mathbf F$. For the labeled scalars,
\[
 \mathbf F=\bigl(f\beta_f d_1,\;-\Psi'(z)[\rho_s(z+d_2)-\rho_s(z)]\bigr).
\]
The strict sign condition gives $\int_\Omega\mathbf F\cdot d\ge c\|d\|_2^2$ and $\|\mathbf F\|_2\le C\|d\|_2$, by the mean value formula along the vertical segment between $z$ and $z+d_2$. Choose $\varepsilon_*$ small in \eqref{eq:small-label-gradient}. Then \eqref{eq:quant-gradient-piece} gives
\begin{equation}\label{eq:quant-force-coercivity}
 c\|d\|_2^2\le (\mathbf G,\mathbb P d)_2
     \le C\|\mathbf G\|_{\mathcal V'}\,\|d\|_{H^1}.
\end{equation}
Here $\mathcal V'$ is the dual of the velocity form space with norm $\|A^{1/2}\cdot\|_2$.

\medskip

\noindent\textit{A compensated energy.}
For a sufficiently small fixed $\epsilon>0$, set
\[
 \mathcal H_\epsilon=\mathcal E-\epsilon(\mathbf u,A^{-1}\mathbf G)_2 .
\]
Since $\|A^{-1}\mathbf G\|_2\le C\|\mathbf F\|_2$, \eqref{eq:quant-energy-equivalence} gives $\mathcal H_\epsilon\asymp\mathcal E$. The differentiated force bound used in \eqref{eq_nonlinear_acceleration_energy} gives
$\|\mathbf G_t\|_{\mathcal V'}\le C\|\mathbf u\|_{H^1}$: to define $\mathbf G_t$ in this dual, test $\mathbf F_t$ against $\mathcal V$. The uniform scalar bounds required there hold here by transport of the bounded initial pair. Testing the projected momentum equation against $A^{-1}\mathbf G$ yields
\[
 \frac{d}{dt}(\mathbf u,A^{-1}\mathbf G)_2
  =\|\mathbf G\|_{\mathcal V'}^2-\nu(\mathbf G,\mathbf u)_2
      -\langle(\mathbf u\cdot\nabla)\mathbf u,A^{-1}\mathbf G\rangle
      +\langle \mathbf G_t,A^{-1}\mathbf u\rangle .
\]
We bound the convection pairing by
$C\|\mathbf u\|_{L^4}^2\|\mathbf G\|_{\mathcal V'}
 \le C\|\mathbf u\|_2\|\mathbf u\|_{H^1}\|\mathbf G\|_{\mathcal V'}$.
The absolute values of the remaining terms on the right are bounded by
$C\|\mathbf u\|_{H^1}\|\mathbf G\|_{\mathcal V'}+C\|\mathbf u\|_{H^1}^2$.
The product rules follow from strong velocity regularity and $\mathbf G_t\in L^2_{\rm loc}(\mathcal V')$, since $A^{-1}$ maps $\mathcal V'$ to $\mathcal V$. Using Young's inequality and \eqref{eq:quant-energy-equivalence}, and decreasing $\epsilon$ if necessary, gives
\begin{equation}\label{eq:quant-compensated}
 \mathcal H_\epsilon'
   \le-c\bigl(\|\mathbf u\|_{H^1}^2+\|\mathbf G\|_{\mathcal V'}^2\bigr).
\end{equation}

\medskip
\noindent\textit{Uniform coercivity on bounded label classes.}
For part (2), consider the $C^1$ closure on the closed half cells of all admissible labels with $\|S-\mathrm{id}\|_{H^m}\le M$. This closure is compact by Sobolev embedding and Rellich compactness, since $m>2$. Every label in the closure is a side-preserving, area-preserving $C^1$ diffeomorphism. Indeed $\det DS=1$ passes to the limit, and the derivative bounds give uniform inverse Lipschitz constants because each half cell is convex. The limits are therefore bijective, with $C^1$ inverse by the inverse derivative formula. Restrict this compact set to $\|D(S-\mathrm{id})\|_{L^\infty}\ge\varepsilon_*$. On that set, if nonempty, $\|\mathbb P\mathbf F\|_{\mathcal V'}$ has a positive minimum. The force depends continuously on $S$ there, and a zero would give the pressure balance by the Helmholtz decomposition. Rigidity in \autoref{thm:rigid-balance}(1) would then give $S=\mathrm{id}$. Since $\|d\|_2$ is bounded, outside the small-gradient regime this gives $\|\mathbf G\|_{\mathcal V'}\ge c_M\|d\|_2^{1+1/m}$.

\medskip
\noindent\textit{Decay.}
By Sobolev interpolation on the two half cells,
\[
 \|d\|_{H^1}\le C_m\|d\|_2^{1-1/m}\|d\|_{H^m}^{1/m}\qquad(m\ge1),
\]
with the right-hand side interpreted as $\|d\|_{H^1}$ for $m=1$. Combining this with \eqref{eq:quant-force-coercivity} and the uniform $H^m$ bound gives
$\|\mathbf G\|_{\mathcal V'}\ge c\|d\|_2^{1+1/m}$ in the small-gradient regime. The $m=1$ bound is automatic from \eqref{eq:small-label-gradient} and the bounded range of the labels in the half cells. For part (2) the compactness estimate extends the bound to the entire trajectory, so we take $T=0$. For part (3), \eqref{eq:lip-force-negative} yields $\|\mathbf G\|_{\mathcal V'}\ge c\|d\|_2^{1/\gamma_{\theta,\kappa}}$ with $T=0$. Write $\mu=m$ in parts (1) and (2), and $\mu=\mu_{\theta,\kappa}$ in part (3), so $1+1/\mu=1/\gamma_{\theta,\kappa}$ in the latter case. Since the kinetic energy is uniformly bounded, \eqref{eq:quant-energy-equivalence} and \eqref{eq:quant-compensated} give $\mathcal H_\epsilon'\le-c\mathcal H_\epsilon^{1+1/\mu}$. Integration yields $\mathcal E(t)\le C(1+t-T)^{-\mu}$. This also controls the tail of $\int_t^\infty\|\mathbf u\|_{H^1}^2$ by \eqref{eq:quant-energy-equivalence}.

Let $y(t)=\|\mathbf u_t(t)\|_2^2$. The uniform bound on $y$ and \eqref{eq_nonlinear_acceleration_energy} imply
$y'+c_1\|\mathbf u_t\|_{H^1}^2\le C\|\mathbf u\|_{H^1}^2$.
With $r=(t+T)/2$, coercivity and the integrating factor give
\[
 y(t)\le e^{-c_2(t-r)}y(r)+C\int_r^t\|\mathbf u(s)\|_{H^1}^2\,ds
 \le C(1+t-T)^{-\mu}.
\]
Here we used the dissipation tail bound at $r$ and the uniform bound on $y(r)$. Integrating the same differential inequality on $[t,\infty)$ then gives
\[
 y(t)+\int_t^\infty\|\mathbf u_t(s)\|_{H^1}^2\,ds
 \le C(1+t-T)^{-\mu}.
\]
The force and velocity $L^2$ norms are $O((1+t-T)^{-\mu/2})$ by the energy equivalence. The preceding acceleration estimate has the same rate in $L^2$. Thus \eqref{eq:pointwise-velocity} first gives the $H^1$ decay at that rate and then the $H^2$ velocity and $H^1$ pressure decay. This proves the scalar and velocity $L^2$-based decay estimates in all three parts. In part (3), the uniform $C^{0,\theta}$ bound of $d$ on each half cell gives $\|d\|_\infty\le C\|d\|_2^{\theta/(1+\theta)}$. Indeed, a ball at a maximum point with radius proportional to $\|d\|_\infty^{1/\theta}$ has intersection area of order $\|d\|_\infty^{2/\theta}$ in the half cell, and $|d|$ is at least half its maximum there. For larger radii one uses the bounded size of the label displacement and a fixed ball. Since $S_t\circ X_t=S_0$, one has $d(t)\circ X_t=S_0-X_t$. Together with the force and label formulas this proves \eqref{eq:lip-uniform-decay}.

\medskip
\noindent\textit{Endpoint strain bound.}
Under the hypotheses of part (2), the force is in global $H^1$ by the matching periodic traces of the displacements across the vertical seams. The lifted label $S_{t,1}$ itself changes by $2\pi$ across the period. Differentiating the force formula and using the smooth background gives $\|\mathbf F\|_{H^1}\le C\|d\|_{H^1}$. Interpolation and the energy bound yield $\|\mathbf F(t)\|_{H^1}\le C(1+t)^{-(m-1)/2}$. The $H^1$ Stokes estimate gives, for almost every $t$,
\[
 \|\mathbf u\|_{H^3}
 \le C\bigl(\|\mathbf F\|_{H^1}+\|\mathbf u_t\|_{H^1}
                              +\|\mathbf u\|_{H^2}^2\bigr).
\]
The convection bound follows from $H^2\hookrightarrow L^\infty$ and $H^1\hookrightarrow L^4$ in two dimensions. For $L\ge1$ the decay and dissipation estimates give
\[
 \int_L^{2L}\|\mathbf u\|_{H^1}^2\,ds\le C L^{-m},\qquad
 \int_L^{2L}\|\mathbf u\|_{H^3}^2\,ds\le C L^{2-m}.
\]
The two-dimensional Gagliardo--Nirenberg inequality applied to $\nabla\mathbf u$ gives $\|\nabla\mathbf u\|_{L^\infty}\le C\|\mathbf u\|_{H^1}^{1/2}\|\mathbf u\|_{H^3}^{1/2}$. H\"older's inequality on $[L,2L]$ therefore gives
\[
 \int_L^{2L}\|\nabla\mathbf u\|_{L^\infty}\,ds
 \le C L^{1/2}(L^{-m})^{1/4}(L^{2-m})^{1/4}
 = C L^{1-m/2}.
\]
The dyadic contributions are summable because $m>2$. The finite-time strain estimate in \autoref{thm_nonlinear_regularity} bounds the integral on $[0,1]$ uniformly, because the scalar $L^q$ bounds follow from their reference joint distribution and the velocity bound is $R$. Thus the total strain is uniformly finite. By \autoref{thm:rigid-balance}(3), $X_t\to S_0$ in $C^1$. Integrating the flow and derivative equations from $t$ to infinity bounds $\|X_t-S_0\|_{C^1}$ by a constant times the strain tail. For the flow itself one uses $\|\mathbf u\|_{L^\infty}\le C\|\nabla\mathbf u\|_{L^\infty}$, by the vanishing traces of its components on the sides. The derivative of the flow is uniformly bounded by the exponential of the total strain. This proves \eqref{eq:quant-strain-tail}.
\end{proof}

\subsubsection{Slow relaxation and a wall obstruction}

\begin{proof}[Proof of \autoref{thm:nonlinear-slow-relaxation}]
Use the real weighted space of \autoref{thm_linear_large_time} and write $\gamma(z)=-\rho_s'(z)\Psi'(z)>0$. Fix an interior point $z_*$, put $\gamma_* =\gamma(z_*)$, and set $a_j=\gamma_*/(\nu j^2)$ for integers $j\ge1$. Given $\delta_*>0$, choose $\chi\in C_c^\infty(0,h)$ with $\|\chi\|_2=1$ supported where $|\gamma-\gamma_*|\le\delta_*$. Form the divergence-free field
\[
 \widetilde V_j=(-j^{-1}\chi'(z)\sin(jx),\ \chi(z)\cos(jx)).
\]
Choose $V_j$ to be a scalar multiple of $\widetilde V_j$ such that $Y_j=(\mathbf U_j,q_j)$, with $\mathbf U_j=a_jV_j$ and $q_j=B^*V_j$, has norm one in $\mathscr X$. The normalizing factors are bounded above and below uniformly for large $j$ by constants depending only on the background, since $B^*$ is pointwise invertible as a map on the velocity components, and $\|\widetilde V_j\|_2^2\to\pi$. For integers $k\ge1$,
\[
\begin{gathered}
 \|V_j\|_2\le C,\qquad \|V_j\|_{H^k}\le C_{k,\chi}j^k,\\
 \|\nabla V_j\|_{L^\infty}+\|q_j\|_{W^{1,\infty}}\le C_\chi j,\qquad
 \|\nabla\mathbf U_j\|_{L^\infty}\le C_\chi j^{-1}.
\end{gathered}
\]
The scalar relation $B^*\mathbf U_j=a_jq_j$ holds exactly. For the velocity define $R_j=(D-a_j I)\mathbf U_j-Bq_j$ with $D=\nu A$. Since $-\partial_x^2 V_j=j^2 V_j$ and $BB^*V_j=\mathbb P(b(V_j)_1,\gamma(z)(V_j)_2)$, one has
\begin{equation}\label{eq:slow-quasimode}
\begin{aligned}
 R_j&=\mathbb P\left((\gamma_*-b)(V_j)_1,\,
                    (\gamma_*-\gamma)(V_j)_2\right)
              -\nu a_j\mathbb P\partial_z^2V_j-a_j^2V_j,\\
 \|R_j\|_2&\le C\delta_*+C_\chi j^{-1}.
\end{aligned}
\end{equation}
The constant multiplying $\delta_*$ is independent of $\chi$ for large $j$. Both $V_j$ and $\mathbf U_j$ vanish near the horizontal walls and satisfy the prescribed velocity parity.

\medskip

\noindent\textit{Exact joint rearrangements.}
Let $S_j(s)$ be the autonomous flow of $V_j$ for time $s\ge0$. It preserves area, reflection symmetry and each side of each half cell, since $V_j$ is solenoidal, has zero normal trace on the horizontal walls and has the velocity parity. Use the periodic lift of its displacement $d_j(s)=S_j(s)-\mathrm{id}$. The derivative equation for the flow gives
\[
 \|\nabla d_j(s)\|_{L^\infty}\le e^{s\|\nabla V_j\|_{L^\infty}}-1 .
\]
The autonomous flow identity yields $(\partial_s-V_j\cdot\nabla)d_j=V_j$. With $n=\lceil m\rceil$, differentiate this equation to order $n$. Impermeability in the leading transport term and the Sobolev product estimate for the commutator give
\[
 \frac{d}{ds}\|d_j\|_{H^n}
 \le C_n\|\nabla V_j\|_{L^\infty}\|d_j\|_{H^n}
      + C_n\|V_j\|_{H^n}\bigl(1+\|\nabla d_j\|_{L^\infty}\bigr).
\]
Since $\|V_j\|_{H^n}\le C_{n,\chi}j^n$, integration gives $\|d_j(s)\|_{H^n}\le C_{n,\chi} s j^n$ for $0\le s\le j^{-1}$. The displacement and the corresponding scalar perturbation $q_{0,j}(s)=Z_s\circ S_j(s)-Z_s$ have the same bound in $H^n$. For the density this follows by applying the Sobolev composition estimate to the smooth function $\rho_s$ and the displacement $d_{j,2}$. The shifted meridional part of the difference is $\beta_f d_{j,1}$.

Since $(V_j\cdot\nabla)Z_s=q_j$ on the representative cell, we have
\[
 q_{0,j}(s)=\int_0^s q_j\circ S_j(\tau)\,d\tau,\qquad
 \|q_{0,j}(s)-s q_j\|_2\le C_\chi s^2 j.
\]
Indeed, the derivative of the integrand has $L^2$ norm bounded by $\|\nabla q_j\|_{L^\infty}\|V_j\|_2$, using area preservation. The flow formula also gives $\|d_j(s)\|_2\le s\|V_j\|_2$, and the integral formula gives $\|q_{0,j}(s)\|_2\le s\|q_j\|_2$. Interpolating with the $H^n$ estimates above yields
\[
 \|d_j(s)\|_{H^m}+\|q_{0,j}(s)\|_{H^m}\le C_{m,\chi}s j^m
 \qquad(0\le s\le j^{-1})
\]
for every real $m\ge1$. The $H^m$ bound for $\mathbf U_j$ follows by the same interpolation of its integer Sobolev estimates. Set $s_j=\kappa j^{-m}$, with $0<\kappa\le1$, and take the initial data $(s_j\mathbf U_j,q_{0,j}(s_j))$. The initial size in the theorem is at most $C_m\kappa$ and the scalar pair has the exact reference joint law. These data have horizontal period $2\pi/j$. For the scalar perturbations this follows because translating the lift of the flow by one period translates its first coordinate by the same amount, and the affine part cancels in $Z_s\circ S_j-Z_s$. The perturbation equations have coefficients independent of $x$. Their solution retains this period: apply the well-posedness proof on the strip of period $2\pi/j$, extend periodically, and use uniqueness on $\Omega$. The velocity has zero horizontal mean, since the first velocity component is odd, and the horizontal average of the vertical component is constant in $z$ by incompressibility and zero at the walls. Thus $\|\mathbf u\|_2\le j^{-1}\|\partial_x\mathbf u\|_2$. The coercive mechanical energy on this class, with its dissipation identity, yields for the nonlinear solution $Y(t)=(\mathbf u(t),q(t))$
\begin{equation}\label{eq:slow-energy-bound}
 \sup_{t\ge0}\|Y(t)\|_{\mathscr X}^2
       +\int_0^\infty\|\mathbf u(\tau)\|_{H^1}^2\,d\tau
 \le C s_j^2 .
\end{equation}
The initial scalar bound also follows from $\|q_{0,j}(s_j)\|_2\le s_j\|q_j\|_2$ by area preservation.

\medskip
\noindent\textit{Testing against the adjoint approximate mode.}
The involution $(\mathbf U,q)\mapsto(-\mathbf U,q)$ conjugates $\mathcal L$ to its adjoint. Thus $Y_j^*=(-\mathbf U_j,q_j)$ satisfies $\mathcal L^*Y_j^*=-a_jY_j^*+(R_j,0)$. The resulting amplitude $b_j(t)=(Y(t),Y_j^*)_{\mathscr X}$ satisfies
\[
 b_j'(t)+a_j b_j(t)=(\mathbf u(t),R_j)_2+N_j(t),\qquad
 |N_j(t)|\le C_\chi\bigl(j^{-1}\|\mathbf u(t)\|_2^2+
                        j\|\mathbf u(t)\|_2\|q(t)\|_2\bigr).
\]
We obtain this bound by integrating each advective term by parts and using incompressibility and the zero normal trace. For the scalar terms the spatial derivative falls on the product of $q_j$ and the smooth scalar weight in \eqref{eq_linear_balance_projection_weight}. The identity holds in time distributions with locally integrable right-hand side by the strong solution regularity. 

Fix $L=1+\gamma_*/\nu$. By \eqref{eq:slow-energy-bound} and the horizontal Poincar\'e estimate, for $0\le t\le L/a_j$,
\[
 \int_0^t\|\mathbf u(\tau)\|_2\,d\tau\le C_L s_j,
 \qquad
 \int_0^t\|\mathbf u(\tau)\|_2^2\,d\tau\le C j^{-2}s_j^2 .
\]
Consequently, the error terms in the variation-of-constants formula are bounded by $C_L(\delta_*+C_\chi/j)s_j+C_{L,\chi}j s_j^2$. The constant multiplying $\delta_*s_j$ is independent of $\chi$ because \eqref{eq:slow-energy-bound} only uses the $L^2$ size of the normalized initial data. The initial amplitude is
\[
 b_j(0)=s_j(1-2\|\mathbf U_j\|_2^2)+O_\chi(s_j^2 j).
\]
First choose $\delta_*$ small relative to $e^{-L}$ and fix $\chi$. Then take $j$ sufficiently large and $\kappa$ sufficiently small. Since $m\ge1$ gives $j s_j\le\kappa$, these choices yield $b_j(t)\ge c_L s_j$ on $[0,L/a_j]$. The velocity term in this pairing is bounded by $C s_j a_j$ by \eqref{eq:slow-energy-bound} and boundedness of $\|V_j\|_2$. For large $j$ it follows that $\|q(t)\|_2\ge c s_j$ on the same interval.

For every sufficiently large $t$, take $j=\lceil\sqrt{t}\,\rceil$. Then $a_jt\le L$. Choose $\kappa$ to be a fixed small multiple of $r$. The bound on the initial norms is at most $r$, and $s_j=\kappa j^{-m}\ge c r(1+t)^{-m/2}$. This proves \eqref{eq:nonlinear-slow-lower}.
\end{proof}

\begin{proof}[Proof of \autoref{prop:wall-obstruction}]
The finite-time velocity bound in $L^2(W^{2,q})$, $q>2$, gives $L^1(C^1)$ in space and hence a $C^1$ flow up to the wall. At $z=h$, $w_x=0$ follows from $w=0$, and $u_z=0$ is the stress-free condition. The velocity gradient there is diagonal. Write $X_t(a,h)=(\xi_t(a),h)$ on the lift of the circle. In the derivative equation for the flow, the diagonal entries have product one by incompressibility, so
\[
 DX_t(a,h)=\begin{pmatrix} J_t(a)&0\\0&J_t(a)^{-1}\end{pmatrix},
 \qquad J_t(a)=\partial_a\xi_t(a)>0 .
\]
The chain rule for $M(t)\circ X_t=M(0)$ now gives
\[
 M_x(t,\xi_t(a),h)=J_t(a)^{-1}M_x(0,a,h),\qquad
 M_z(t,\xi_t(a),h)=J_t(a)M_z(0,a,h).
\]
Multiplying the trace factors in \eqref{eq:wall-invariant} cancels $dx=J_t(a)\,da$. The derivatives are periodic because $M$ is the sum of an affine function and a periodic perturbation. The reciprocal factors are shown in \autoref{fig:wall-layer}(a).

For the strain conclusion assume the reference rearrangement representation $Z(0)=Z_s\circ S_0$. Suppose a sequence $t_j\to\infty$ gave relative $C^1$ compactness of the flow maps on both closed half cells. As in the proof of \autoref{thm:rigid-balance}(2), any $C^1$ limit yields the original balanced pair through the inverse flow. Along the resulting subsequence, the meridional perturbations tend to zero in piecewise $C^1$. The invariant then vanishes. Part (3) of that theorem also shows that finite total strain would produce such a sequence. These alternatives are therefore excluded when $\mathcal I(0)>0$. The uniform $H^m$ label bound for $m>2$ gives relative $C^1$ compactness of the inverse flow through $X_t=S_t^{-1}\circ S_0$ and the inverse derivative formula. It too requires $\mathcal I(0)=0$. If $M_x(0,x,h)>0$ everywhere, this forces $v_{0,z}(x,h)=0$.

To obtain the sharp trace estimate, let $g$ be a smooth periodic function on $\Omega$. We use a Sobolev extension from the strip to $\mathbb T\times\mathbb R$ bounded in both $L^2$ and $H^s$. Decompose the extension into pieces $g_j$ with joint Fourier frequencies of magnitude comparable to $2^j$ for $j\ge0$, including the bounded-frequency piece at $j=0$. For either derivative, Cauchy--Schwarz in the vertical Fourier variable and Parseval in $x$ give
\[
 \|\nabla g_j(\cdot,h)\|_{L^2(\mathbb T)}\le C 2^{3j/2}\|g_j\|_2 .
\]
Summing up to $j=J$ and using Cauchy--Schwarz bounds that part by $C2^{3J/2}\|g\|_2$. The remaining sum is at most $C_s2^{(3/2-s)J}\|g\|_{H^s}$ because $s>3/2$. Choose $2^J$ comparable to $(\|g\|_{H^s}/\|g\|_2)^{1/s}$, using $J=0$ if the ratio is bounded. This gives
\begin{equation}\label{eq:wall-multiplicative-trace}
 \|\nabla g(\cdot,h)\|_{L^2(\mathbb T)}
 \le C_s\|g\|_2^{1-3/(2s)}\|g\|_{H^s}^{3/(2s)} .
\end{equation}
By density, the bound extends to $H^s$. In \eqref{eq:wall-invariant}, H\"older's inequality gives
\[
 \mathcal I(0)\le
 \|v_z(t,\cdot,h)\|_{L^2}^{1/2}
 \bigl(\beta_f\sqrt{2\pi}+\|v_x(t,\cdot,h)\|_{L^2}\bigr)^{3/2}.
\]
When $\mathcal I(0)>0$ the sum of the two derivative trace norms is bounded away from zero. Applying \eqref{eq:wall-multiplicative-trace} to $v(t)$ proves \eqref{eq:wall-tradeoff}.

\begin{figure}[!htbp]
\centering
\begin{tikzpicture}[x=1cm,y=.93cm]
  \node[figPanel] at (2.72,2.94) {(a) Top-wall deformation};
  \draw[figInk,thick] (.10,2.18)--(5.40,2.18);
  \filldraw[fill=figBlue!14,draw=figBlue] (.38,1.20) rectangle (1.50,2.18);
  \filldraw[fill=figTeal!14,draw=figTeal] (3.12,1.57) rectangle (4.92,2.18);
  \draw[figArrow] (1.79,1.84)--(2.84,1.84) node[midway,above] {$DX_t$};
  \node[figBlue,below] at (.94,1.19) {$da$};
  \node[figTeal,below] at (4.02,1.54) {$dx=J_t\,da$};
  \node[align=center] at (2.74,.40) {$M_x\mapsto J_t^{-1}M_x$\qquad $M_z\mapsto J_tM_z$};
  \node[figPanel] at (8.68,2.94) {(b) Thin exact rearrangement};
  \fill[figBlue!7] (6.75,.36) rectangle (10.76,2.22);
  \fill[figOrange!17] (6.75,1.65) rectangle (10.76,2.22);
  \draw[figInk] (6.75,.36) rectangle (10.76,2.22);
  \draw[figOrange,dashed] (6.75,1.65)--(10.76,1.65);
  \draw[figInk,<->,>=Latex] (11.03,1.65)--(11.03,2.22) node[midway,right] {$\ell$};
  \node[above] at (8.73,2.23) {$z=h$};
  \node[figOrange] at (8.74,1.91) {$S_\ell$ acts here};
  \node at (8.74,1.0) {$S_\ell=\mathrm{id}$};
\end{tikzpicture}
\caption{Wall transport and saturation of the trace bound. (a) The linearized deformation of a small wall element preserves area and multiplies tangential and normal lengths by $J_t$ and $J_t^{-1}$. The two scalar derivative factors and the wall measure cancel in \eqref{eq:wall-invariant}. (b) The exact joint rearrangement in a layer of thickness $\ell$ fixes the top wall pointwise and keeps the positive wall invariant fixed while giving the two Sobolev scales in \eqref{eq:wall-layer-scaling}. Geometry is schematic.}
\label{fig:wall-layer}
\end{figure}

In the top layer shown in \autoref{fig:wall-layer}(b), the map fixes the wall pointwise while its first label acquires a nonzero normal derivative. For the exact layer construction, fix a smooth $\chi$ on $[0,\infty)$ supported in $[0,1)$, with $\chi(0)=\chi'(0)=0$ and $\chi''(0)\ne0$. For $0<\ell<h$ set
\[
 \psi_\ell(x,z)=\ell^2\chi((h-z)/\ell)\sin x,\qquad
 W_\ell=(-\partial_z\psi_\ell,\partial_x\psi_\ell).
\]
Let $S_\ell$ be the flow of $W_\ell$ at a fixed nonzero auxiliary time $\epsilon$, and define the scalar pair $Z_s\circ S_\ell$. The generator is divergence free, has the velocity parity and preserves the sides of each half cell. It vanishes at the top wall and outside the layer. At $z=h$, $DW_\ell$ has only one nonzero entry, $(DW_\ell)_{12}=-\chi''(0)\sin x$. Hence $S_\ell$ fixes the wall pointwise and
\[
 M_x(x,h)=\beta_f,\qquad M_z(x,h)=-\epsilon\beta_f\chi''(0)\sin x .
\]
Thus the wall invariant is positive and independent of $\ell$.

In the coordinates $(x,r)$ with $r=(h-z)/\ell$, the generator is $\ell(\chi'(r)\sin x,-\chi(r)\cos x)$. Smooth dependence of this fixed vector field's flow on its auxiliary time $\epsilon\ell$ gives
\[
 \frac{v_\ell(x,h-\ell r)}{\ell}
   =\epsilon\beta_f\chi'(r)\sin x+O(\ell)
\]
in each fixed $C^k$ on $\mathbb T\times[0,1]$. Rescaling the integral shows $\|v_\ell\|_2\asymp\ell^{3/2}$, while differentiating in $z$ gives $\|v_\ell\|_{H^k}\le C_k\ell^{3/2-k}$ for integer $k$. Interpolating the upper bounds gives $\|v_\ell\|_{H^s}\le C_s\ell^{3/2-s}$, and \eqref{eq:wall-tradeoff} at these initial configurations gives the matching lower bound. This proves \eqref{eq:wall-layer-scaling}. To obtain arbitrarily small smooth data with $\mathcal I(0)>0$ in any fixed $H^m$, fix $\ell$ and send the nonzero auxiliary time $\epsilon$ to zero. One may set the planar initial velocity to zero in all these constructions.
\end{proof}

\paragraph{Balanced data with a positive wall invariant.}
In the strictly stable setting, take
\[
 \mathbf u=0,\qquad
 \rho=\epsilon\cos x,\qquad
 v=\tfrac{\epsilon}{f}(\Psi(z)-\Psi(h))\sin x,\qquad
 p=-\epsilon(\Psi(z)-\Psi(h))\cos x .
\]
These fields satisfy the periodic perturbation equations and the wall and reflection conditions. Indeed, $p_x=fv$ and $p_z=-\rho\Psi'$. At $z=h$, the wall derivatives are $M_x=\beta_f>0$ and $M_z=(\epsilon/f)\Psi'(h)\sin x$. Strict stratification gives $\Psi'(h)\ne0$, so $\mathcal I(0)>0$ for $\epsilon\ne0$. The flow is the identity at all times. Thus the reference rearrangement hypothesis is essential for the implication from positive wall invariant to infinite strain.

\FloatBarrier

\section{Conclusions}\label{sec:conclusions}

The growing spectrum is governed by the sign of the projected quadratic drive. In vertical stratification, the joint Weyl law separates the information in the leading spectral count from the information in fixed-mode thresholds. The former recovers the distribution of the positive stratification when the coupling varies, while the latter distinguishes some profiles with the same distribution.

For the nonlinear problem, the joint law of the transported density and shifted meridional velocity provides the common constraint for the energy comparison and the balanced minimization. Exact rearrangements realize nonlinear escape in the unstable regime. In the stable regime, folded stationary labels and the wall invariant explain the role of the hypotheses in the reference-state selection results. The relaxed transport minimum is a lower bound for the limiting defect of a trajectory. Equality selects the optimal balanced state when the target label law is absolutely continuous. For diffeomorphic rearrangements of the reference state, the decay estimates use quantitative control of the evolving maps. Whether compatible small perturbations propagate such control in the mixed-slip problem remains open here.

\appendix
\section{The Stokes operator for mixed slip walls}\label{sec:stokes}

\begin{lemma}\label{lem_stokes_estimates}
Let $\mathcal V$ be the space of periodic $\mathbf{u}=(u,w)\in H^1(\Omega)^2$
with $\nabla\cdot\mathbf{u}=0$, $w|_{z=0,h}=0$, and the velocity parity in
\eqref{eq_perturbation_parity}. Let $\mathcal H$ be its closure in $L^2(\Omega)^2$
and let $\mathbb P$ be the $L^2$ orthogonal projection onto $\mathcal H$.
For $\mathbf{u},\boldsymbol{\varphi}\in\mathcal V$, set
\begin{equation}\label{eq_stokes_form}
\begin{aligned}
a(\mathbf{u},\boldsymbol{\varphi})
&=\int_\Omega \nabla\mathbf{u}:\nabla\boldsymbol{\varphi}\,dx\,dz
  +\alpha\int_{\mathbb T}u(x,0)\varphi_1(x,0)\,dx .
\end{aligned}
\end{equation}
The displayed form is real bilinear. On complexified spaces we conjugate the second argument. The form is coercive on $\mathcal V$. Its operator $A$
is positive and self-adjoint in $\mathcal H$ and satisfies
\begin{equation}\label{eq_stokes_domains}
\begin{aligned}
D(A)&=\{\mathbf{u}\in\mathcal V\cap H^2(\Omega)^2:
               \mathbf{u}\text{ satisfies }\eqref{bd-con-2}\},&
A\mathbf{u}&=-\mathbb P\Delta\mathbf{u},\\
D(A^{1/2})&=\mathcal V,&
\|A^{1/2}\mathbf{u}\|_{L^2}^2&=a(\mathbf{u},\mathbf{u}).
\end{aligned}
\end{equation}
For a periodic force $\mathbf F=(F_1,F_2)\in L^2(\Omega)^2$ with the same
reflection parity as $\mathbf{u}$, the steady Stokes problem
\begin{equation}\label{eq_steady_stokes}
\begin{aligned}
-\Delta\mathbf{u}+\nabla p&=\mathbf F,\quad
\nabla\cdot\mathbf{u}=0 &&\text{in }\Omega,\\
\mathbf{u}&\text{ satisfies }\eqref{bd-con-2},&
\int_\Omega p\,dx\,dz&=0
\end{aligned}
\end{equation}
has a unique solution $\mathbf{u}\in D(A)$ with a periodic even pressure
$p\in H^1(\Omega)$. The following estimates hold whenever the norm on the right-hand side is finite, with $2<\ell<\infty$ in the second line:
\begin{equation}\label{eq_stokes_estimates}
\begin{aligned}
\|\mathbf{u}\|_{H^2}+\|p\|_{H^1}&\le C\|\mathbf F\|_{L^2},\\
\|\mathbf{u}\|_{W^{2,\ell}}+\|p\|_{W^{1,\ell}}&\le C_\ell\|\mathbf F\|_{L^\ell},\\
\|\mathbf{u}\|_{H^3}+\|p\|_{H^2}&\le C\|\mathbf F\|_{H^1}.
\end{aligned}
\end{equation}
\end{lemma}

\begin{proof}
Integrate in $z$ from the bottom wall and apply Cauchy--Schwarz. Since
$w(x,0)=0$,
\begin{equation}\label{eq_stokes_coercivity}
\begin{aligned}
\|u\|_{L^2}^2
 &\le 2h\|u(\cdot,0)\|_{L^2(\mathbb T)}^2+2h^2\|\partial_z u\|_{L^2}^2,\\
\|w\|_{L^2}^2&\le h^2\|\partial_z w\|_{L^2}^2,\\
\|\mathbf{u}\|_{H^1}^2&\le C a(\mathbf{u},\mathbf{u}).
\end{aligned}
\end{equation}
Coercivity holds on the full velocity space as well. The trace theorem and the representation theorem for closed forms give the self-adjoint operator and its square-root domain in \eqref{eq_stokes_domains}.

Lax--Milgram gives the weak velocity in \eqref{eq_steady_stokes}. The
solenoidal weak equation extends to periodic solenoidal test fields with
zero normal trace by taking their part with the velocity reflection parity.
Indeed, both the form and the forcing pairings with the opposite parity
vanish. On the periodic strip the divergence of periodic $H^1$ test fields
with zero normal trace has a bounded right inverse on mean-zero $L^2$.
Given a mean-zero $g\in L^2$, solve $\Delta\phi=g$ with homogeneous
Neumann conditions at the walls and take $\boldsymbol{\varphi}=\nabla\phi$.
The scalar Neumann estimate gives
$\|\boldsymbol{\varphi}\|_{H^1}\le C\|g\|_{L^2}$. This yields a mean-zero
periodic pressure through
\begin{equation}\label{eq_stokes_weak_pressure}
\begin{aligned}
a(\mathbf{u},\boldsymbol{\varphi})
-\int_\Omega p\,\nabla\cdot\boldsymbol{\varphi}\,dx\,dz
 &=\int_\Omega \mathbf F\cdot\boldsymbol{\varphi}\,dx\,dz,\\
\|\mathbf{u}\|_{H^1}+\|p\|_{L^2}&\le C\|\mathbf F\|_{L^2}.
\end{aligned}
\end{equation}
The first identity holds for all periodic $H^1$ test fields with zero normal
trace. Uniqueness and reflection of that identity give evenness of $p$.

To prove the first estimate in \eqref{eq_stokes_estimates}, denote the
horizontal difference quotient by $D_\varepsilon$. Taking the difference
of the weak equation and testing with $D_\varepsilon\mathbf{u}$ gives
\begin{equation}
\begin{aligned}
a(D_\varepsilon\mathbf{u},D_\varepsilon\mathbf{u})
&=-\int_\Omega \mathbf F\cdot D_{-\varepsilon}D_\varepsilon\mathbf{u}\,dx\,dz
\le \|\mathbf F\|_{L^2}\|\partial_xD_\varepsilon\mathbf{u}\|_{L^2},\\
\|D_\varepsilon\mathbf{u}\|_{H^1}+\|D_\varepsilon p\|_{L^2}
&\le C\|\mathbf F\|_{L^2}.
\end{aligned}
\end{equation}
The pressure bound follows from the divergence right inverse in
the difference of \eqref{eq_stokes_weak_pressure}, using
$|\int_\Omega D_\varepsilon\mathbf F\cdot\boldsymbol{\varphi}|
\le \|\mathbf F\|_{L^2}\|\partial_x\boldsymbol{\varphi}\|_{L^2}$.
Thus $\partial_x\mathbf{u}\in H^1$ and $\partial_x p\in L^2$. The remaining
derivatives follow from the equations in distributions:
\begin{equation}\label{eq_stokes_normal_derivatives}
\begin{aligned}
\partial_z^2 w&=-\partial_z\partial_x u,\\
\partial_z^2 u&=-\partial_x^2u+\partial_xp-F_1,\\
\partial_z p&=F_2+\partial_x^2w+\partial_z^2 w.
\end{aligned}
\end{equation}
The $H^2$ velocity and $H^1$ pressure allow integration by parts in
\eqref{eq_stokes_weak_pressure}. Arbitrary tangential wall traces of the
test field identify $\partial_z u|_{z=h}=0$ and
$\partial_z u|_{z=0}=\alpha u|_{z=0}$. The argument for the first estimate
uses the full periodic test space after \eqref{eq_stokes_weak_pressure}
and hence also proves this estimate for general periodic forces and
velocities without prescribed reflection parity.

For the first line of \eqref{eq_stokes_domains}, the weak solution for each
$\mathbf F\in\mathcal H$ yields an $H^2$ velocity with the stated traces.
Conversely, integration by parts for such an $H^2$ velocity gives
$a(\mathbf{u},\boldsymbol{\varphi})
=(-\Delta\mathbf{u},\boldsymbol{\varphi})_{L^2}$ for
$\boldsymbol{\varphi}\in\mathcal V$. This identifies $A$ and its domain.
Within the velocity reflection class, $\mathcal H$ consists of the $L^2$ fields
orthogonal to gradients of periodic $H^1$ functions. To see density in
this characterization, extend the horizontal component evenly and the
vertical component oddly across the walls to the torus of vertical period
$2h$. The extension is divergence-free in distributions by testing with
the sum of a test function and its vertical reflection. Periodic
mollification with even kernels preserves the horizontal parity and
the zero normal trace of the smooth restrictions. In particular, periodic
pressure gradients are removed by $\mathbb P$.

Suppose now $\mathbf F\in L^\ell$. The first estimate in
\eqref{eq_stokes_estimates} and the two-dimensional Sobolev embedding give
$\mathbf{u}\in W^{1,\ell}$. The vorticity
$\omega=\partial_xw-\partial_z u$ satisfies
\begin{equation}
\begin{aligned}
-\Delta\omega &=\partial_xF_2-\partial_z F_1 &&\text{in }\Omega,\\
\omega|_{z=h}&=0,& \omega|_{z=0}&=-\alpha u|_{z=0}.
\end{aligned}
\end{equation}
The scalar Dirichlet estimate with a divergence-form right-hand side gives
\begin{equation}
\begin{aligned}
\|\omega\|_{W^{1,\ell}}
 &\le C\bigl(\|\mathbf F\|_{L^\ell}
                 +\|u(\cdot,0)\|_{W^{1-1/\ell,\ell}(\mathbb T)}\bigr)
                 \le C\bigl(\|\mathbf F\|_{L^\ell}+\|\mathbf{u}\|_{H^2}\bigr)
 \le C\|\mathbf F\|_{L^\ell}.
\end{aligned}
\end{equation}
We obtain the estimate by lifting the boundary trace in $W^{1,\ell}$ and
reflecting the homogeneous Dirichlet problem oddly across the walls
before applying the periodic scalar Laplacian estimate. Define
$\psi(x,z)=-\int_0^z u(x,s)\,ds$. Incompressibility and the bottom trace
give $\partial_x\psi=w$, $\partial_z\psi=-u$. The top trace of $\psi$ is
constant in $x$ because $w|_{z=h}=0$, and the oddness of $\psi$ makes that
constant zero. Hence $\Delta\psi=\omega$ with zero Dirichlet traces.
The scalar $W^{2,\ell}$ estimate for $\psi$ and $\partial_x\psi$, followed
by $\partial_z^3\psi=\partial_z\omega-\partial_x^2\partial_z\psi$, gives
\begin{equation}
\begin{aligned}
\|\mathbf{u}\|_{W^{2,\ell}}
 &\le C\|\psi\|_{W^{3,\ell}}
 \le C\|\omega\|_{W^{1,\ell}}
 \le C\|\mathbf F\|_{L^\ell},\\
\|p\|_{W^{1,\ell}}
 &\le C\|\nabla p\|_{L^\ell}
 \le C\bigl(\|\mathbf F\|_{L^\ell}+\|\Delta\mathbf{u}\|_{L^\ell}\bigr).
\end{aligned}
\end{equation}

Finally, if $\mathbf F\in H^1$, differentiate \eqref{eq_stokes_weak_pressure}
in $x$. The first estimate applies to the resulting weak solution
$(\partial_x\mathbf{u},\partial_xp)$ with forcing $\partial_x\mathbf F$ and the
same boundary form.
Differentiate \eqref{eq_stokes_normal_derivatives} in $z$ to obtain
\begin{equation}
\begin{aligned}
\|\partial_z^3 w\|_{L^2}&\le \|\partial_x\partial_z^2 u\|_{L^2},\\
\|\partial_z^3 u\|_{L^2}
 &\le \|\partial_x^2\partial_z u\|_{L^2}
      +\|\partial_x\partial_zp\|_{L^2}+\|\partial_z F_1\|_{L^2},\\
\|\partial_z^2p\|_{L^2}
 &\le \|\partial_z F_2\|_{L^2}
      +\|\partial_x^2\partial_z w\|_{L^2}+\|\partial_z^3w\|_{L^2}.
\end{aligned}
\end{equation}
These inequalities, combined with the tangential estimate, prove the last
bound in \eqref{eq_stokes_estimates}.
\end{proof}

\Needspace{5\baselineskip}
\phantomsection
\section*{Acknowledgments}
\addcontentsline{toc}{section}{Acknowledgments}
Q. Wang acknowledges support from the Natural Science Foundation of
Sichuan Province (grant \mbox{2025ZNSFSC0072}) and the National Natural
Science Foundation of China (grant \mbox{12301131}).

\Needspace{5\baselineskip}
\phantomsection

\end{document}